\documentclass[11pt,oneside]{article}
\usepackage{amssymb,amsmath,latexsym,amsfonts,amscd,amsthm,multirow,ctable,colortbl,mathdots,caption,array}
\usepackage{upgreek}
\usepackage{dsfont}
\usepackage{mathtools}
\usepackage{wasysym}
\usepackage{fancyhdr,tabularx,cite,mathrsfs,graphicx}
\usepackage{authblk}
\usepackage[headings]{fullpage}
\usepackage{stmaryrd}
\usepackage[all]{xy}
\usepackage{stmaryrd,rotating}
\usepackage{pdflscape,longtable}
\usepackage[font=small]{caption}
\usepackage{setspace}

\usepackage{hyperref}

\def\a{\alpha} 
 
\def\c{\gamma}

\def\vf{\varphi}

\def\s{\sigma}

\def\D{\Delta}

\newcolumntype{R}{ >{$}r <{$}}
\newcolumntype{C}{ >{$}c <{$}}
\newcolumntype{L}{ >{$}l <{$}}
\newcolumntype{F}{>{\centering\arraybackslash}m{1.5cm}}

\def\LL{\Lambda}

\newcommand{\mc}[1]{\mathcal{#1}}

\newcommand{\comment}[1]{}

\newcommand{\commentapd}[1]{}

\newcommand{\RR}{{\mathbb R}}
\newcommand{\CC}{{\mathbb C}}
\newcommand{\ZZ}{{\mathbb Z}}
\newcommand{\QQ}{{\mathbb Q}}
\newcommand{\HH}{{\mathbb H}}
\newcommand{\FF}{{\mathbb F}}

\newcommand{\Aut}{\operatorname{Aut}}

\newcommand{\Disc}{\operatorname{Disc}}

\newcommand{\tr}{\operatorname{{tr}}}

\newcommand{\ex}{\operatorname{e}} 

\newcommand{\Ex}{\operatorname{Ex}}

\newcommand{\wh}{{\rm wh}}	
\newcommand{\wsh}{{\rm wsh}}

\newcommand{\E}{\mathcal{E}}
\newcommand{\W}{\mathcal{W}}
\newcommand{\xmod}{{\rm \;mod\;}}

\newcommand{\PSL}{\operatorname{\textsl{PSL}}}    
\newcommand{\SL}{\operatorname{\textsl{SL}}}      
\newcommand{\mpt}{\widetilde{\SL}_2}      
\newcommand{\PGL}{\operatorname{\textsl{PGL}}}    
\newcommand{\GL}{{\textsl{GL}}}      

\newcommand{\MM}{\mathbb{M}}	
\newcommand{\Co}{{\operatorname{\textsl{Co}}}}	
\newcommand{\Th}{{\operatorname{\textsl{Th}}}}    

\newcommand{\pd}{\uplambda}

\newtheorem{thm}{Theorem}[section]
\newtheorem*{thm*}{Theorem}

\newtheorem{con}[thm]{Conjecture}

\theoremstyle{definition}

\theoremstyle{remark}

\numberwithin{equation}{section}

\begin{document}

\setstretch{1.26}

\title{
\vspace{-35pt}
\textsc{\huge{ \hspace{2.7pt}
An Overview of Penumbral Moonshine
}}
\renewcommand{\thefootnote}{\fnsymbol{footnote}} 
}


\renewcommand{\thefootnote}{\arabic{footnote}} 

\author[1,2]{John F.\ R.\ Duncan\thanks{jduncan@gate.sinica.edu.tw}\thanks{john.duncan@emory.edu}}
\author[3]{Jeffrey A.\ Harvey\thanks{j-harvey@uchicago.edu}}
\author[4]{Brandon C.\ Rayhaun\thanks{brayhaun@stanford.edu}}

\affil[1]{Institute of Mathematics, Academia Sinica, Taipei, Taiwan.}
\affil[2]{Department of Mathematics, Emory University, Atlanta, GA 30322, U.S.A.}
\affil[3]{Enrico Fermi Institute and Department of Physics, University of Chicago, Chicago, IL 60637, U.S.A.}
\affil[4]{Institute for Theoretical Physics, Stanford University, Palo Alto, CA 94305, U.S.A.}

\date{} 

\maketitle

\begin{center}
\emph{Dedicated to the memory of John H.~Conway.}
\end{center}
\vspace{.24in}

\abstract{
As Mathieu moonshine is a special case of umbral moonshine, Thompson moonshine (in half-integral weight) is a special case of a family of similar relationships between 
finite groups and vector-valued modular forms of a certain kind. 
We call this penumbral moonshine.
We introduce and explain some features of this phenomenon in this work. 
}

\clearpage

\tableofcontents

\clearpage

\section{Introduction}\label{sec:int}

To 
introduce penumbral moonshine, we first recall 
primary features of two of its precedents, monstrous and umbral moonshine. We review the former 
in \S~\ref{sec:int-mns}, and the latter
in \S~\ref{sec:int-umb}. 
We use our review of umbral moonshine to motivate the main question of this work 
in \S~\ref{sec:int-mot},
and then, in \S~\ref{sec:int-thm}, explain how it is that Thompson moonshine, viewed from an appropriate place,
sets us on a path toward an answer.
To
push further along this path we
point out some distinguishing
features of the forms of umbral moonshine in \S~\ref{sec:int-opt}. This motivates the formulation of the Thompson-moonshine-compatible counterpart features that we present in \S~\ref{sec:int-pfm}.
By the end of \S~\ref{sec:int-pfm} we are at a point where penumbral moonshine may be viewed. We describe what we see,
and explain the content of this work, in \S~\ref{sec:int-pmo}.

\subsection{Monstrous Moonshine}\label{sec:int-mns}

In 1978 McKay famously observed a relationship between the Fourier expansion of the \emph{elliptic modular invariant}
\begin{align}\label{eqn:int-mns:ellmodinv}
J(\tau) =  q^{-1}+196884q+21493760q^2+864299970q^3+\cdots \ \   (q=e^{2\pi i \tau}, \ \tau\in\mathbb{H})
\end{align}
and the representation theory of the sporadic simple monster group, $\mathbb{M}$, which was at that point a recent discovery of Fischer and Griess (see \cite{MR0399248,MR671653}). 
Specifically, McKay noticed that the coefficient of $q$ in (\ref{eqn:int-mns:ellmodinv}) is just one more than (what was then conjectured to be) the dimension of the smallest non-trivial ordinary representation of $\MM$. 
Shortly after, Thompson, building upon similar observations \cite{Tho_NmrlgyMonsEllModFn}, conjectured \cite{Tho_FinGpsModFns} the existence of an infinite-dimensional graded $\mathbb{M}$-module 
\begin{align}\label{eqn:int:Vnatural}
V^\natural = \bigoplus_{n\geq -1}V^\natural_n
\end{align}
whose \emph{McKay--Thompson series} 
\begin{align}\label{eqn:int:Tg}
	T_g(\tau) := \sum_{n\geq -1} 
		\tr(g|V^\natural_n)q^n,
\end{align}
for $g\in \MM$,
are the normalized principal moduli associated to genus zero subgroups of $\SL_2(\RR)$. 
That is to say, the McKay--Thompson series $T_g$ should be very special functions. 
(See \S~\ref{sec:pre-mod} for the relevant definitions.) 

In this picture, the elliptic modular invariant (\ref{eqn:int-mns:ellmodinv}) arises when one chooses $g$ to be the identity element of $\mathbb{M}$, and in particular recovers the graded dimension of $V^{\natural}$. 
Away from the identity element, Thompson's conjecture was made precise by Conway and Norton, with their concrete specification \cite{MR554399} of an association of genus zero groups $\Gamma_g<\SL_2(\RR)$ to elements of the monster $g\in \MM$. 
The resulting {\em monstrous moonshine conjecture} was proven by Borcherds \cite{MR1172696}, following the explicit algebraic construction \cite{FLMPNAS,FLMBerk,FLM} of a candidate $\MM$-module (\ref{eqn:int:Vnatural}) by Frenkel, Lepowsky, and Meurman. 

Frenkel, Lepowsky and Meurman established a vertex operator algebra (VOA) structure on $V^\natural$ in \cite{FLM}, following the introduction of the notion of vertex algebra by Borcherds in \cite{Bor_PNAS}, and this structure plays a key role in Borcherds' proof.
Interestingly, this particular example of a VOA serves to define a chiral conformal field theory (CFT) in the sense of physics (see e.g.\ \cite{MR968697,MR1160677}). 
So we may consider applying physical arguments---even if they are not always rigorous---to monstrous moonshine. 

From the physical point of view, the fact that the McKay--Thompson series $T_g$ are modular functions follows once one has the VOA structure in place. 
(Cf.\ e.g.\ \cite{MR2793423,Gaberdiel2010}, and see \cite{Zhu_ModInv,Dong2000,MR2046807,MR3435813} for rigorous mathematical results along these lines.) 
However, the groups $\Gamma_g$, which govern the transformation properties of the $T_g$, are often larger than the 
{\em Hecke congruence groups},
\begin{gather}\label{eqn:int:Gamma0m}
\Gamma_0(m) := \left.\left\{\left(\begin{matrix}a & b \\ cm & d \end{matrix}\right)\,\right|\, a,b,c,d\in \ZZ, ad-bcm=1       \right\},
\end{gather}
that one would expect from established CFT or VOA arguments. 
In particular, the relevant congruence groups (\ref{eqn:int:Gamma0m})
are often not genus zero. 
This enhancement of the invariance groups of the McKay--Thompson series $T_g$ to conspiratorially yield genus zero subgroups $\Gamma_g$ of $\SL_2(\mathbb{R})$ is simultaneously the most mysterious aspect of monstrous moonshine and its defining feature. 

The conceptual origins of monstrous moonshine are not a main focus of this work, although we do not doubt that concrete connections will emerge in due course. 
In the meantime we refer to \cite{DunFre_RSMG} for a proposal that the genus zero property emerges from a $3$-dimensional gravity theory, and refer to \cite{frenkel2020sketch} for connections to universal Teichm\"uller space, and a celebrated infinite simple Thompson group.
An analysis of the genus zero property from a physical, string theoretic point of view, which reinterprets and refines important earlier work \cite{MR1165184,MR1312433} of Tuite, is given in
\cite{MR3573665,MR3708102}. 
For recent 
reviews of moonshine see \cite{Anagiannis:2018jqf,Duncan2020,MR3375653,kachru2016elementary}.

\subsection{Umbral Moonshine}\label{sec:int-umb}

In the time that has passed since McKay's original observation there have been similar discoveries relating privileged automorphic functions to exceptional finite groups, and they have led to new ties between mathematics and physical models arising from CFT and string theory. 
In particular, significant developments have followed from the {\em Mathieu moonshine} observation of Eguchi, Ooguri, and Tachikawa \cite{Eguchi2010}, in which the sporadic simple Mathieu group $M_{24}$ (see e.g.\ \cite{MR0338152,atlas}) governs the Fourier coefficients of a weight $\frac{1}{2}$ mock modular form
\begin{align}\label{eqn:int:H2}
H^{(2)}_1(\tau) = 
-2q^{-\frac{1}{8}}+2\cdot 45q^{\frac 78}+2\cdot 231q^{\frac {15}{8}}+2\cdot 770q^{\frac{23}{8}}+\cdots,
\end{align}
which may be extracted from the elliptic genus of a K3 surface. 
(See e.g.\ \cite{MR2985326} or \cite{MR3375653} for introductory accounts.) 

It was later realized that Mathieu moonshine belongs to a family of interrelated moonshine phenomena. Specifically, {\em umbral moonshine} \cite{UM,MUM,MR3433373,MR3766220} attaches to each Niemeier lattice $L^{(\ell)}$ \cite{Nie_DefQdtFrm24,MR558941} a vector-valued mock modular form 
$H^{(\ell)}=(H^{(\ell)}_r)$
of weight $\frac12$, 
and a finite group 
\begin{gather}\label{eqn:Gell}
G^{(\ell)} := {\Aut}(L^{(\ell)})/W^{(\ell)}
\end{gather} 
(here $W^{(\ell)}$ is the Weyl group of the root system of $L^{(\ell)}$), in such a way that $H^{(\ell)}$ recovers the graded dimension of a virtual graded $G^{(\ell)}$-module,
\begin{gather}\label{eqn:int:Kell}
	K^{(\ell)}=\bigoplus_{r\xmod 2m}\bigoplus_{D\equiv r^2\xmod 4m}K^{(\ell)}_{r,-\frac{D}{4m}}.
\end{gather}
Moreover, the associated vector-valued graded trace functions, $H^{(\ell)}_g=(H^{(\ell)}_{g,r})$, defined for $g\in G^{(\ell)}$ 
by
\begin{gather}\label{eqn:int:Hellgr}
	H^{(\ell)}_{g,r}(\tau):=
	\sum_{D\equiv r^2\xmod 4m}\tr(g|K^{(\ell)}_{r,-\frac{D}{4m}})q^{-\frac{D}{4m}},
\end{gather}
and also called McKay--Thompson series,
are optimal, in a sense which identifies them as natural analogues of normalized principal moduli. 
(See \cite{MR3021323,MR3766220,MR4139238}, and also 
\S~\ref{sec:int-opt} below, for more discussion of this.) 
In brief, the McKay--Thompson series (\ref{eqn:int:Hellgr}) 
of umbral moonshine are very special, just like the McKay--Thompson series (\ref{eqn:int:Tg}) of monstrous moonshine.

We refer to \cite{Dabholkar:2012nd,MR2555930,MR2605321} for introductory accounts of mock modular forms, and refer to \cite{MR3729259} for a thorough reference.
For an introduction to umbral moonshine see \cite{MR3375653}.

\subsection{Motivation}\label{sec:int-mot}

The indexing symbol in umbral moonshine (\ref{eqn:Gell}--\ref{eqn:int:Hellgr})
is called a lambency, following \cite{UM,MUM}. 
It
takes the form
\begin{gather}
\label{eqn:int:ell}
\ell=m+n,n',\dots,
\end{gather} 
and thereby specifies the $m$ in (\ref{eqn:int:Kell}--\ref{eqn:int:Hellgr}).
Mathieu moonshine (\ref{eqn:int:H2}) corresponds to the choice $\ell=2$, 
in which case the root system is 
$A_1^{24}$. 

The 
lambency notation (\ref{eqn:int:ell})
is inspired 
by the original work \cite{MR554399} of Conway and Norton on monstrous moonshine, in which such symbols 
are used to denote some of the aforementioned (in \S~\ref{sec:int-mns}) extensions of 
$\Gamma_0(m)$ that conspire to be genus zero. 
Their convention is that the 
$n,n',\dots$ in (\ref{eqn:int:ell}) represent the 
presence of 
{Atkin--Lehner involutions}
$W_n, W_{n'},\dots$ (see (\ref{eqn:pre-mod:Wn})) in the extension of $\Gamma_0(m)$ defined by (\ref{eqn:int:ell}) (cf.\ (\ref{eqn:pre-mod:Gamma0m+K})). 
In particular, 
\begin{enumerate}
\item
the $n,n',\dots$ in 
(\ref{eqn:int:ell}) should be exact divisors of $m$ (so that $\gcd(n,\frac{m}{n})=1$, and similarly for $n',\dots$), and 
\item
the set $\{1,n,n',\dots\}$ should be closed under the natural group operation (see (\ref{eqn:pre-mod:nstarnprime})) on exact divisors of $m$.
\end{enumerate}
In this work we refer to general symbols $\ell$ as in (\ref{eqn:int:ell}) that satisfy the above two conditions as {\em lambencies}, and use adjectives such as umbral, penumbral, \&c., to refer to those particular lambencies that appear in the corresponding 
settings.
We also say that $\ell$ has {\em genus zero}, or not, according as the corresponding subgroup (see (\ref{eqn:pre-mod:Gamma0m+K})) of $\SL_2(\RR)$ defines a genus zero quotient (see (\ref{eqn:pre-mod:XGamma})) of the upper half-plane, or not. 

The Atkin--Lehner involutions
$W_n$ (\ref{eqn:pre-mod:Wn}) 
for $\Gamma_0(m)$ act naturally, as involutions, on spaces of modular forms for $\Gamma_0(m)$. 
For this reason we may use the $W_n$ 
to refine spaces of modular forms, by splitting them into invariant and anti-invariant parts. A special role in the modular form theory---and in moonshine---is played by the 
{\em Fricke involution}, 
$W_m$, being the coset of $\Gamma_0(m)$ in $\SL_2(\RR)$ represented by 
\begin{gather}\label{eqn:int:Fricke}
	\frac1{\sqrt{m}}\left(\begin{matrix}0 & -1 \\ m & 0 \end{matrix}\right).
\end{gather}
On this account, a lambency $\ell=m+n,n',\dots$, 
and the discrete group it represents, are called {\em Fricke} or {\em non-Fricke}, according as $m$ is included in $\{1,n,n',\dots\}$ or not.
It develops that all of the lambencies of umbral moonshine (\ref{eqn:Gell}--\ref{eqn:int:Hellgr}) are non-Fricke and genus zero. 
(We will say more about this in \S~\ref{sec:int-opt}.)

With the connection to subgroups of $\SL_2(\RR)$ in mind, 
let us 
call $m$ the {\em level} of 
a lambency 
$\ell=m+n,n',\dots$. 
It develops that, for $\ell$ one of the lambencies of umbral moonshine, and $m$ the level of $\ell$, the (non-holomorphic) modular completion $\widehat{H}^{(\ell)}$ of $H^{(\ell)}$ 
transforms in the same way as a vector-valued modular form of weight $\frac12$ 
for the index $m$ Weil representation 
(see (\ref{eqn:pre-mod:slashknugammaupsilon}), (\ref{eqn:pre-mod:ind_m_Weil_rep})) 
of the modular group $\SL_2(\ZZ)=\Gamma_0(1)$ (\ref{eqn:int:Gamma0m}). 
This means that the components $\widehat{H}^{(\ell)}_r$ of $\widehat{H}^{(\ell)}$ transform in the same way as the theta-coefficients (cf.\ (\ref{eqn:pre-jac:thetadecomposition})) of a holomorphic Jacobi form of weight $1$ and index $m$.
That is, if we set
\begin{gather}\label{eqn:int:phiellequalsHellrthetamr}
	{\phi}^{(\ell)}(\tau,z):=\sum_{r\xmod 2m}{H}^{(\ell)}_r(\tau)\theta_{m,r}(\tau,z),
\end{gather}
where the $\theta_{m,r}$ are the theta series of the even positive-definite lattices of rank $1$ (as defined in (\ref{eqn:thetamr})), and if we define $\widehat{\phi}^{(\ell)}$ similarly, taking $\widehat{H}^{(\ell)}_r$ in place of $H^{(\ell)}_r$ in (\ref{eqn:int:phiellequalsHellrthetamr}), then we have the {\em modular invariance} identities
\begin{gather}\label{eqn:int:modularinvariance}
	\widehat{\phi}^{(\ell)}(\tau+1,z)
	=\widehat{\phi}^{(\ell)}(\tau,z)
	=\widehat{\phi}^{(\ell)}\left(-\frac{1}{\tau},\frac{z}{\tau}\right)\frac{1}{\tau}e^{-2\pi i m\frac{z^2}{\tau}}
	,
\end{gather}
and also $\phi(\tau+1,z)=\phi(\tau,z)$, for $\tau\in \HH$ and $z\in \CC$.
(See e.g.\ \cite{MR781735} for a detailed introduction to Jacobi forms, and see \cite{Dabholkar:2012nd} for a more succinct account, that also discusses mock Jacobi forms.)

Given the connection (\ref{eqn:int:phiellequalsHellrthetamr}--\ref{eqn:int:modularinvariance}), between Jacobi forms and umbral moonshine, it is natural to consider the broader theory of the former, in case it can 
shed light on the latter.
Some particularly fruitful aspects of this broader theory have developed from 
relationships to automorphic functions of other kinds; e.g.\ genus $2$ Siegel modular forms \cite{feingold_frenkel}, and elliptic modular forms of integral weight \cite{MR958592}. 
Significantly for us, such relationships have illuminated new, but similar structures. 
For instance, in the setting of \cite{MR958592}, Fricke anti-invariant elliptic modular forms (cf.\ (\ref{eqn:int:Fricke})) are absent.
But they appear naturally---see \cite{MR1074485,MR1116103}---with the introduction \cite{MR1096975} of {skew-holomorphic Jacobi forms},
which supply another non-holomorphic variant of the more typical holomorphic notion.

Briefly, a \emph{skew-holomorphic Jacobi form} of weight $k$ and index $m$ is a real-analytic function $\varphi:\mathbb{H}\times\mathbb{C}\to \mathbb{C}$, which, similar to (\ref{eqn:int:phiellequalsHellrthetamr}), admits a {theta-decomposition} 
\begin{align}\label{eqn:int:skw-hlm-tht-dcm}
\varphi(\tau,z) = \sum_{r\xmod{2m}} \overline{f_r(\tau)}\theta_{m,r}(\tau,z),
\end{align}
where the $\theta_{m,r}$ are as in (\ref{eqn:int:phiellequalsHellrthetamr}), but where the $f_r$ in (\ref{eqn:int:skw-hlm-tht-dcm}) comprise the components of a holomorphic vector-valued modular form $f=(f_r)$ of weight $k-\frac{1}{2}$ for the dual of the index $m$ Weil representation (see (\ref{eqn:pre-mod:ind_m_Weil_rep})). 
In the case that $k=1$, this entails the {\em skew-modular invariance} identities
\begin{gather}\label{eqn:int:skewmodularinvariance}
	{\vf}(\tau+1,z)
	={\vf}(\tau,z)
	={\vf}\left(-\frac{1}{\tau},\frac{z}{\tau}\right)\frac{1}{|\tau|}e^{-2\pi i m\frac{z^2}{\tau}}
	,
\end{gather}
for $\tau\in \HH$ and $z\in \CC$, which may be compared to (\ref{eqn:int:modularinvariance}).

Comparing (\ref{eqn:int:phiellequalsHellrthetamr}--\ref{eqn:int:modularinvariance}) 
with (\ref{eqn:int:skw-hlm-tht-dcm}--\ref{eqn:int:skewmodularinvariance}) 
we see that 
skew-holomorphic Jacobi forms are close cousins of the kinds of functions 
that arise in umbral moonshine. In light of this, a question naturally arises:
\begin{quote}
{\sl Is there a skew-holomorphic analog of umbral moonshine?}
\end{quote}

This is the main 
question of this work, 
and our primary purpose here 
is to provide a positive answer: the phenomenon we call penumbral moonshine. 
In the remainder of this introduction we will offer an approach to navigating, and further delineating, 
the penumbral moonshine landscape.

\subsection{Thompson Moonshine}\label{sec:int-thm}

For a first step along a path toward penumbral moonshine we note that a candidate for a ``first'' case  of a skew-holomorphic analog of umbral moonshine 
has actually already appeared. Indeed, it was conjectured in \cite{MR3521908}, 
and proven in \cite{MR3582425}, 
that a certain 
(scalar-valued) weakly holomorphic 
modular form 
of weight $\frac{1}{2}$ for $\Gamma_0(4)$,
satisfying\footnote{The symbol $\mathcal{F}_3$ was used to denote the function (\ref{eqn:int:F-31}) in \cite{MR3521908}.}
\begin{align}
\begin{split}\label{eqn:int:F-31}
\breve{F}^{(-3,1)}(\tau) &= 
\sum_{D\geq -3} \breve{C}^{(-3,1)}(D)q^D \\
&=  2q^{-3} +248 +2\cdot 27000 q^4 -2\cdot 85995 q^5+2\cdot 1707264 q^8 +\dots,
\end{split}
\end{align}
and closely related to a form appearing in work of Borcherds (see Example 2 in \S~14 of \cite{MR1323986}) 
and 
Zagier (see \S~5 of \cite{Zag_TrcSngMdl}), 
supports moonshine for the sporadic simple group $\Th$ of Thompson. That is, there exists a virtual graded $\Th$-module 
\begin{align}\label{eqn:int:W-31}
\breve W^{(-3,1)} = 
\bigoplus_{D\geq -3} \breve W^{(-3,1)}_D,
\end{align}
whose graded dimension
recovers $\breve{F}^{(-3,1)}$. 
Moreover, the associated graded trace functions, 
\begin{gather}\label{eqn:int:F-31gW-31}
\breve{F}^{(-3,1)}_g(\tau):=\sum_{D\geq -3}\tr(g|{\breve{W}^{(-3,1)}_D})q^D
\end{gather}
for $g\in \Th$, 
are also optimal (cf.\ (\ref{eqn:int:Hellgr})), 
in a sense which identifies them as 
analogs of normalized principal moduli. 
(See \S\S~\ref{sec:int-opt}--\ref{sec:int-pfm} 
below, and 
also \cite{MR4139238}, for more expository discussion of this.)

The connection to skew-holomorphic Jacobi forms (\ref{eqn:int:skw-hlm-tht-dcm}--\ref{eqn:int:skewmodularinvariance}) is
that the 
vector-valued function 
$F^{(-3,1)}=(F_0^{(-3,1)},F_1^{(-3,1)})$, 
defined by
\begin{align}\label{eqn:int:F-31Fourierseries}
F^{(-3,1)}_r(\tau) := 
\sum_{D\equiv r^2\xmod{4}}C^{(-3,1)}(D,r)q^{\frac{D}{4}},
\end{align}
where $C^{(-3,1)}(D,r)$ is taken to be $\breve{C}^{(-3,1)}(D)$ (\ref{eqn:int:F-31}) in case $D\equiv r^2\xmod 4$, and $0$ otherwise, transforms under the dual of the index $m=1$ 
Weil representation. 
That is, the 
function
\begin{gather}\label{eqn:int:varphi31-tht-dcm}
	\vf^{(-3,1)}(\tau,z):=\sum_{r\xmod 2} \overline{F_r^{(-3,1)}(\tau)}\theta_{1,r}(\tau,z)
\end{gather}
(cf.\ (\ref{eqn:int:skw-hlm-tht-dcm})) satisfies the same skew-modular invariance identities (\ref{eqn:int:skewmodularinvariance}) 
as a 
skew-holomorphic Jacobi form of weight $1$ and index $1$. 
Thus the Thompson moonshine (\ref{eqn:int:F-31}--\ref{eqn:int:F-31gW-31}) of \cite{MR3521908} furnishes a candidate first case (\ref{eqn:int:F-31Fourierseries}--\ref{eqn:int:varphi31-tht-dcm}) for a skew-holomorphic counterpart to the umbral moonshine story (\ref{eqn:Gell}--\ref{eqn:int:Hellgr}).

We seek
to situate this Thompson moonshine (\ref{eqn:int:F-31}--\ref{eqn:int:varphi31-tht-dcm}) 
inside its own skew-holomorphic family, 
so let us review 
the role that lambencies (\ref{eqn:int:ell}) play, in organizing the various instances (\ref{eqn:int:phiellequalsHellrthetamr}--\ref{eqn:int:modularinvariance}) of umbral moonshine. This role is structural at the level of Jacobi forms, for 
just as the $n,n',\dots$ in a lambency $\ell=m+n,n',\dots$ as in (\ref{eqn:int:ell}) determine an extension (\ref{eqn:pre-mod:Gamma0m+K}) of the congruence group $\Gamma_0(m)$, 
and refine spaces of modular forms for this group, 
they simultaneously 
refine spaces of Jacobi forms of index $m$, 
by imposing symmetries amongst the theta-coefficients that appear in a theta-decomposition such as
(\ref{eqn:int:phiellequalsHellrthetamr}). 
Specifically, there is a map $n\mapsto a(n)$ which naturally associates a square root of $1$ modulo $4m$ to each exact divisor $n$ of $m$ (see (\ref{eqn:pre-jac:an})),
and the 
$n,n',\dots$ present in a lambency $\ell=m+n,n',\dots$ 
of umbral moonshine represent the requirement that
\begin{gather}\label{eqn:int:HellranequalsHellr}
{H}^{(\ell)}_{ra(n)}={H}^{(\ell)}_{r}
\end{gather} 
for all $r$, and similarly with $n',\dots$ in place of $n$ (and likewise for the completions $\widehat{H}^{(\ell)}_r$). 

Recall from \S~\ref{sec:int-umb}
that the theta-decompositions of holomorphic (\ref{eqn:int:phiellequalsHellrthetamr}) and skew-holomorphic (\ref{eqn:int:skw-hlm-tht-dcm}) Jacobi forms of index $m$ are governed by dual representations of the modular group.
For this reason, 
a lambency $\ell=m+n,n',\dots$ 
may also be used to refine skew-holomorphic Jacobi forms, via precisely the same requirement. That is, we may restrict to skew-holomorphic Jacobi forms $\vf$ of index $m$ satisfying 
\begin{gather}\label{eqn:int:franequalsfr}
f_{ra(n)}=f_r
\end{gather}
for all $r$, and similarly with $n',\dots$ in place of $n$, 
where the $f_r$ are as in (\ref{eqn:int:skw-hlm-tht-dcm}). 

An interesting point of contrast between the holomorphic and skew-holomorphic cases now appears. To see this observe that we obtain ${\phi}^{(\ell)}(\tau,-z)=-{\phi}^{(\ell)}(\tau,z)$ by making two applications of the second modular invariance identity in (\ref{eqn:int:modularinvariance}) (and noting that the completion map $\phi\mapsto \widehat{\phi}$ is linear). In terms of the $H^{(\ell)}_r$ (\ref{eqn:int:phiellequalsHellrthetamr}) this manifests in the requirement that 
\begin{gather}\label{eqn:int:HellminusrequalsminusHellr}
{H}^{(\ell)}_{-r}=-{H}^{(\ell)}_{r}
\end{gather} 
for all $r$, since $\theta_{m,r}(\tau,-z)=\theta_{m,-r}(\tau,z)$ (cf.\ (\ref{eqn:thetamr})). 
On the other hand, two applications of the second identity in (\ref{eqn:int:skewmodularinvariance}) yields $\vf(\tau,-z)=\vf(\tau,z)$, 
so we must have
\begin{gather}\label{eqn:int:fminusrequalsfr}
	f_{-r}=f_{r}
\end{gather} 
for theta-coefficients (\ref{eqn:int:skw-hlm-tht-dcm}) 
in the skew-holomorphic case. (Cf.\ (\ref{eqn:pre-jac:fminusr}).)
We also have $a(m)=-1$ (see (\ref{eqn:pre-jac:an})).
So comparing (\ref{eqn:int:HellranequalsHellr}) and (\ref{eqn:int:HellminusrequalsminusHellr}) we see one reason why 
the lambencies of umbral moonshine are non-Fricke (cf.\ (\ref{eqn:int:Fricke})), and comparing (\ref{eqn:int:franequalsfr}) and (\ref{eqn:int:fminusrequalsfr}) we see that any 
lambencies of relevance to penumbral moonshine should be Fricke. 

Another consequence of (\ref{eqn:int:HellminusrequalsminusHellr}) is that the components $H^{(2)}_r$ of $H^{(2)}$ are completely determined by $H^{(2)}_1$ (\ref{eqn:int:H2}). Indeed, the $H^{(2)}_r$ are indexed by $r\xmod 4$ according to (\ref{eqn:int:phiellequalsHellrthetamr}), and (\ref{eqn:int:HellminusrequalsminusHellr}) demands that $H^{(2)}_3=-H^{(2)}_{1}$, and $H^{(2)}_r=0$ for $r$ even. 
By a similar token, there is no $\ell=1$ case of umbral moonshine, because $r\equiv-r\xmod 2m$ for all $r$, when $m=1$, and this forces $H^{(\ell)}_r=-H^{(\ell)}_r$ for all $r$ in (\ref{eqn:int:HellminusrequalsminusHellr}).

In contrast, the skew-modular invariance requirement (\ref{eqn:int:skewmodularinvariance}) does not 
rule out examples with $m=1$, as 
the identity (\ref{eqn:int:fminusrequalsfr}) is tautological in this case. Sure enough, Thompson moonshine in the skew-holomorphic formulation (\ref{eqn:int:F-31Fourierseries}--\ref{eqn:int:varphi31-tht-dcm}) is consistent with the ``first'' Fricke choice, $\ell=1$.

\subsection{Umbral Forms}\label{sec:int-opt}

For another step in the direction of a distinguished skew-holomorphic family, we recall in more detail what distinguishes the forms $H^{(\ell)}$ of umbral moonshine. 
For this, note that by taking (\ref{eqn:int:phiellequalsHellrthetamr}) 
together with 
translation invariance $\phi^{(\ell)}(\tau+1,z)=\phi^{(\ell)}(\tau,z)$ (cf.\ (\ref{eqn:int:modularinvariance})), or simply by (\ref{eqn:int:Hellgr}), 
we obtain {Fourier series} expansions
\begin{gather}\label{eqn:int:phiellFourierseries}
	H^{(\ell)}_r(\tau)=\sum_{{D\equiv r^2\xmod 4m}}C^{(\ell)}(D,r)q^{-\frac{D}{4m}},
\end{gather}
for the component functions $H^{(\ell)}_r$ of the $H^{(\ell)}$. 
From (\ref{eqn:int:phiellFourierseries})
we see that
$H^{(\ell)}_r$ can only fail to be bounded as $\Im(\tau)\to \infty$, 
if a {Fourier coefficient} $C^{(\ell)}(D,r)$ is 
non-zero for some $D>0$. 

If it were the case that $C^{(\ell)}(D,r)$ were vanishing, for all $r$ and all $D>0$, then the $H^{(\ell)}_r$ would be unary theta series of weight $\frac12$, according to a result of Serre and Stark \cite{MR0472707} (and we would have $\widehat{H}^{(\ell)}_r=H^{(\ell)}_r$ and $\widehat{\phi}^{(\ell)}=\phi^{(\ell)}$). 
That is, the $H^{(\ell)}_r$ would be linear combinations of the Thetanullwerte $\theta^0_{m',r'}$ (see (\ref{eqn:thetamr})). However, Skoruppa showed \cite{Sko_Thesis} that there are no unary theta series compatible with the conditions (\ref{eqn:int:phiellequalsHellrthetamr}--\ref{eqn:int:modularinvariance}) satisfied by holomorphic Jacobi forms of weight $1$.
So the $H^{(\ell)}$ and $\phi^{(\ell)}$ (\ref{eqn:int:phiellequalsHellrthetamr}) of umbral moonshine are uniquely determined by the Fourier coefficients $C^{(\ell)}(D,r)$ (\ref{eqn:int:phiellFourierseries}) with $D> 0$.

We can now explain the sense in which the 
$H^{(\ell)}$ of umbral moonshine 
naturally serve as analogs of normalized principal moduli. Indeed,
{\em optimality} for the mock modular forms $H^{(\ell)}=(H^{(\ell)}_r)$ of umbral moonshine is the condition that 
\begin{gather}\label{eqn:int:umbral_optimality}
C^{(\ell)}(D,r)=0\text{ for $D>0$ unless $D=1$.} 
\end{gather}
And
the lambency $\ell$ (\ref{eqn:int:ell}) 
actually determines the $C^{(\ell)}(D,r)$ with $D=1$, 
once we've fixed a value for $C^{(\ell)}(1,1)$,
according to the following rule.
\begin{gather}\label{eqn:int:umbral_lambency_condition}
C^{(\ell)}(1,r)=
\begin{cases}
	C^{(\ell)}(1,1)&\text{ if $r\equiv a\xmod 2m$ for some $a\in \{1,a(n),a(n'),\dots\}$,}\\
	-C^{(\ell)}(1,1)&\text{ if $r\equiv -a\xmod 2m$ for some $a\in \{1,a(n),a(n'),\dots\}$,}\\
	0&\text{ else.}
\end{cases}
\end{gather}

Note that the consistency of this rule (\ref{eqn:int:umbral_lambency_condition}) depends upon the fact that $\{1,n,n',\dots\}$ is a group (see \S~\ref{sec:int-mot}) under the natural operation (\ref{eqn:pre-mod:nstarnprime}). 
Note also that (\ref{eqn:int:umbral_lambency_condition}) is consistent with (\ref{eqn:int:HellranequalsHellr}) and (\ref{eqn:int:HellminusrequalsminusHellr}). For this reason (\ref{eqn:int:umbral_lambency_condition}) would force $C^{(\ell)}(1,r)=0$ for all $r$, if $\ell$ were not non-Fricke.

In umbral moonshine we have the convention that $C^{(\ell)}(1,1)=-2$ for all $\ell$ (cf.\ (\ref{eqn:int:H2})), so $H^{(\ell)}$ is uniquely determined by the lambency $\ell$ via (\ref{eqn:int:umbral_lambency_condition}), and the umbral optimality condition (\ref{eqn:int:umbral_optimality}), 
just as the McKay--Thompson series $T_g$ (\ref{eqn:int:Tg}) of monstrous moonshine is uniquely determined by its invariance group $\Gamma_g$, and the condition that it be a normalized principal modulus for that group. 

We have mentioned in \S~\ref{sec:int-mot} that each lambency $\ell=m+n,n',\dots$ 
specifies a 
subgroup 
of $\SL_2(\RR)$ 
that extends 
$\Gamma_0(m)$ (see (\ref{eqn:pre-mod:Gamma0m+K})), 
and we have explained in \S~\ref{sec:int-thm} how such a symbol simultaneously refines, via symmetries (\ref{eqn:int:HellranequalsHellr}),
vector-valued mock modular forms 
that are compatible with (\ref{eqn:int:phiellequalsHellrthetamr}--\ref{eqn:int:modularinvariance}), so long as it is non-Fricke. 
It develops that, just as normalized principal moduli exist only for groups that are genus zero, a vector-valued mock modular form 
that is optimal in the umbral sense (\ref{eqn:int:umbral_optimality}), and satisfies 
(\ref{eqn:int:umbral_lambency_condition}) for some non-Fricke lambency 
$\ell$,
can be involved in moonshine if and only if 
$\ell$ 
has genus zero. 
Indeed, 
examples of such forms, satisfying (\ref{eqn:int:umbral_optimality}--\ref{eqn:int:umbral_lambency_condition}) and having integer Fourier coefficients (\ref{eqn:int:phiellFourierseries}), are exhibited in \cite{MR4127159} for each non-Fricke 
lambency (\ref{eqn:int:ell}) 
of genus zero. 
On the other hand,
it is proven in \cite{MR4127159} that such a vector-valued mock modular form necessarily has transcendental Fourier coefficients (\ref{eqn:int:phiellFourierseries}), if the subgroup (\ref{eqn:pre-mod:Gamma0m+K}) of $\SL_2(\RR)$ represented by the lambency 
in question has positive genus. 
Forms with transcendental coefficients are problematic for moonshine, because transcendental numbers do not arise as dimensions of representations (cf.\ (\ref{eqn:int:Hellgr})).

\subsection{Penumbral Forms}\label{sec:int-pfm}

We now pursue parallels,
in the skew-holomorphic setting,
to the distinguishing provisions (\ref{eqn:int:umbral_optimality}--\ref{eqn:int:umbral_lambency_condition}) that 
we have just reviewed. 
For this, (\ref{eqn:int:skw-hlm-tht-dcm}) and (\ref{eqn:int:skewmodularinvariance}) take the place of (\ref{eqn:int:phiellequalsHellrthetamr}) and (\ref{eqn:int:modularinvariance}), 
and the counterpart to (\ref{eqn:int:phiellFourierseries}) is
\begin{gather}\label{eqn:int:varphiFourierseries}
	f_r(\tau)=\sum_{{D\equiv r^2\xmod 4m}}C_\vf(D,r)q^{\frac{D}{4m}}.
\end{gather}
(Sure enough, (\ref{eqn:int:F-31Fourierseries}) is consistent with the $m=1$ case of this.)
So in particular, 
the 
$f_r$ in (\ref{eqn:int:skw-hlm-tht-dcm}) 
can only fail to be bounded as $\Im(\tau)\to \infty$ if $C_\vf(D,r)\neq 0$ for some $D<0$. 

Thus, for a skew-holomorphic analog of umbral optimality (\ref{eqn:int:umbral_optimality}), we should choose a $D_0<0$, and 
require 
that 
\begin{gather}\label{eqn:int:penumbral_optimality}
C_\vf(D,r)=0\text{ for $D<0$ unless $D=D_0$.}
\end{gather}
Then, for skew-holomorphic counterparts to the functions of umbral moonshine, we should consider lambencies
$\ell$
that are Fricke (\ref{eqn:int:Fricke}) (on account of (\ref{eqn:int:franequalsfr}) and (\ref{eqn:int:fminusrequalsfr})), and 
examine the weakly holomorphic vector-valued modular forms $f=(f_r)$, 
compatible with (\ref{eqn:int:skw-hlm-tht-dcm}--\ref{eqn:int:skewmodularinvariance}), that satisfy {\em $D_0$-optimality} (\ref{eqn:int:penumbral_optimality}), and 
a suitable analog of the lambency condition (\ref{eqn:int:umbral_lambency_condition}) for $\ell$. 
Since $a$ is congruent to one of $a(1),a(n),a(n'),\dots$ modulo $2m$ if and only if $-a$ is, when $\ell=m+n,n',\dots$ is Fricke, and since $C_\vf(D,r)=C_\vf(D,-r)$ for all $r$ according to (\ref{eqn:int:fminusrequalsfr}), the natural skew-holomorphic counterpart to (\ref{eqn:int:umbral_lambency_condition}) is 
\begin{gather}\label{eqn:int:penumbral_lambency_condition}
C_\vf(D_0,r)=
\begin{cases}
	C_\vf(D_0,r_0)&\text{ if $r\equiv r_0a\xmod 2m$ for some $a\in \{1,a(n),a(n'),\dots\}$,}\\
	0&\text{ else,}
\end{cases}
\end{gather}
where $r_0$ is a fixed square root of $D_0$ modulo $4m$.

The reader may have anticipated by this point that the defining function 
of Thompson moonshine (\ref{eqn:int:F-31}--\ref{eqn:int:F-31gW-31}) in the skew-holomorphic formulation,
$F^{(-3,1)}=(F^{(-3,1)}_r)$ (\ref{eqn:int:F-31Fourierseries}), 
satisfies $D_0$-optimality (\ref{eqn:int:penumbral_optimality}) and the lambency condition (\ref{eqn:int:penumbral_lambency_condition}),
for 
$D_0$ and $\ell$ as indicated by the superscript, $(D_0,\ell)=(-3,1)$. 
Penumbral moonshine develops from similarly distinguished 
vector-valued modular forms $F^{(\pd)}=(F^{(\pd)}_r)$, 
determined, via (\ref{eqn:int:penumbral_optimality}--\ref{eqn:int:penumbral_lambency_condition}), by more general values of $D_0$ and $\ell$, 
as encoded by the {\em lambdency} 
symbols\footnote{As a mnemonic, a lambdency is ``a lambency with a $D$.'' Pronunciation: ``{\sl lam}-den-see.''}
\begin{gather}\label{eqn:int:lambdency}
\pd=(D_0,\ell).
\end{gather} 
But in advance of more discussion of moonshine, 
we should detail how it is that these counterparts $F^{(\pd)}$ to $F^{(-3,1)}$ are determined by their lambdencies (\ref{eqn:int:lambdency}), and upon what grounds 
they are 
distinguished.
Indeed, 
we have explained in \S~\ref{sec:int-opt} how it is that the conditions (\ref{eqn:int:umbral_optimality}--\ref{eqn:int:umbral_lambency_condition}), 
relevant to umbral moonshine, 
provide an 
analog of the notion of normalized principal modulus, the latter being crucial in 
monstrous moonshine (see \S~\ref{sec:int-mns}).
To evaluate the effectiveness of (\ref{eqn:int:penumbral_optimality}--\ref{eqn:int:penumbral_lambency_condition})
we should address
the extent to which that analogy extends. 

In a word, it extends, but there are three 
subtleties arising in the skew-holomorphic setting that should be addressed:
\begin{enumerate}
\item\label{itm:int-pfm:subtlety1}
There are unary theta series 
that are compatible with 
the conditions (\ref{eqn:int:skw-hlm-tht-dcm}--\ref{eqn:int:skewmodularinvariance}). For example, we may take 
\begin{gather}\label{eqn:int-pfm:frthetamr0}
f_r(\tau)=\theta_{m,r}^0(\tau):
=\sum_n q^{\frac{(2mn+r)^2}{4m}}
\end{gather}
(cf.\ (\ref{eqn:thetamr}))
in (\ref{eqn:int:skw-hlm-tht-dcm}),
and this is compatible with (\ref{eqn:int:skewmodularinvariance}) for any $m$. 
Other solutions to 
(\ref{eqn:int:skw-hlm-tht-dcm}--\ref{eqn:int:skewmodularinvariance}), with $C_\vf(D,r)=0$ in (\ref{eqn:int:varphiFourierseries}) for all $D<0$ and all $r$, 
are 
also unary theta series, 
by the same result \cite{MR0472707} we mentioned in \S~\ref{sec:int-opt}. 
Thus we may say that a weakly holomorphic modular form $f=(f_r)$ as in (\ref{eqn:int:skw-hlm-tht-dcm}--\ref{eqn:int:skewmodularinvariance}) is determined by its Fourier coefficients $C_\vf(D,r)$ (\ref{eqn:int:varphiFourierseries}) with $D<0$, but only ``up to 
theta series.'' 
Also,
the symmetry condition (\ref{eqn:int:franequalsfr}) will not automatically be satisfied by a theta series solution, unless we impose it explicitly.
\item\label{itm:int-pfm:subtlety2}
There may be no non-theta series $D_0$-optimal (\ref{eqn:int:penumbral_optimality}) solutions to the 
lambency condition (\ref{eqn:int:penumbral_lambency_condition}), depending on how the lambency $\ell=m+n,n',\dots$ 
is chosen. 
For example, if $D_0$ is not a square modulo $4m$ then there is no $r$ for which $C_\vf(D_0,r)$ appears in (\ref{eqn:int:phiellFourierseries}), and in particular no choice for $r_0$ in (\ref{eqn:int:penumbral_lambency_condition}). 
And when 
$D_0$ is a square modulo $4m$, it may be that 
any solution to (\ref{eqn:int:penumbral_lambency_condition}) 
automatically solves it for a ``larger'' lambency. 
For example, if $D_0=-4$ and $\ell=10+10$, 
then we must take $r_0=6$ or $r_0=14$. Either way we have $r_0a(2)\equiv r_0 \xmod 20$ (cf.\ \S~\ref{sec:not}), so $\ell=10+10$ may be replaced with $\ell=10+2,5,10$ in (\ref{eqn:int:penumbral_lambency_condition}).
\item\label{itm:int-pfm:subtlety3}
It is possible that there is a serious ambiguity in the choice of $r_0$ in (\ref{eqn:int:penumbral_lambency_condition}). Specifically, there may be solutions $r$ to $D_0\equiv r^2\xmod 4m$ that do not arise as $r_0a(n)\xmod 2m$, for any exact divisor $n$ of $m$. 
In this case different choices of $r_0$ can 
lead to substantially different forms  
(by singling out different subrepresentations of the 
relevant
Weil representation, cf.\ (\ref{eqn:pre-jac:Jwshkm_decomp})). 
\end{enumerate}

Now, the absence of theta series in the umbral setting 
(\ref{eqn:int:phiellequalsHellrthetamr}--\ref{eqn:int:modularinvariance}) 
does not represent as significant a 
discrepancy as it might seem, 
because unary theta series solutions do arise when the modular invariance condition (\ref{eqn:int:modularinvariance}), representing invariance under the action of the modular group $\SL_2(\ZZ)=\Gamma_0(1)$, 
is relaxed to invariance under $\Gamma_0(n)$ 
for suitable $n$. 
(See \cite{MR3766220} for more discussion of this.)
So the absence of theta series solutions to 
(\ref{eqn:int:phiellequalsHellrthetamr}--\ref{eqn:int:modularinvariance}) 
is the exception, rather than the rule.
Moreover, the coefficients of unary theta series, like 
(\ref{eqn:int-pfm:frthetamr0}), do not grow, and are sparse, being supported at perfect square powers (of some fixed power of $q$, cf.\ (\ref{eqn:int-pfm:frthetamr0}), (\ref{eqn:thetamr})). On this basis, and with the comparison to principal moduli in mind, it is natural to regard the theta series solutions to (\ref{eqn:int:penumbral_optimality}--\ref{eqn:int:penumbral_lambency_condition}) as counterparts to constant functions.

Regarding the second subtlety, 
we can do no less than respect the restriction on lambencies (\ref{eqn:int:ell}) that it represents. We also desire that lambencies accurately represent the symmetries (\ref{eqn:int:franequalsfr}) of the forms they index (cf.\ (\ref{eqn:int:lambdency})), so we formalize the situation by saying that $D_0<0$ is {\em $\ell$-admissible}, for $\ell=m+n,n',\dots$ as in (\ref{eqn:int:ell}) and Fricke (\ref{eqn:int:Fricke}), if $D_0$ is a square modulo $4m$, and if solutions to the lambency condition (\ref{eqn:int:penumbral_lambency_condition}) do not automatically satisfy it for a larger group of exact divisors of $m$ than the one $\{1,n,n',\dots\}$ that is encoded in $\ell$.

It develops that the third subtlety does not arise under suitable restrictions on $D_0$ and $m$. For example, we can rule it out by 
requiring that $D_0$ is a {\em fundamental discriminant}, in the sense that it coincides with the discriminant of the number field that its square root defines, 
\begin{gather}\label{eqn:int-pfm:D0fund}
	D_0= 
	\Disc(\QQ(\sqrt{D_0})).
\end{gather}
We make this restriction (\ref{eqn:int-pfm:D0fund}) in this work, and 
moreover adopt the convention that $r_0$ 
be the smallest non-negative solution to 
\begin{gather}\label{eqn:int-pfm:D0equivr02mod4m}
D_0\equiv r_0^2\xmod 4m
\end{gather}
in (\ref{eqn:int:penumbral_lambency_condition}).
Thus we avoid 
ambiguity in 
(\ref{eqn:int:penumbral_lambency_condition}) in what follows. 

With these considerations in place, we are ready to present a comparison to 
principal moduli. For this, choose a fundamental $D_0<0$ (\ref{eqn:int-pfm:D0fund}), 
and a (necessarily Fricke) lambency $\ell=m+n,n',\dots$ 
such that $D_0$ is $\ell$-admissible.
Then a weakly holomorphic modular form $f=(f_r)$ as in (\ref{eqn:int:skw-hlm-tht-dcm}--\ref{eqn:int:skewmodularinvariance}) is uniquely determined, up to scale, and up to theta series, by $D_0$-optimality (\ref{eqn:int:penumbral_optimality}) and the lambency condition (\ref{eqn:int:penumbral_lambency_condition}). 
Thus, if 
$F^{(\pd)}=(F^{(\pd)}_r)$ is 
a non-theta series
$D_0$-optimal (\ref{eqn:int:penumbral_optimality}) solution 
to (\ref{eqn:int:penumbral_lambency_condition}) for $\ell$, 
where 
$\pd=(D_0,\ell)$, 
then $F^{(\pd)}$
is analogous to a principal modulus. Furthermore, writing
\begin{gather}\label{eqn:int-pfm:Fpdr_Fourier}
	F^{(\pd)}_r(\tau) = \sum_{D\equiv r^2\xmod 4m} C^{(\pd)}(D,r)q^{\frac{D}{4m}}
\end{gather}
(cf.\ (\ref{eqn:int:varphiFourierseries})),
we obtain analogs of normalized principal moduli 
by fixing $C^{(\pd)}(D_0,r_0)$ (where $r_0$ is as in (\ref{eqn:int-pfm:D0equivr02mod4m})), 
and fixing enough $C^{(\pd)}(D,r)$ with $D=d^2$ a perfect square (cf.\ (\ref{eqn:int-pfm:frthetamr0})), to pin down the contributions from theta series. 
Also, we reduce the number of $d$ to be considered by 
requiring the symmetry 
(\ref{eqn:int:franequalsfr}) represented by $\ell$.

We have mentioned in \S~\ref{sec:int-opt} that the vector-valued mock modular form determined by umbral optimality (\ref{eqn:int:umbral_optimality}) and the lambency condition (\ref{eqn:int:umbral_lambency_condition}), 
for $\ell$ 
non-Fricke,
can only 
participate in moonshine if 
$\ell$ 
has genus zero, on account of the fact that it necessarily has transcendental Fourier coefficients, if $\ell$ has positive genus.
In the present setting the connection to genus zero groups may be considered to be even 
closer to the 
monstrous one, in that it 
manifests in the very existence of $F^{(\pd)}$ as above. To wit, we prove in our companion paper \cite{pmz} that $F^{(\pd)}=(F^{(\pd)}_r)$, as in (\ref{eqn:int-pfm:Fpdr_Fourier}), for $\pd=(D_0,\ell)$, is guaranteed to exist if 
$\ell$ 
has genus zero, 
but fails to exist for generic choices of $D_0$ and $\ell$.
 With regard to rationality, we mention that a function $F^{(\pd)}=(F^{(\pd)}_r)$ as in 
(\ref{eqn:int-pfm:Fpdr_Fourier}) has integral Fourier coefficients, for suitable choices of $C^{(\pd)}(D_0,r_0)$ 
and $C^{(\pd)}(d^2,r)$, 
according to the main result of 
\cite{MR1981614}. 

We 
emphasize here that while the functions (\ref{eqn:int:Hellgr}) appearing in umbral moonshine are generally mock modular, the functions of 
penumbral moonshine 
are genuinely 
modular. This is part of the reason we refer to penumbral moonshine as such: It is similar to umbral moonshine in various ways, but less shadowy. However, as we will 
see presently,
many features of penumbral moonshine remain to be illuminated.

\subsection{Penumbral Moonshine}\label{sec:int-pmo}

We have arrived at the following place. For each lambdency $\pd=(D_0,\ell)$ (\ref{eqn:int:lambdency}), where $\ell$ is a lambency (\ref{eqn:int:ell})
that is Fricke (\ref{eqn:int:Fricke}) and genus zero, and 
$D_0$ is an
$\ell$-admissible (see \S~\ref{sec:int-pfm})
negative fundamental discriminant (\ref{eqn:int-pfm:D0fund}), we have a distinguished weakly holomorphic vector-valued modular form $F^{(\pd)}=(F^{(\pd)}_r)$ as in (\ref{eqn:int-pfm:Fpdr_Fourier}), 
uniquely determined by (\ref{eqn:int:penumbral_optimality}--\ref{eqn:int:penumbral_lambency_condition}), up to the choice of $C^{(\pd)}(D_0,r_0)$ and $C^{(\pd)}(d^2,r)$ (for finitely many $d$ and $r$). Furthermore, for suitable choices 
of $C^{(\pd)}(D_0,r_0)$ and the $C^{(\pd)}(d^2,r)$,
the Fourier coefficients 
of the component functions $F^{(\pd)}_r$ are integers.

The mysterious promise of moonshine is that, for each such lambdency $\pd=(D_0,\ell)$ as above 
(and for appropriate choices of $C^{(\pd)}(D_0,r_0)$ and the $C^{(\pd)}(d^2,r)$), 
there should be an associated finite group $G^{(\pd)}$ (cf.\ (\ref{eqn:Gell})), 
and a virtual graded $G^{(\pd)}$-module
\begin{gather}\label{eqn:int-pmo:Wpd}
	W^{(\pd)}=\bigoplus_{r\xmod 2m}\bigoplus_{D\equiv r^2\xmod 4m}W^{(\pd)}_{r,\frac{D}{4m}}
\end{gather}
(cf.\ (\ref{eqn:int:Kell})), such that $F^{(\pd)}$ is the graded dimension function of $W^{(\pd)}$.
Moreover, the associated vector-valued graded trace functions (a.k.a.\ McKay--Thompson series) $F^{(\pd)}_g=(F^{(\pd)}_{g,r})$, defined for $g\in G^{(\pd)}$ by
\begin{gather}\label{eqn:int-pmo:Fpdgr}
	F^{(\pd)}_{g,r}(\tau):=
	\sum_{D\equiv r^2\xmod 4m}\tr(g|W^{(\pd)}_{r,\frac{D}{4m}})q^{\frac{D}{4m}}
\end{gather}
(cf.\ (\ref{eqn:int:Hellgr})), 
should 
be distinguished directly similarly to $F^{(\pd)}$, 
via suitable (higher-level) counterparts to (\ref{eqn:int:penumbral_optimality}--\ref{eqn:int:penumbral_lambency_condition}).

We present 
proof of this promise in this work, for $D_0=-3$ and $D_0=-4$. Indeed, $\pd=(-3,1)$ corresponds to the original motivating example of Thompson moonshine (\ref{eqn:int:F-31}--\ref{eqn:int:varphi31-tht-dcm}), and the more general cases $\pd=(-3,\ell)$ define our sought-after skew-holomorphic family; the goal we set for ourselves in \S~\ref{sec:int-thm}. 
But although $D_0=-3$ is the smallest (in magnitude) value we can take for $D_0$ in (\ref{eqn:int:penumbral_optimality}), 
nothing we have seen along the path we have taken rules out the possibility of other worthy values of $D_0$. 
We manifest this possibility in this work, by providing 
another skew-holomorphic family---and therefore another skew-holomorphic analog of umbral moonshine---that is defined by lambdencies of the form $\pd=(-4,\ell)$. 

In this paper, 
moonshine manifests as follows.
\begin{enumerate}
\item\label{itm:int-pmo:evidence1}
In \S~\ref{sec:mns-gps} we identify finite groups $G^{(\pd)}$, for $\pd=(D_0,\ell)$ as above, with $D_0=-3$ and $D_0=-4$. 
(In particular, each $\ell$ arising is Fricke and genus zero, and $\pd=(D_0,\ell)$ appears only when $D_0$ is $\ell$-admissible.) 
\item\label{itm:int-pmo:evidence2}
In \S~\ref{sec:mns-fms} we 
specify the 
vector-valued modular forms $F^{(\pd)}_g=(F^{(\pd)}_{g,r})$, for $g\in G^{(\pd)}$, that 
we expect to arise as McKay--Thompson series (\ref{eqn:int-pmo:Fpdgr}). 
By construction they are distinguished, in 
that they satisfy higher-level counterparts to (\ref{eqn:int:penumbral_optimality}--\ref{eqn:int:penumbral_lambency_condition}).
\item\label{itm:int-pmo:evidence3}
In \S~\ref{sec:mns-mds} we prove the existence of virtual graded $G^{(\pd)}$-modules $W^{(\pd)}$ (\ref{eqn:int-pmo:Wpd}), that realize the distinguished forms $F^{(\pd)}_g=(F^{(\pd)}_{g,r})$ as McKay--Thompson series (\ref{eqn:int-pmo:Fpdgr}).
\item\label{itm:int-pmo:evidence4}
Also in \S~\ref{sec:mns-mds}, we further demonstrate the distinguished nature of the forms $F^{(\pd)}_g$, at least for some lambdencies $\pd$, by 
identifying discriminant properties of the corresponding $G^{(\pd)}$-modules $W^{(\pd)}$ (\ref{eqn:int-pmo:Wpd}), 
similar 
to those observed in umbral moonshine (see \cite{UM,MUM} and \cite{MR3231314}) and Thompson moonshine (see \cite{MR3521908}). 
These properties serve to predict the minimal extension of $\QQ$ over which the $G^{(\pd)}$-module structure on $W^{(\pd)}_{r,\frac{D}{4m}}$ can be defined, in terms of the value $D$. 
\item\label{itm:int-pmo:evidence5}
For some of the lambdencies $\pd$ arising we give descriptions of the corresponding groups $G^{(\pd)}$ in terms of 
modular lattices in \S~\ref{sec:mns-lts}. 
This suggests that modular lattices 
may play an organizing role in penumbral moonshine, similar to that played by the Niemeier lattices (\ref{eqn:Gell}) in umbral moonshine 
(though the precise nature of this aspect of penumbral moonshine remains to be elucidated).
\end{enumerate}

We restrict here to $D_0$-optimality for $D_0=-3$ and $D_0=-4$, both for practical and theoretical reasons. 
One practical concern is the size of this manuscript. 
A slightly less prosaic practical concern is the size of the Fourier coefficients in (\ref{eqn:int-pfm:Fpdr_Fourier}). These coefficients grow more and more quickly as $D_0$ increases in magnitude,
and the more quickly they grow, the harder it gets to guess the corresponding groups $G^{(\pd)}$, and therefore also the corresponding structures (\ref{eqn:int-pmo:Wpd}--\ref{eqn:int-pmo:Fpdgr}).

A theoretical, and perhaps more interesting reason for restricting to $D_0=-3$ and $D_0=-4$ is the fact that the genus zero criterion for the existence of $F^{(\pd)}=(F^{(\pd)}_r)$ as in (\ref{eqn:int-pfm:Fpdr_Fourier}), that we mentioned in the penultimate paragraph of \S~\ref{sec:int-pfm}, admits a converse for these values of $D_0$, but not in general. 
Indeed, we prove in \cite{pmz} that $F^{(\pd)}=(F^{(\pd)}_r)$ as in (\ref{eqn:int-pfm:Fpdr_Fourier}) does not exist if $D_0=-3$ or $D_0=-4$ and the lambency $\ell$ in $\pd=(D_0,\ell)$ has positive genus (at least under some conditions on the level of $\ell$), and also explain some ways that this statement fails, for more general values of $D_0$. 
(See op.\ cit.\ for explicit examples and more discussion of this.) 

The reader may recall from \S~\ref{sec:int-opt} that the genus zero criterion on $\ell$ is both sufficient and necessary in the umbral case. That is, a mock modular form $H^{(\ell)}$ as in (\ref{eqn:int:phiellequalsHellrthetamr}--\ref{eqn:int:modularinvariance}), that satisfies the optimality condition (\ref{eqn:int:umbral_optimality}) and the lambency condition (\ref{eqn:int:umbral_lambency_condition}), can be involved in moonshine if and only if the (necessarily non-Fricke) lambency $\ell$ has genus zero.
With this in mind we take the 
above-mentioned necessity of the genus zero criterion for the (necessarily Fricke) lambency $\ell$,
in the penumbral cases that $D_0=-3$ and $D_0=-4$, as motivation for our restriction to these values of $D_0$ in this work. 

The above notwithstanding, it is not our intention to discourage the pursuit of penumbral moonshine for $D_0<-4$. Indeed, we have no concrete evidence to offer 
that penumbral moonshine is not more general than the 
phenomena that we present in this paper. It may be that the most interesting instances are yet to be discovered.
(That said, the adventurous reader should beware of the non-examples that we discuss in \S~\ref{app:non}, and also keep in mind that the extension of penumbral moonshine to non-fundamental $D_0$ will generally necessitate a choice of $r_0$ in (\ref{eqn:int:penumbral_lambency_condition}), and a corresponding enrichment $\pd=(D_0,r_0,\ell)$ of the notation (\ref{eqn:int:lambdency}).)

Another restriction we make in this work is to only consider the lambencies $\ell=m+n,n',\dots$ with square-free level $m$. Again there are practical and theoretical reasons for this. A practical reason is the fact that the Fourier coefficients in (\ref{eqn:int-pfm:Fpdr_Fourier}) are very small, for the (very few) genus zero lambencies arising that don't have square-free level. This prevents the corresponding groups $G^{(\pd)}$, and module structures $W^{(\pd)}$ (\ref{eqn:int-pmo:Wpd}), from being very interesting in these cases. A theoretical reason is the fact that the theta series ambiguity in $F^{(\pd)}=(F^{(\pd)}_r)$ as in (\ref{eqn:int-pfm:Fpdr_Fourier}) turns out to be as small as possible when $m$ is square-free, at least if we require the symmetry condition (\ref{eqn:int:franequalsfr}). That is, for $\ell$ Fricke, genus zero, and of square-free level $m$, the prescription (\ref{eqn:int-pfm:frthetamr0}) gives the only theta series solution to (\ref{eqn:int:penumbral_optimality}--\ref{eqn:int:penumbral_lambency_condition}) that is also consistent with the symmetry condition (\ref{eqn:int:franequalsfr}), so that the vector-valued modular form $F^{(\pd)}$ in (\ref{eqn:int-pfm:Fpdr_Fourier}) is uniquely determined by $C^{(\pd)}(D_0,r_0)$ and $C^{(\pd)}(1,1)$. On this basis, $F^{(\pd)}$ 
is as strongly as possible an analog of a normalized principal modulus, when the level of $\ell$ in $\pd=(D_0,\ell)$ is square-free.

In this work we take $C^{(\pd)}(1,1)=0$ for all $\pd$ arising (cf.\ (\ref{eqn:int:F-31}) and (\ref{eqn:int:F-31Fourierseries})), and use representation theoretic considerations to fix the $C^{(\pd)}(D_0,r_0)$. 
The condition that $C^{(\pd)}(1,1)=0$ just seems to be preferred by penumbral moonshine, at least for $D_0=-3$ and $D_0=-4$.
The determination of $C^{(\pd)}(D_0,r_0)$ comes down to the degrees of the number fields defined by the irreducible characters of $G^{(\pd)}$, and a (practical) desire that the McKay--Thompson series (\ref{eqn:int-pmo:Fpdgr}) all have integer coefficients. In umbral moonshine the maximal degree of a number field defined by an irreducible character is $1$ or $2$, for every umbral group $G^{(\ell)}$ (\ref{eqn:Gell}). 
This is the reason we can take $C^{(\ell)}(1,1)=-2$ for all umbral lambencies $\ell$ (cf.\ \S~\ref{sec:int-opt}), and avoid irrationalities in the Fourier coefficients (\ref{eqn:int:Hellgr}) of the $H^{(\ell)}_g$. In penumbral moonshine the maximal degree is sometimes larger than $2$. Here we opt to adjust $C^{(\pd)}(D_0,r_0)$ accordingly, rather than making a uniform choice.

Although we cannot offer a comprehensive penumbral counterpart to the role that Niemeier lattices play (\ref{eqn:Gell}), in defining the groups of umbral moonshine (cf.\ \ref{itm:int-pmo:evidence5}.\ above), we can point to a  connection to the monster, at least for $D_0=-3$ and $D_0=-4$, and $m$ square-free. 
For this note the curious fact that $G^{(-3,1)}=\Th$ centralizes an element of order $3$ in the monster. 
We find that this 
connection extends, 
in the sense that $G^{(-4,1)}$, being a group of the form $2.F_4(2).2$ (see (\ref{eqn:pmo-gps:2F422seq})), centralizes an element of order $4$ in the monster,
and for either $D_0=-3$ or $D_0=-4$, the group $G^{(\pd)}$ is 
closely related to the centralizer of an element of order $m$ in $G^{(D_0,1)}$, when $\pd=(D_0,\ell)$ and $m$ is the level of $\ell$. In particular, at least when $D_0$ and $m$ are coprime, the group $G^{(\pd)}$ is closely related to the centralizer of some element of $\MM$.

In the spirit of McKay's original observation on the elliptic modular invariant (\ref{eqn:int-mns:ellmodinv}) and the monster, we pause here to note that the $\pd=(-4,1)$ counterpart to (\ref{eqn:int:F-31}) is
\begin{align}
\begin{split}\label{eqn:int:F-41}
\breve{F}^{(-4,1)}(\tau) &= 
\sum_{D\geq -4} \breve{C}^{(-4,1)}(D)q^D \\
&=  2q^{-4} -492 +2\cdot 142884 q^4 -2\cdot 565760 q^5+2\cdot 18473000 q^8 +\dots,
\end{split}
\end{align}
 and both $142884$ and $565760$ occur as dimensions of irreducible representations of $G^{(-4,1)}$.

We offer two more comments on the nature of the modules $W^{(\pd)}$ (\ref{eqn:int-pmo:Wpd}). Firstly, we have described the homogeneous components of $W^{(\pd)}$ as virtual $G^{(\pd)}$-modules, and similarly for the homogeneous components of the $G^{(\ell)}$-modules (\ref{eqn:int:Kell}) of umbral moonshine, but the reader may be aware that the virtual-ness in the umbral case is of a very special kind. 
Indeed, each $K^{(\ell)}_{r,-\frac{D}{4m}}$ is actually either a module, or $-1$ times a module, depending on $r$ (cf.\ (\ref{eqn:int:umbral_lambency_condition})) and the sign of $D$. A similar statement holds for the modules $W^{(\pd)}_{r,\frac{D}{4m}}$ of Thompson moonshine, where $\pd=(-3,1)$ (but the dependence on $r$ is a bit different 
in this case).
We find (see \S~\ref{sec:mns-mds}) that this practical non-virtual-ness of the modules of Thompson moonshine extends to all cases $\pd=(D_0,\ell)$ of penumbral moonshine with $D_0=-3$, but not, in general for $D_0=-4$.

For a final comment on modules, 
and 
for a conclusion to this introductory exploration of the penumbral moonshine landscape, 
we note that the 
notation (\ref{eqn:int:Kell}), for the modules of umbral moonshine, reflects the 
origin of the original Mathieu moonshine observation (\ref{eqn:int:H2}), in the physics of K3 surfaces. 
(See \cite{MR3366057,MR3832169,MR3465528} for results that relate other cases of umbral moonshine to physical models involving K3 surfaces.)
We are not yet aware of any counterpart geometric or physical aspect to penumbral moonshine, either for $D_0=-3$ or $D_0=-4$.
However, in another companion paper, \cite{pmt}, we develop connections between Thompson moonshine (\ref{eqn:int:F-31}--\ref{eqn:int:varphi31-tht-dcm}) 
and generalized moonshine for the monster (see \cite{Carnahan:2012gx,generalized_moonshine}), and provide evidence that these connections extend to other cases of penumbral moonshine (see also \cite{pmp}).  
Generalized monstrous moonshine is now well-understood at the level of CFTs, VOAs, and Borcherds--Kac--Moody algebras, 
so the results of op.\ cit.\ give us some confidence that penumbral moonshine will ultimately manifest rich algebraic and physical features. 
We also mention \cite{MR3858593}, on account of the fact that it includes physical constructions of skew-holomorphic Jacobi forms, which---although different from those that appear in this work---are involved in similar phenomena.

The organization of the rest of the paper is as follows. In \S~\ref{sec:not} we give a guide to notation, and in \S~\ref{sec:pre} we present a review of 
necessary background material. 
The content of \S~\ref{sec:mns} is as described above. In particular, our main results appear in \S~\ref{sec:mns-mds}.
In \S~\ref{sec:sum} we summarize our results and suggest some open problems for future study. 
The appendices contain comments on $D_0<-4$, and 
data that defines the penumbral McKay--Thompson series $F^{(\pd)}_g$.

\section*{Acknowledgements}

We would like to thank Nathan Benjamin, Scott Carnahan, Miranda Cheng, Sarah Harrison, Shamit Kachru, Michael Mertens, Ken Ono, Natalie Paquette, Roberto Volpato, and Max Zimet for helpful discussions. 
B.R.~acknowledges support from the U.S.\ National Science Foundation (NSF) under grant PHY-1720397. 
J.D.~acknowledges support from the NSF (DMS-1203162, DMS-1601306), and the Simons Foundation (\#316779, \#708354). 
J.H.~acknowledges support from the NSF under grant PHY-1520748, and the hospitality of the Aspen Center for Physics supported by the NSF under grant PHY-1607611. 
This research was also supported in part by the NSF under grant PHY-1748958.

\section{Notation}
\label{sec:not}

\begin{footnotesize}

\begin{list}{}{\itemsep -1pt \labelwidth 23ex \leftmargin 13ex}

\item
[$\sqrt{\cdot}$]
We define $\sqrt{z}:=e^{i\frac\theta2}$ for $z=e^{i\theta}$ where $-\pi<\theta\leq\pi$.

\item
[$n*n'$]
We set $n*n':=\frac{nn'}{\gcd(n,n')^2}$ for integers $n$ and $n'$, not both zero. See (\ref{eqn:pre-mod:nstarnprime}).

\item
[$a(n)$]
The assignment $n\mapsto a(n)$ defines an isomorphism $\Ex_m\to O_m$. 
See (\ref{eqn:pre-jac:an}).

\item
[$C_\vf(D,r)$]
A Fourier coefficient of a 
weakly skew-holomorphic 
Jacobi form $\varphi$. See (\ref{eqn:int:varphiFourierseries}), (\ref{eqn:pre-jac:Fou}).

\item
[$C^{(\pd)}(D,r)$]
A Fourier coefficient of $F^{(\pd)}_{r}$. 
See (\ref{eqn:int-pfm:Fpdr_Fourier}) and (\ref{eqn:mns-fms:Fpdgr}). 

\item
[$C^{(\pd)}_g(D,r)$]
A Fourier coefficient of $F^{(\pd)}_{g,r}$.
See (\ref{eqn:mns-fms:Fpdgr}).

\item
[$C^{m,D_0,r_0}_{n,\chi}(D,r)$]
A Fourier coefficient of $R^{m,D_0,r_0}_{n,\chi,r}$.
See (\ref{rad-fourier}).

\item
[${C}^{(\pd)}_{n,\chi}(D,r)$]
A Fourier coefficient of ${R}^{(\pd)}_{n,\chi}$.
See (\ref{normalized_rademacher_Fou}).

\item
[$D_0$]
A negative discriminant. 
See (\ref{eqn:int:penumbral_optimality}).

\item
[$\ex(\,\cdot\,)$]
We set $\ex(x):=e^{2\pi i x}$.

\item
[$\E_m$]
The space of elliptic forms of index $m$. See (\ref{eqn:pre-jac:thetadecomposition}).

\item
[$\E_m^\a$]
An eigenspace for the action of $\Ex_m$ on $\E_m$. 
See (\ref{eqn:pre-jac:mcEm_decomp}).

\item
[$\Ex_m$]
The group of exact divisors of $m$. Cf.\ (\ref{eqn:pre-jac:an}).

\item
[$F^{(\pd)}$]
A shorthand for $F^{(\pd)}_{g}$ in case $g=e$. 
Cf.\ (\ref{eqn:int-pfm:Fpdr_Fourier}) and (\ref{eqn:int-pmo:Fpdgr}).

\item
[$\breve F^{(\pd)}$]
A shorthand for $\breve F^{(\pd)}_{g}$ in case $g=e$. 
Cf.\ (\ref{eqn:int:F-31}), (\ref{eqn:int:F-41}) and (\ref{eqn:mns-mod:breveFpdg}).

\item
[$F^{(\pd)}_r$]
A shorthand for $F^{(\pd)}_{g,r}$ in case $g=e$. 
See (\ref{eqn:int-pfm:Fpdr_Fourier}) and cf.\ (\ref{eqn:int-pmo:Fpdgr}).

\item
[$F^{(\pd)}_g$]
A McKay--Thompson series arising from
$W^{(\pd)}$. 
Cf.\ (\ref{eqn:int-pmo:Fpdgr}).

\item
[$\breve F^{(\pd)}_g$]
A McKay--Thompson series arising from $\breve W^{(\pd)}$.
See (\ref{eqn:pre-vvf:vector_to_scalar}) and (\ref{eqn:mns-mod:breveFpdg}).

\item
[$F^{(\pd)}_{g,r}$]
A component function of $F^{(\pd)}_g$. 
See (\ref{eqn:int-pmo:Fpdgr}).

\item
[$G^{(\pd)}$]
The finite group attached to a lambdency $\pd$. See Tables \ref{tab:pmo-gps:penumbralgroupsD0m3} and \ref{tab:pmo-gps:penumbralgroupsD0m4}.

\item
[$(\gamma,\upsilon)$]
An element of $\widetilde{\SL}_2(\ZZ)$. 
Cf.\ (\ref{eqn:pre-mod:slashknugammaupsilon}).

\item
[$\Gamma$]
A subgroup of $\SL_2(\RR)$ that is commensurable with $\SL_2(\ZZ)$. Cf.\ (\ref{eqn:pre-mod:XGamma}).

\item
[$\widetilde\Gamma_0(n)$]
The preimage of $\Gamma_0(n)$ under the natural map $\widetilde{\SL}_2(\ZZ)\to \SL_2(\ZZ)$. 
See \S~\ref{sec:pre-vvf}.

\item
[$\Gamma_0(m)$]
The Hecke congruence group of level $m$. 
See (\ref{eqn:int:Gamma0m}).

\item
[$\Gamma_0(m)+K$] 
A subgroup of $\SL_2(\RR)$ that normalizes $\Gamma_0(m)$. See (\ref{eqn:pre-mod:Gamma0m+K}).

\item
[$J^\wsh_{k,m}$]
A shorthand for $J^\wsh_{k,m}(1)$. 
See \S~\ref{sec:pre-jac}.

\item
[$J^{\wsh,\a}_{k,m}$]
For $\a\in \widehat{O}_m$ we set $J^{\wsh,\a}_{k,m}:=J^{\wsh}_{k,m}\cap\E_m^\a$. See (\ref{eqn:pre-jac:Jwshkm_decomp}).

\item
[$J^\wsh_{k,m}(n)$]
A shorthand for $J^\wsh_{k,m}(n,1)$. 
See \S~\ref{sec:pre-jac}.

\item
[$J^\wsh_{k,m}(n,\chi)$]
{A certain space of weakly skew-holomorphic Jacobi forms. See \S~\ref{sec:pre-jac}.} 

\item
[$K$]
A subgroup of $\Ex_m$ for some $m$. See (\ref{eqn:pre-mod:Gamma0m+K}).

\item
[$\ell$]
A lambency (pronounced ``{\sl lam}-ben-see''). 
See (\ref{eqn:int:ell}). 

\item
[$\pd$]
A lambdency (pronounced ``{\sl lam}-den-see'').
See (\ref{eqn:int:lambdency}).

\item
[$L^{(\pd)}$]
A lattice attached to $\pd$, for certain lambdencies $\pd$. See Tables \ref{tab:pmo-gps:penumbralgroupsD0m3} and \ref{tab:pmo-gps:penumbralgroupsD0m4}, and \S~\ref{sec:mns-lts}.

\item
[$m+n,n',\dots$]
A lambency and a shorthand for a discrete subgroup of $\SL_2(\RR)$. See (\ref{eqn:int:ell}) and (\ref{eqn:pre-mod:Gamma0m+K}). 

\item
[$M^{\wh,+}_{k,m}(n,\chi)$] 
A certain space of weakly holomorphic modular forms of half-integral weight. See (\ref{eqn:pre-mod:Mwhpluskmnchi}).

\item
[$O_m$]
The multiplicative group of $a\xmod 2m$ such that $a^2\equiv 1\xmod 4m$. See (\ref{eqn:pre-jac:Om}).

\item
[$\widehat{O}_m$]
The group $\widehat{O}_m:=\hom(O_m,\CC^\ast)$ of characters of $O_m$. Cf.\ (\ref{eqn:pre-jac:mcEm_decomp}).

\item
[$\Omega_m(n)$]
The Omega matrix associated to a divisor $n$ of $m$. See (\ref{Omega_matrices}).

\item
[$\Omega_m^\a$]
A linear combination of Omega matrices associated to a character $\a$ of $O_m$. See (\ref{eqn:pre-jac:OmegamK}).

\item
[$P_m^\a$]
A projection operator on elliptic forms. See (\ref{eqn:pre-jac:projector}).

\item
[$\vf^{(\pd)}$]
A shorthand for $\vf^{(\pd)}_g$ in case $g=e$. 
Cf.\ (\ref{eqn:mns-fms:varphipdg}).

\item
[$\vf^{(\pd)}_g$]
The weakly skew-holomorphic Jacobi form attached to the element $g\in G^{(\pd)}$. 
See (\ref{eqn:mns-fms:varphipdg}).

\item
[$q$]
We set $q:=\ex(\tau)$ for $\tau\in \HH$.

\item
[$R^{m,D_0,r_0}_{n,\chi,r}$]
A Rademacher sum.
See (\ref{rad-fourier}).

\item
[${R}^{(\pd)}_{n,\chi}$]
A normalized Rademacher sum.
See (\ref{normalized_rademacher}--\ref{normalized_rademacher_Fou}).

\item
[$\rho_{n|h}$]
A homomorphism of groups from $\Gamma_0(n)$ to the $h$-th roots of unity in $\CC$. See (\ref{eqn:pre-mod:rhonh}).

\item
[$\varrho_m$]
The Weil representation $\varrho_m:\mpt(\ZZ)\to\GL_{2m}(\CC)$ of index $m$. 
See (\ref{eqn:pre-mod:ind_m_Weil_rep}).

\item
[$\mpt(\ZZ)$]
The metaplectic double cover of $\SL_2(\ZZ)$. 
See \S~\ref{sec:pre-vvf}.

\item
[$\theta_m$]
The vector-valued function whose components are the $\theta_{m,r}$. 
Cf.\ (\ref{eqn:thetamr}).

\item
[$\theta_m^0$]
The vector-valued function whose components are the $\theta_{m,r}^0$. 
Cf.\ (\ref{eqn:pre-vvf:thetaL0}--\ref{eqn:thetamr}).

\item
[$\theta_{m,r}$]
A $2$-variable theta series.
See (\ref{eqn:thetamr}).

\item
[$\theta^{0}_{m,r}$]
A Thetanullwert.
Cf.\ (\ref{eqn:pre-vvf:thetaL0}--\ref{eqn:thetamr}).

\item
[$W_n$]
An Atkin--Lehner involution. See (\ref{eqn:pre-mod:Wn}).

\item
[$W^{(\pd)}$]
A virtual graded $G^{(\pd)}$-module attached to a lambdency $\pd$. 
See (\ref{eqn:int-pmo:Wpd}) and \S~\ref{sec:mns-mds}.

\item
[$\breve W^{(\pd)}$]
A virtual graded $G^{(\pd)}$-module attached to a lambdency $\pd$. 
Cf.\ (\ref{eqn:int:W-31}--\ref{eqn:int:F-31gW-31}) and see (\ref{eqn:mns-mds:breveWpdfromWpd}).

\item
[$\breve W^{(\pd)}_D$]
A homogeneous component of $\breve W^{(\pd)}$. 
See (\ref{eqn:mns-mds:breveWpdfromWpd}).

\item
[$W^{(\pd)}_{r,\frac{D}{4m}}$]
A homogeneous component of $W^{(\pd)}$. 
See (\ref{eqn:int-pmo:Wpd}) and \S~\ref{sec:mns-mds}.

\item
[$\W_m(n)$] 
An Eichler--Zagier operator on $\E_m$. 
See (\ref{eqn:pre-jac:EZop}).

\item[$X_\Gamma$]
The Riemann surface associated to $\Gamma<\SL_2(\RR)$.
See (\ref{eqn:pre-mod:XGamma}). 

\item
[$\chi$] 
A character of $\Gamma_0(n)$ of the form $\chi=\rho_{n|h}^v$, for positive integers $n$, $h$ and $v$.

\item
[$y$]
We set $y:=\ex(z)$ for $z\in\CC$.

\end{list}

\end{footnotesize}

\section{Preliminaries}\label{sec:pre}

In this section we review some preliminary notions. 
We review 
scalar-valued modular 
forms in \S~\ref{sec:pre-mod}, and review vector-valued modular forms in \S~\ref{sec:pre-vvf}. We review skew-holomorphic Jacobi forms in \S~\ref{sec:pre-jac}, and the Eichler--Zagier operators that refine them in \S~\ref{sec:pre-ezo}. Finally, we discuss Rademacher sums in \S~\ref{sec:pre-rad}.

\subsection{Modular Forms}\label{sec:pre-mod}

The {\em modular group} is $\SL_2(\ZZ)$. 
Say that a group $\Gamma<\SL_2(\RR)$ is {\em commensurable} {with the modular group} 
if the intersection $\Gamma\cap\SL_2(\ZZ)$ has finite index in both $\Gamma$ and $\SL_2(\ZZ)$. If $\Gamma$ is such a group, its action on $\HH$, given by $\gamma\tau:=\frac{a\tau+b}{c\tau+d}$ for $\gamma=\left(\begin{smallmatrix} a&b\\c&d\end{smallmatrix}\right)$, extends naturally to 
$\widehat\QQ:=\QQ\cup\{\infty\}$.
(See e.g.\ Proposition 1.10.2 in \cite{Dun_ArthGrpsAffE8Dyn}.)

We will only consider subgroups of $\SL_2(\RR)$ that are commensurable with the modular group in the above sense. 
The particular examples of interest here are the {Hecke congruence subgroups} (\ref{eqn:int:Gamma0m}), and the extensions of these by Atkin--Lehner involutions. To define these extensions, first recall that a positive integer $n$ is said to be an \emph{exact divisor} of $m$ if it is both a divisor of $m$ and also coprime with $\frac{m}{n}$. 
The set of exact divisors of $m$ is denoted $\mathrm{Ex}_m$. 
For each $n\in\mathrm{Ex}_m$, the set 
\begin{align}\label{eqn:pre-mod:Wn}
W_n := 
\left.\left\{\frac{1}{\sqrt{n}}
\left(\begin{matrix}an & b \\ cm & dn \end{matrix}\right)\,\right|\, a,b,c,d\in \ZZ, adn^2-bcm=n       \right\}
\end{align}
is a coset of $\Gamma_0(m)$ in $\SL_2(\RR)$. In particular, $\Gamma_0(m)=W_1$. 
In this work we refer to the cosets $W_n$, and sometimes also their representative elements, as {\em Atkin--Lehner involutions} of $\Gamma_0(m)$. 
These cosets obey the same group law as exact divisors of $m$, namely $W_n W_{n'} = W_{n\ast n'}$ for $n,n'\in\Ex_m$,
where 
\begin{gather}\label{eqn:pre-mod:nstarnprime}
	n*n':=\frac{nn'}{\gcd(n,n')^2}.
\end{gather}
In particular, $W_n^2 = W_1=\Gamma_0(m)$ for $n\in \Ex_m$. 
As we have mentioned in \S~\ref{sec:int-mot}, the coset $W_m$ is represented by the {Fricke involution} (\ref{eqn:int:Fricke}). 

To any subgroup $K=\{1,n,n'\dots\}$ of $\mathrm{Ex}_{m}$ 
we can attach an {\em Atkin--Lehner extension} of $\Gamma_0(m)$ 
by setting 
\begin{align}\label{eqn:pre-mod:Gamma0m+K}
\Gamma_0(m) +K:=
\bigcup_{n\in K} W_n.
\end{align}
We say that $\Gamma_0(m)+K$ is {\em Fricke} or {\em non-Fricke} according as $K$ contains $m$ or not. 
The subgroup 
of $\SL_2(\RR)$ represented by a lambency $\ell=m+n,n',\dots$ (\ref{eqn:int:ell}) is none other than  $\Gamma_0(m)+K$ (\ref{eqn:pre-mod:Gamma0m+K}), where $K=\{1,n,n',\dots\}$.

Of particular importance in this work, and moonshine more generally, are the $\Gamma_0(m)+K$ which have genus zero. To define the genus of a group $\Gamma<\SL_2(\RR)$ that is commensurable with the modular group note that
\begin{gather}\label{eqn:pre-mod:XGamma}
X_\Gamma:=
\Gamma\backslash \HH\cup\widehat \QQ
\end{gather}
naturally admits the structure of a compact Riemann surface (see e.g.\ \cite{Shi_IntThyAutFns}).
The {\em genus} of $\Gamma$ is just the genus of $X_\Gamma$ as a surface. So $\Gamma$ has genus zero if $X_\Gamma$ is isomorphic to the Riemann sphere 
$\widehat\CC=\CC\cup\{\infty\}$.
There are only finitely many Atkin--Lehner extensions $\Gamma_0(m)+K$ with genus zero. These are classified in \cite{Fer_Genus0prob}. See Table \ref{tab:pmo-gps:FrickeGenusZero} for a list of the Fricke cases.

Modular forms for $\Gamma$ are functions on $\HH$ that define sections of line bundles on $X_\Gamma$ (\ref{eqn:pre-mod:XGamma}). 
To put this concretely let $k\in \frac12\ZZ$, let $\nu$ be a $\CC$-valued function on $\Gamma$, 
and for $\gamma\in \SL_2(\RR)$ define
\begin{gather}\label{eqn:pre-mod:slashaction}
	(f|_{k,\nu}\gamma)(\tau):=\left(\sqrt{c\tau+d}\right)^{-2k}\nu(\gamma)f(\gamma\tau)
\end{gather}
in case $\gamma=\left(\begin{smallmatrix}*&*\\c&d\end{smallmatrix}\right)$, 
where $\sqrt{re^{2\pi i t}}:=\sqrt{r}e^{\pi i t}$ for $r>0$ and $0<t<\pi$.
Then a {\em weakly holomorphic modular form} of weight $k\in \frac12\ZZ$ for $\Gamma$ with multiplier $\nu$ 
is a holomorphic function $f:\HH\to \CC$ such that $f|_{k,\nu}\gamma=f$ for all $\gamma\in \Gamma$, 
and such that for every $\gamma\in \SL_2(\ZZ)$ there is a constant $C>0$ such that $(f|_{k}\gamma)(\tau) = O(e^{C\Im(\tau)})$ as $\Im(\tau) \to \infty$. Here we write $f|_{k}\gamma$ for 
$f|_{k,\nu}\gamma$ when $\nu\equiv 1$ is trivial. A weakly holomorphic  modular form is called a {\em holomorphic} modular form if it is bounded near the cusps of $\Gamma$. That is, a holomorphic modular form should satisfy $(f|_k\gamma)(\tau)=O(1)$ as $\Im(\tau)\to \infty$, for each $\gamma\in \SL_2(\ZZ)$. A holomorphic modular form $f$ is called {\em cuspidal} if $f|_k\gamma$ vanishes as $\Im(\tau)\to \infty$, for each $\gamma\in \SL_2(\ZZ)$. 

A weakly holomorphic modular form of weight $0$ with trivial multiplier for $\Gamma$ naturally defines a morphism of Riemann surfaces $X_\Gamma\to\widehat{\CC}$ (cf.\ (\ref{eqn:pre-mod:XGamma})). 
Such a weakly holomorphic modular form is called a {\em principal modulus} or {\em hauptmodul} for $\Gamma$ if the corresponding map $X_\Gamma\to \widehat{\CC}$ is an isomorphism. So principal moduli exist only for genus zero groups. A principal modulus $f$ for $\Gamma$ is called {\em normalized} if $f(\tau)=q^{-\frac1h}+O(q^{\frac1h})$ as $\Im(\tau)\to \infty$, where $q=\ex(\tau)$ and where $h$ is the smallest rational number such that $\Gamma$ contains $\left(\begin{smallmatrix}1&h\\0&1\end{smallmatrix}\right)$. The McKay--Thompson series (\ref{eqn:int:Tg}) of monstrous moonshine are normalized principal moduli---in particular, the elliptic modular invariant (\ref{eqn:int-mns:ellmodinv}) is the normalized principal modulus for $\SL_2(\ZZ)$---and it is in this sense that they are organized by genus zero groups.

In this work we are only interested in weakly holomorphic modular forms for groups $\Gamma$ that contain some Hecke congruence group $\Gamma_0(n)$. We say that a weakly holomorphic modular form for such a group $\Gamma$ has {\em level} $n$ if $n$ is the smallest positive integer for which this containment holds.

Other than modular forms with trivial multiplier, we will also be interested in modular forms that lie in Kohnen plus spaces of half-integral weight \cite{MR575942,MR660784}. To explain what these are, 
let $\psi_0:\Gamma_0(4)\to\mathbb{C}^\ast$ be the 
function defined by setting
\begin{align}\label{eqn:pre-mod:psi0}
\psi_0(\gamma) = \left(\frac{c}{d}\right)\epsilon_d,
\end{align}
for $\gamma=\left(\begin{smallmatrix} *&*\\c& d\end{smallmatrix}\right)$, where $\left(\frac{c}{d}\right)$ is the Kronecker symbol (see e.g.\ p.503 of \cite{GKZ87}), and $\epsilon_d$ is $1$ or $i$ according as $d\equiv 1\xmod 4$ or $d\equiv 3\xmod 4$. 
Next, for $k\in \ZZ+\frac12$ and $M$ a positive integer, suppose 
that $f$ is a weakly holomorphic modular form of weight $k$ for $\Gamma_0(4M)$, with multiplier $\nu=\psi_0$. 
Then $f$ admits a {\em Fourier series expansion} of the form
\begin{gather}\label{eqn:pre-svf:FourierSeries}
	f(\tau)=\sum_n c_f(n)q^n
\end{gather}
(where $q=\ex(\tau)$). 
If $k\in 2\ZZ+\frac12$ then the function $f$ in (\ref{eqn:pre-svf:FourierSeries}) is said to be a weakly holomorphic modular form of weight $k$ in the {\em Kohnen plus space} for $\Gamma_0(4M)$ if $c_f(n)=0$ whenever $n$ is not a square modulo $4$, whereas if $k\in 2\ZZ-\frac12$ then the condition on exponents is that $c_f(n)=0$ whenever $-n$ is not a square modulo $4$.

In this work we will be interested in certain refinements of this notion. To detail these we utilize the function $\rho_{n|h}$, defined for $\gamma\in \Gamma_0(n)$ by setting
\begin{gather}\label{eqn:pre-mod:rhonh}
\rho_{n|h}(\gamma) := \ex\left(-\frac{cd}{nh}\right)
\end{gather}
in case $\gamma=\left(\begin{smallmatrix}*&*\\c&d\end{smallmatrix}\right)$. (Cf.\ \S~2 of \cite{MR3021323}.) 
Now $\rho_{n|h}$ is trivial when restricted to $\Gamma_0(nh)$, so given positive integers $M$ and $h$, it is natural to consider those weakly holomorphic modular forms of weight $k$ in the Kohnen plus space for $\Gamma_0(4Mh)$ that satisfy $f|_{k,\nu}\gamma=f$ 
for $\gamma\in \Gamma_0(4M)$, where now $\nu=\psi_0\rho_{4M|h}^v$ for some $v$. We call such a function a weakly holomorphic modular form of weight $k$ in the Kohnen plus space for $\Gamma_0(4M)$ with {\em character} $\rho_{4M|h}^v$. In practice, we will have $M=mn$ where $m$ is the level of a lambency $\ell$ (\ref{eqn:int:ell}), and $n$ is the order of an element in a finite group. 
Taking $\chi=\rho_{4mn|h}^v$, we will use
\begin{gather}\label{eqn:pre-mod:Mwhpluskmnchi}
M_{k,m}^{\wh,+}(n,\chi)
\end{gather}
to denote the space of weakly holomorphic modular forms $f$ in the Kohnen plus space of weight $k$ for $\Gamma_0(4mn)$ with character $\chi$, whose Fourier coefficients $c_f(n)$ vanish unless $(-1)^{k-\frac12}n$ is a square modulo $4m$. 

Holomorphic and cuspidal modular forms in the Kohnen plus spaces are defined via growth conditions at cusps, just as for modular forms in general.

\subsection{Vector-Valued Modular Forms}\label{sec:pre-vvf}

In \S~\ref{sec:int} the McKay--Thompson series of penumbral moonshine have been described as weakly holomorphic vector-valued modular forms of half-integral weight. To explain this notion in more detail it is convenient to work with the {\em metaplectic group} $\widetilde{\SL}_2(\ZZ)$, which we may define to be the double cover of $\SL_2(\ZZ)$ whose elements are the pairs $(\gamma,\upsilon)$, where $\gamma\in \SL_2(\ZZ)$ and $\upsilon:\HH\to \CC$ is either of the two holomorphic functions that satisfy $\upsilon(\tau)^2=c\tau+d$ in case $\gamma=\left(\begin{smallmatrix}*&*\\c&d\end{smallmatrix}\right)$. The multiplication rule is $(\gamma,\upsilon)(\gamma',\upsilon')=(\gamma'',\upsilon'')$, where $\gamma'':=\gamma\gamma'$ and $\upsilon''(\tau):=\upsilon(\gamma'\tau)\upsilon'(\tau)$. 

For $n$ a positive integer 
we reserve the notation $\widetilde{\Gamma}_0(n)$ for the inverse image of $\Gamma_0(n)$ under the natural map $\widetilde{\SL}_2(\ZZ)\to\SL_2(\ZZ)$. 
Also, we set 
\begin{gather}\label{eqn:pre-vvf:tildeStildeT}
\widetilde{S}
:=\left(\left(\begin{smallmatrix}0&-1\\1&0\end{smallmatrix}\right),\sqrt{\tau}\right),\quad
\widetilde{T}
:=\left(\left(\begin{smallmatrix}1&1\\0&1\end{smallmatrix}\right),1\right).
\end{gather} 
Then $\widetilde{S}$ and $\widetilde{T}$ serve as generators for $\widetilde{\SL}_2(\ZZ)$.

For $n$ a positive integer, for $k\in \frac12\ZZ$, and $\nu$ a $\GL(V)$-valued function on $\widetilde\Gamma_0(n)$, 
where $V$ is a complex vector space, a holomorphic function $f:\HH\to V$ is called a {\em weakly holomorphic vector-valued modular form} of weight $k$ and level $n$ with multiplier system $\nu$, 
if $f|_{k,\nu}(\gamma,\upsilon)=f$ for all $(\gamma,\upsilon)\in \widetilde\Gamma_0(n)$, where now
\begin{gather}\label{eqn:pre-mod:slashknugammaupsilon}
	(f|_{k,\nu}(\gamma,\upsilon))(\tau):=\upsilon(\tau)^{-2k}\nu(\gamma,\upsilon)f(\gamma\tau)
\end{gather} 
(cf.\ (\ref{eqn:pre-mod:slashaction})), and if for any $(\gamma,\upsilon)\in \widetilde{\SL}_2(\ZZ)$ there is a $C>0$ such that $\alpha((f|_{k}(\gamma,\upsilon))(\tau))=O(e^{C\Im(\tau)})$ as $\Im(\tau)\to \infty$, for any linear functional $\alpha:V\to \CC$. 
As before we write $f|_{k}(\gamma,\upsilon)$ for the action (\ref{eqn:pre-mod:slashknugammaupsilon}) when $\nu\equiv 1$.
{Holomorphic} vector-valued modular forms and {cuspidal} vector-valued modular forms are defined as in the scalar-valued case (discussed in \S~\ref{sec:pre-mod}), replacing the exponential growth condition with boundedness and vanishing, respectively. 

Of increasing importance in moonshine are vector-valued modular forms for which the multiplier system is the {Weil representation} associated to an even lattice, or the dual of this. 
To define this notion let $L$ be an even lattice,
let $(b^+,b^-)$ be its signature, 
and let $L^\ast$ denote the dual of $L$. 
(For background on lattices see e.g.\ \cite{MR1662447} for a lot, or \S~\ref{sec:mns-lts} for a little.) 
Then the {\em Weil representation} 
attached to $L$ is the unitary representation
$\varrho_L : \widetilde{\SL}_2(\mathbb{Z}) \to \GL(\mathbb{C}[L^\ast/L])$ 
defined by requiring that
\begin{align}\label{weil-rep}
\varrho_L\big(\widetilde{T}\big)e_\gamma = 
\ex\left(\tfrac12{(\gamma,\gamma)}\right)
e_\gamma, \quad
\varrho_L\big(\widetilde{S}\big) e_\gamma = 
\frac{
\ex\left(\frac18{(b^--b^+)}\right)
}{\sqrt{|L^\ast/L|}}\sum_{\delta\in L^\ast/L}
\ex\left(-(\gamma,\delta)\right)
e_\delta,
\end{align}
where $(\cdot\,,\cdot)$ is the symmetric bilinear form of $L$, and $\{e_\gamma\}_{\gamma\in L^\ast/L}$ is the natural basis for the group ring $\mathbb{C}[L^\ast/L]$. 

Here we will be concerned with the case that $(b^+,b^-)=(1,0)$, so that $L=\sqrt{2m}\mathbb{Z}$ for some positive integer $m$. 
For such $L$ we call $\varrho_L$ the {\em index $m$ Weil representation},
write $\varrho_m$ in place of $\varrho_{L}$, and identify $L^*/L$ with $\ZZ/2m\ZZ$ in the natural way.
Then
(\ref{weil-rep})
reduces to 
\begin{align}\label{eqn:pre-mod:ind_m_Weil_rep}
\varrho_m\big(\widetilde{T}\big)e_r = 
\ex\left(\frac{r^2}{4m}\right)
e_r, \quad
\varrho_m\big(\widetilde{S}\big)e_r= 
\frac{
\ex\left(-\frac18\right)
}{\sqrt{2m}}\sum_{s\xmod{2m}} 
\ex\left(-\frac{rs}{2m}\right)
e_s,
\end{align}
for $r$ defined modulo $2m$. 
In practice it can be useful to have explicit expressions for Weil representation matrix elements associated to general elements of the metaplectic group, especially when working with its subgroups. Computationally efficient expressions can be found in \cite{MR3101818}.

It is possible to pass from vector-valued modular forms for a Weil representation (or the dual of such a thing), to scalar-valued modular forms, at the price of an increase in level. In this work are interested in the case that $f=(f_r)$ is a vector-valued modular form for the restriction to $\widetilde{\Gamma}_0(n)$ of the dual of the index $m$ Weil representation, for some $m$. Then 
\begin{align}\label{eqn:pre-vvf:vector_to_scalar}
\breve{f}(\tau) := \sum_{r\xmod 2m}f_r(4m\tau)
\end{align}
is a scalar-valued modular form of level $4mn$. Moreover, $\breve f$ belongs to a Kohnen plus space (cf.\ \S~\ref{sec:pre-mod}).
Specifically, a weakly holomorphic vector-valued modular form of weight $\frac12$, level $n$, and multiplier system $\varrho_m^{-1}\rho_{n|h}^v$ is sent under the map \eqref{eqn:pre-vvf:vector_to_scalar} to a function in the space $M_{\frac12,m}^{\wh,+}(n,\rho_{4mn|h}^v)$ (see (\ref{eqn:pre-mod:Mwhpluskmnchi})).

Note that the lattice $L$ naturally furnishes an example of a vector-valued modular form with multiplier system $\varrho_L^{-1}$. Indeed, if $\theta^0_\gamma(\tau):=\sum_{\lambda\in \gamma+L}q^{\frac12(\lambda,\lambda)}$ denotes the theta series naturally attached to the coset $\gamma\in L^*/L$ 
then 
\begin{gather}\label{eqn:pre-vvf:thetaL0}
\theta_L^0
:=
(\theta^0_\gamma)
=
\sum_{\gamma\in L^*/L}\theta^0_\gamma e_\gamma
\end{gather}
is a vector-valued modular form of weight $k=\frac12(b^-+b^+)$ for $\widetilde{\SL}_2(\ZZ)$ with multiplier $\varrho_L^{-1}$. 

We reserve the notation $\theta^0_{m,r}$ for $\theta^0_\gamma$ when 
$L=\sqrt{2m}\ZZ$
and $r\in \ZZ/2m\ZZ$ corresponds to $\gamma\in L^*/L$, 
and write 
$\theta^0_m:=(\theta^0_{m,r})$ 
for the vector-valued function that takes the {\em Thetanullwerte} $\theta^0_{m,r}$ 
as its components.
Thus $\theta_m^0$
is a vector-valued modular form of weight $\frac12$ for $\widetilde{\SL}_2(\ZZ)$ with multiplier $\varrho^{-1}_m=\varrho^*_m$.
A result of Serre and Stark \cite{MR0472707} states that any holomorphic modular form of weight $\frac12$ is a linear combination of Thetanullwerte $\theta^0_{m,r}$.

Note that the {Thetanullwert} $\theta^0_{m,r}$ is the specialization to $z=0$ of the two variable theta series
\begin{gather}\label{eqn:thetamr}
	\theta_{m,r}(\tau,z):=\sum_{s \equiv r\xmod{2m}} q^{\frac{s^2}{4m}}y^{s},
\end{gather}
where $y=\ex(z)$. We write $\theta_m=(\theta_{m,r})$ for the $2$-variable vector-valued function whose $r$-th component is $\theta_{m,r}$.

\subsection{Jacobi Forms}\label{sec:pre-jac}

We now review the notion of 
skew-holomorphic Jacobi form in more detail. 
To begin say that a smooth function $\varphi:\mathbb{H}\times\mathbb{C}\to\mathbb{C}$ is an {\em elliptic form} of index $m$ if $z\mapsto\varphi(\tau,z)$ is holomorphic for every $\tau\in \mathbb{H}$, and if $\varphi$ is invariant under the index $m$ {\em elliptic action} of $\mathbb{Z}^2$ on functions $\HH\times\CC\to\CC$, 
\begin{align}
(\varphi\vert_m(\lambda,\mu))(\tau,z):= \ex(m\lambda^2\tau+2m\lambda z)\varphi(\tau,z+\lambda\tau+\mu), \ \ \ \ (\lambda,\mu)\in\mathbb{Z}^2.
\end{align}
Write $\mathcal{E}_m$ for the space of index $m$ elliptic forms. Any $\varphi\in\mathcal{E}_m$ admits a {\em theta-decomposition}
\begin{align}\label{eqn:pre-jac:thetadecomposition}
\varphi(\tau,z) = \sum_{r\xmod{2m}} h_r(\tau)\theta_{m,r}(\tau,z)
\end{align}
where $\theta_{m,r}$ is as in (\ref{eqn:thetamr}), and the smooth functions $h_r$ are called the {\em theta-coefficients} of $\varphi$. 
So in particular, the $\theta_{m,r}$ are invariant for the index $m$ elliptic action. 
Concretely, 
we have
\begin{align}\label{eqn:pre-jac:thetamindexmWeil}
\theta_m(\tau,z)=
\upsilon(\tau)^{-1}
\ex\left(-\frac{cmz^2}{c\tau+d}\right)
\varrho_m^*(\gamma,\upsilon)
\theta_m\left(\frac{a\tau+b}{c\tau+d},\frac{z}{c\tau+d}\right)
\end{align}
for $(\gamma,\upsilon)\in \widetilde{\SL}_2(\ZZ)$, 
where 
$\gamma=\left(\begin{smallmatrix}a&b\\c&d\end{smallmatrix}\right)$,
and $\varrho_m^*$ denotes the dual of the representation $\varrho_m$ defined by (\ref{eqn:pre-mod:ind_m_Weil_rep}).
This extends the statement just made (see (\ref{eqn:thetamr})), that the {Thetanullwerte} 
$\theta_{m,r}^0$ comprise the components of a vector-valued modular form 
of weight $\frac12$ for the dual  of the index $m$ Weil representation.

Now for $k\in \ZZ$, and for positive integers $n$, $h$ and $v$, set $\chi=\rho_{n|h}^v$ (see (\ref{eqn:pre-mod:rhonh})), and 
define the weight $k$ and character $\chi$ 
{\em skew-modular} action of $\Gamma_0(n)$ on elliptic functions of index $m$ as 
\begin{align}
\label{eqn:pre-jac:modskewmodaction}
(\varphi\vert^{\mathrm{sk}}_{k,m,\chi}\gamma)(\tau,z) 
:=
(c\tau+d)^{-k}\frac{c\bar\tau+d}{|c\tau+d|}
\ex\left(-\frac{cmz^2}{c\tau+d}\right)
\chi(\gamma)
\varphi\left(\frac{a\tau+b}{c\tau+d},\frac{z}{c\tau+d}\right),
\end{align}
when $\gamma=\left(\begin{smallmatrix}a&b\\c&d\end{smallmatrix}\right)\in\Gamma_0(n)$. Then a {\em weakly skew-holomorphic Jacobi form} of weight $k$, index $m$, level $n$, and character $\chi=\rho_{n|h}^v$ is an elliptic function $\varphi\in\mathcal{E}_m$ that has anti-holomorphic theta-coefficients, is invariant for the weight $k$ and character $\chi$ skew-modular action (\ref{eqn:pre-jac:modskewmodaction}), and is such that the theta-coefficients (see (\ref{eqn:pre-jac:thetadecomposition})) satisfy an exponential growth condition as in \S~\ref{sec:pre-mod}, near each cusp. We denote the space of such functions by $J^{\wsh}_{k,m}(n,\chi)$. 
Note that we suppress the character $\chi$ from notation when it is trivial, and write $J^\wsh_{k,m}$ as a shorthand for $J^\wsh_{k,m}(1)$. In the case that $\vf$ is a weakly skew-holomorphic Jacobi form,
for the theta-decomposition (\ref{eqn:pre-jac:thetadecomposition}) we write
\begin{gather}\label{eqn:pre-jac:thetadecomposition_wsh}
	\vf(\tau,z)=\sum_{r\xmod 2m}\overline{f_r(\tau)}\theta_{m,r}(\tau,z)
\end{gather}
(cf.\ (\ref{eqn:int:skw-hlm-tht-dcm})), since in this situation the theta-coefficients are anti-holomorphic by definition.

It follows from the definitions that, for $s$ an integer, the assignment $\varphi(\tau,z)\mapsto\varphi(\tau,sz)$ defines an injective map 
\begin{gather}\label{eqn:pre-jac:ztosz}
J^\wsh_{k,m}(n)\to J^\wsh_{k,ms^2}(n).
\end{gather} 
Taking $n=1$ in (\ref{eqn:pre-jac:ztosz}) we see that the index $m'$ Weil representation of $\widetilde{\SL}_2(\ZZ)$ contains a copy of the index $m=\frac{m'}{s^2}$ Weil representation in case $m'$ is divisible by a perfect square $s^2$.
Another consequence of the definitions is the fact that 
$\varphi(\tau,z)\mapsto \varphi(M\tau,z)$, for $M$ a positive integer, defines an injection 
\begin{gather}\label{eqn:pre-jac:tautostau}
		J^\wsh_{k,mM}(n)
		\to J^\wsh_{k,m}(Mn).
\end{gather}
Taking $n=1$ in 
(\ref{eqn:pre-jac:tautostau}) we obtain a way to apply statements about weakly skew-holomorphic Jacobi forms of level $1$ to those of higher level. (This will manifest in \S\S~\ref{sec:mns-gps}--\ref{sec:mns-fms}. See in particular Tables \ref{tab:mns-fms:mltrln_m3} and \ref{tab:mns-fms:mltrln_m4}.)

From (\ref{eqn:pre-jac:thetamindexmWeil}--\ref{eqn:pre-jac:thetadecomposition_wsh}) we see that 
the complex conjugates of the theta-coefficients of a weakly skew-holomorphic Jacobi form of weight $k$ and index $m$ comprise the components of a weakly holomorphic vector-valued modular form of weight $k-\frac12$ for the dual of the index $m$ Weil representation (\ref{eqn:pre-mod:ind_m_Weil_rep}).
Also, we see that 
a weakly skew-holomorphic Jacobi form $\vf\in J^\wsh_{k,m}(n,\chi)$ admits a Fourier series expansion
\begin{gather}\label{eqn:pre-jac:Fou}
	\vf(\tau,z)=\sum_{\substack{D,s\in\ZZ\\D\equiv s^2\xmod 4m}}C_\vf(D,s)\bar{q}^{\frac{D}{4m}}q^{\frac{s^2}{4m}}y^s,
\end{gather}
with $C_\vf(D,s)=0$ for $D\ll 0$. 

Note that an application of the skew-modular action (\ref{eqn:pre-jac:modskewmodaction}) with $\gamma=-I$ yields 
\begin{gather}\label{eqn:pre-jac:varphitauminusz}
\vf(\tau,-z)=(-1)^{k+1}\vf(\tau,z)
\end{gather} 
for $\vf\in J^\wsh_{k,m}(n,\chi)$, since $\chi(-I)=1$ when $\chi$ takes the form $\chi=\rho_{n|h}^v$ (see (\ref{eqn:pre-mod:rhonh})).
We also have $\theta_{m,r}(\tau,-z)=\theta_{m,-r}(\tau,z)$ (cf.\ (\ref{eqn:thetamr})), so it follows that
\begin{gather}\label{eqn:pre-jac:fminusr}
	f_{-r}=(-1)^{k+1}f_r
\end{gather}
for $f_r$ as in (\ref{eqn:pre-jac:thetadecomposition_wsh}).

\subsection{Eichler--Zagier Operators}\label{sec:pre-ezo}

For each divisor $n$ of $m$ we may consider the \emph{Eichler--Zagier operator} $\mathcal{W}_m(n)$, which acts on $\mathcal{E}_m$ according to the rule
\begin{align}\label{eqn:pre-jac:EZop}
(\varphi\vert\mathcal{W}_m(n))(\tau,z) := \frac{1}{n}\sum_{a,b=0}^{n-1}\ex\left(m\left(\frac{a^2}{n^2}\tau+2\frac{a}{n}z+\frac{ab}{n^2}\right)\right)\varphi\left(\tau,z+\frac{a}{n}\tau+\frac{n}{n}\right).
\end{align}
In terms of theta-coefficients (\ref{eqn:pre-jac:thetadecomposition}), the action of the Eichler--Zagier operators is $h\mapsto \Omega_m(n)h$, where $h=(h_r)$, and where $\Omega_m(n)$ is a $2m\times 2m$ \emph{Omega matrix}. The Omega matrices commute with the Weil representation 
and have entries given by
\begin{align}\label{Omega_matrices}
\Omega_m(n)_{r,s} := \begin{cases} 1 & \text{if }r=-s\xmod{2}n \text{ and }r=s\xmod\frac{2m}n, \\
0 & \text{otherwise}.
\end{cases}
\end{align}

In the case that $n$ is an exact divisor of $m$ we may alternatively express the action (\ref{eqn:pre-jac:EZop}) as follows. 
Define 
\begin{gather}\label{eqn:pre-jac:Om}
O_m:=\left\{a\xmod 2m\mid a^2\equiv 1 \xmod 4m\right\},
\end{gather}
regard $O_m$ as a group under multiplication, and note that we obtain an isomorphism from the group $\Ex_m$ of exact divisors of $m$ (cf.\ (\ref{eqn:pre-mod:nstarnprime})) to $O_m$, to be denoted $n\mapsto a(n)$, by requiring that $a(n)$ is the unique integer modulo $2m$ such that
\begin{gather}
	\label{eqn:pre-jac:an}
	a(n)
	\equiv 
	-1\xmod 2n,\quad
	a(n)
	\equiv 
	1\xmod \frac {2m}n.	
\end{gather}
Then for $n$ an exact divisor of $m$, for $\varphi$ an elliptic form of index $m$, and for $h_r$ as in (\ref{eqn:pre-jac:thetadecomposition}) we have 
\begin{gather}\label{eqn:pre-jac:Wmnan}
\varphi|\mc{W}_m(n) = \sum_r h_{ra(n)}\theta_{m,r}.
\end{gather}

Note also that, when restricted to exact divisors of $m$, the Eichler--Zagier operators (\ref{eqn:pre-jac:EZop}), and also the Omega matrices (\ref{Omega_matrices}), satisfy the defining relations
\begin{gather}\label{eqn:pre-jac:WmnWmnprime}
	\mc{W}_m(n)\mc{W}_m(n')=\mc{W}_m(n * n'),\\
	\Omega_m(n)\Omega_m(n')=\Omega_m(n * n'),
\end{gather}
(cf.\ (\ref{eqn:pre-mod:nstarnprime})) of the Atkin--Lehner operators (\ref{eqn:pre-mod:Wn}).

In \S~\ref{sec:int-thm} we explained a way (see (\ref{eqn:int:franequalsfr})) in which lambencies may be used to refine spaces of weakly skew-holomorphic Jacobi forms. We now use the Eichler--Zagier operators (\ref{eqn:pre-jac:EZop}--\ref{Omega_matrices}) to perform a closely related task. For this let $m$ be a positive integer, let $\a:O_m\to \CC^*$ be a character of $O_m$ (\ref{eqn:pre-jac:Om}), and let 
$\mc{E}_m^\a$ be the subspace of 
$\mc{E}_m$ composed of forms $\varphi$ such that $\varphi|\mc{W}_m(n) = \a(a(n))\varphi$ for each $n\in \Ex_m$, where $a(n)$ is as in (\ref{eqn:pre-jac:an}). 
Then, writing $\widehat{O}_m$ for the group of characters of $O_m$, we have
\begin{gather}\label{eqn:pre-jac:mcEm_decomp}
	\mc{E}_{m} = \bigoplus_{\a\in\widehat{O}_m} \mc{E}^{\a}_{m}.
\end{gather}
Now set $J^{\wsh,\a}_{k,m}:= J^\wsh_{k,m}\cap \mc{E}^\a_m$. Then it follows from (\ref{eqn:pre-jac:mcEm_decomp}) that we have
\begin{gather}\label{eqn:pre-jac:Jwshkm_decomp}
	J^\wsh_{k,m} = \bigoplus_{\a\in\widehat{O}_m} J^{\wsh,\a}_{k,m}.
\end{gather}

Taking theta-coefficients on both sides of (\ref{eqn:pre-jac:Jwshkm_decomp}) we obtain a statement about vector-valued modular forms for the (dual of the) index $m$ Weil representation. Namely, the characters $\a\in \widehat{O}_m$ define forms that are governed by subrepresentations. It develops that these subrepresentations are irreducible if $m$ is square-free (see e.g.\ \cite{MR2512363}). The existence of the injective map (\ref{eqn:pre-jac:ztosz}) shows how it is that this statement fails when $m$ is not square-free (cf. \S~3.4 of \cite{MR4127159}).

Note that half of the summands in (\ref{eqn:pre-jac:Jwshkm_decomp}) actually vanish. To see this take
$m=n$ in (\ref{eqn:pre-jac:an}--\ref{eqn:pre-jac:Wmnan}) so as to obtain $\vf|\mc{W}_m(m)=\sum_{r\xmod 2m}\overline{f}_{-r}\theta_{m,r}$ for $\vf\in J^\wsh_{k,m}$ as in (\ref{eqn:pre-jac:thetadecomposition_wsh}).
Then from (\ref{eqn:pre-jac:fminusr})
we have that
\begin{gather}\label{eqn:pre-jac:phiWm}
	\vf|\mc{W}_m(m)=(-1)^{k+1}\vf
\end{gather}
for $\varphi\in J^\wsh_{k,m}$. 
Now say that $\alpha\in \widehat{O}_m$ is {\em even} or {\em odd} according as $\alpha(-1)=1$ or $\alpha(-1)=-1$. Then from (\ref{eqn:pre-jac:phiWm}) we conclude that the space $J^{\wsh,\alpha}_{k,m}$ is trivial unless $\alpha$ and $k$ have opposite parity.
In particular, weakly skew-holomorphic Jacobi forms of weight $1$, which are a focus of this work, are necessarily invariant under $\mc{W}_m(m)$. (Cf.\ (\ref{eqn:int:fminusrequalsfr}).)

It will be useful in what follows to consider the projection operators 
$P_m^\a:\mc{E}_m\to\mc{E}_m^\a$.
We will describe them explicitly by defining matrices $\Omega_m^{\alpha}$ that implement the projection at the level of theta-coefficients, 
according to the rule
\begin{align}\label{eqn:pre-jac:projector}
P_m^{\alpha}(h^{\rm t} \theta_{m}) = h^{\rm t}  \Omega_m^{\alpha} \theta_{m}.
\end{align}
(Here $h^{\rm t}$ is the row-vector form of $h=(h_r)$.)
This requirement (\ref{eqn:pre-jac:projector}) does not uniquely determine $\Omega_m^{\alpha}$, but it does if we further demand that $(\Omega_m^{\alpha})_{r,s}=(\Omega_m^{\alpha})_{r,-s}=(\Omega_m^{\alpha})_{-r,s}$. On general grounds, these matrices commute with the Weil representation and therefore must be expressible as a linear combination of the $\Omega_m(n)$ (\ref{Omega_matrices}). Concretely, we have
\begin{align}\label{eqn:pre-jac:OmegamK}
\Omega_m^{\alpha} =
\frac1{\#\Ex_m}\sum_{n\in \Ex_m}\a(a(n))\Omega_m(n).
\end{align}

\subsection{Rademacher Sums}\label{sec:pre-rad}

In \S~\ref{sec:mns-fms} we will use Rademacher sums to specify the McKay--Thompson series of penumbral moonshine at $D_0=-3$. We explain the details of this construction here.
The formulae we present are modest generalizations of expressions that have appeared in various places in the literature, e.g.\ \cite{MR3559204,DunFre_RSMG,2014arXiv1406.0571W}. 

Here the data of a Rademacher sum is as follows.
\begin{enumerate}
\item A positive integer $m$, the index.
\item A pair $(D_0,r_0)$ consisting of a negative discriminant $D_0$ and an integer $r_0$ defined modulo $2m$ which satisfies $D_0\equiv r_0^2\xmod{4}m$. This specifies the singular behavior of the Rademacher sum as $\tau$ approaches the cusp at infinity.
\item A positive integer $n$, the level.
\item A character $\chi:\Gamma_0(n)\to\mathbb{C}^\ast$ of the form $\chi=\rho_{n|h}^v$ (see (\ref{eqn:pre-mod:rhonh})), which contributes to the multiplier system.
\end{enumerate}

To this data, we (formally) attach a Rademacher sum $R^{m,D_0,r_0}_{n,\chi}$. 
When it converges it defines a vector-valued mock modular form of weight $\frac12$ which furnishes the theta-coefficients of a weakly skew-holomorphic mock Jacobi form. In all the cases of interest to us it will actually be modular, belonging to one of the spaces ${J}_{k,m}^{\wsh}(n,\chi^\ast)$.
We will define this vector-valued function 
$R^{m,D_0,r_0}_{n,\chi}=(R^{m,D_0,r_0}_{n,\chi,r})$ 
through its Fourier development, which takes the form
\begin{align}\label{rad-fourier}
R^{m,D_0,r_0}_{n,\chi,r}(\tau) = \delta_{r,\pm r_0}q^{\frac{D_0}{4m}} + \sum_{\substack{D\geq 0\\ D\equiv r^2\xmod{4m}}} C^{m,D_0,r_0}_{n,\chi}(D,r)q^{\frac{D}{4m}},
\end{align}
where the Fourier coefficients are given by 
\begin{align}
\begin{split}
C^{m,D_0,r_0}_{n,\chi}(D,r) &= \sum_{s \xmod{2m}}\sum_{\substack{c>0\\c\equiv 0\xmod{n} }}\sum_{\substack{0\leq a < c\\(a,c) = 1}} B_{\gamma,\frac{1}{2}}\left(\vec{\mu},\frac{D}{4m}\right)_s K_{\gamma,\chi}\left(\vec{\mu},\frac{D}{4m}\right)_{s r}, \\
C^{m,D_0,r_0}_{n,\chi}(0,0) &= \frac{1}{2}
\ex(-\tfrac18)
\sum_{s\xmod{2m}}\frac{\sqrt{-\mu_s}}{\Gamma(\frac{3}{2})} \sum_{\substack{c>0\\c\equiv 0\xmod{n}}}\left(\frac{2\pi }{c}\right)^{\frac{3}{2}}\sum_{\substack{0\leq a < c \\ (a,c)=1}}K_{\gamma,\chi}(\vec{\mu},0)_{sr}  .
\end{split}
\end{align}
In the above, $\gamma = \left(\begin{smallmatrix} a & \ast \\ c & \ast\end{smallmatrix}\right)$, we use $\Gamma(x)$ for the usual Gamma function, and $\vec{\mu}=(\mu_r)$ is the $2m$ component vector which describes the poles of the Rademacher sum: it is zero in every entry except for $r_0$ and $-r_0$, where it is $\mu_{r_0} = \mu_{-r_0} = \frac{D_0}{4m}$. 
We also define 
\begin{align}\label{kloosterman}
\begin{split}
K_{\gamma,\chi}\left(\vec{\mu},\frac{D}{4m}\right)_{sr} &:=  \ex\left(\frac{D}{4m}\frac{d}{c}\right)\chi(\gamma)\varrho_m(\gamma,\upsilon)_{rs}^\dagger\ex\left(\mu_s \frac{a}{c}\right),\\
B_{\gamma,\frac{1}{2}}\left(\vec{\mu},\frac{D}{4m}\right)_{s} &:= 
\ex(-\tfrac18)
\left(-\frac{4m\mu_s}{D}\right)^{\frac{1}{4}}\frac{2\pi}{c} I_{\frac{1}{2}}\left(\frac{4\pi}{c} \sqrt{\frac{-D\mu_s}{4m}}\right),
\end{split}
\end{align}
for $\gamma=\left(\begin{smallmatrix} a&b\\c&d\end{smallmatrix}\right)$ and $\upsilon=\sqrt{c\tau+d}$ (cf.\ \S~\ref{sec:not}),
where $I_{\alpha}(x)$ denotes the modified Bessel function of the first kind, given by the power series expression
\begin{align}
I_\alpha(x) := \sum_{n\geq 0} \frac{1}{\Gamma(m+\alpha+1)m!}\left(\frac{x}{2}\right)^{2m+\alpha}.
\end{align}

We adopt the following conventions. In the case that the character $\chi$ is trivial, we suppress it from the notation, i.e.\ we write $R^{m,D_0,r_0}_{n}$. In the case that $r_0$ is uniquely determined (modulo $2m$) by $D_0$, we suppress it from the notation, and write $R^{m,D_0}_{n,\chi}$. We also attach a {\em normalized} Rademacher sum 
\begin{align}\label{normalized_rademacher}
{R}^{(\pd)}_{n,\chi}(\tau) 
:= \Omega_m^{\alpha}\big(aR^{m,D_0,r_0}_{n,\chi}(\tau)+b\theta_m(\tau)\big) 
\end{align}
to each lambdency $\pd=(D_0,\ell)$. Here $r_0$ is the minimal non-negative integer such that $D_0\equiv r_0^2\xmod 4m$, and the character $\a\in \widehat{O}_m$ is chosen so that $\ker(\a)=\{1,n,n',\dots\}$ when $\ell=m+n,n',\dots$.
Also, $a$ and $b$ are chosen 
so that ${C}^{(\pd)}_{n,\chi}(1,\pm 1)=0$ and ${C}^{(\pd)}_{n,\chi}(D_0,\pm r_0) = 1$,
where 
\begin{align}\label{normalized_rademacher_Fou}
{R}^{(\pd)}_{n,\chi}(\tau) 
= 
\sum_{\substack{D\geq D_0\\D\equiv r^2\xmod{4m}}}{C}^{(\pd)}_{n,\chi}(D,r)q^{\frac{D}{4m}}
\end{align}
is the Fourier series expansion of ${R}^{(\pd)}_{n,\chi}$.

The convergence of the Rademacher sums (\ref{rad-fourier}), (\ref{normalized_rademacher_Fou}) follows from similar arguments to those that appear in \cite{MR3433373}.

\section{Moonshine}\label{sec:mns}

We give a detailed description of penumbral moonshine at $D_0=-3$ and $D_0=-4$ in this section. We discuss the relevant groups---both the infinite ones described by lambencies, and the finite ones indexed by lambdencies---in \S~\ref{sec:mns-gps}. Then we describe the associated weakly holomorphic vector-valued modular forms in \S~\ref{sec:mns-fms}. Our main results appear in \S~\ref{sec:mns-mds}, wherein we prove that the forms of \S~\ref{sec:mns-fms} are realized as the McKay--Thompson series associated to modules for the finite groups that appear in \S~\ref{sec:mns-gps}. We also discuss some notable properties of these modules in \S~\ref{sec:mns-mds}, including a discriminant property, similar to that which appears in umbral moonshine. We discuss a connection between penumbral moonshine and distinguished lattices in \S~\ref{sec:mns-lts}.

\subsection{Groups}\label{sec:mns-gps}

As we have explained in \S~\ref{sec:int-pmo}, the cases of penumbral moonshine that we present in this work are parameterized by the lambdencies $\pd=(D_0,\ell)$ (\ref{eqn:int:lambdency}), where 
\begin{itemize}
\item
$\ell$ is a lambency (\ref{eqn:int:ell}) that represents a Fricke genus zero subgroup of $\SL_2(\RR)$ (cf.\ (\ref{eqn:pre-mod:Gamma0m+K})), and 
\item
$D_0$ is an $\ell$-admissible discriminant (see \S~\ref{sec:int-pfm}), and 
\item
either $D_0=-3$ or $D_0=-4$.
\end{itemize} 
There are finitely many Fricke genus zero lambencies, and they are all on display in Table \ref{tab:pmo-gps:FrickeGenusZero}. In this table the lambencies $\ell$ for which $D_0=-3$ is $\ell$-admissible are indicated with the symbol $\star$, and we use the symbol $\ddagger$ to indicate those for which $D_0=-4$ are $\ell$-admissible.

Recall from \S~\ref{sec:int-pmo} that we also restrict attention to lambencies $\ell=m+n,n',\dots$ with square-free level $m$. Inspecting Table \ref{tab:pmo-gps:FrickeGenusZero} we see that 
\begin{gather}
\begin{gathered}\label{eqn:pmo-gps:D0-3}
	1, \\
	3+3,\; 7+7,\; 13+13,\; 19+19,\; 31+31, \\
	21+3,7,21,\; 39+3,13,39,
\end{gathered}
\end{gather}
are the lambencies to be considered for $D_0=-3$,
while for $D_0=-4$ the relevant lambencies are 
\begin{gather}
\begin{gathered}\label{eqn:pmo-gps:D0-4}
	1, \\
	2+2,\; 5+5,\; 17+17,\; 29+29,\; 41+41, \\
	10+2,5,10,\;\; 26+2,13,26,\;\; 34+2,17,34.
\end{gathered}
\end{gather}
(We also see that, for $D_0=-3$, the only 
lambency that we neglect by restricting to square-free levels is 
$\ell=49+49$, while for $D_0=-4$ we neglect just $\ell=25+25$ and $\ell=50+2,25,50$.)

\begin{table}[ht]
\begin{footnotesize}
\begin{center} 
\caption{\small \label{tab:pmo-gps:FrickeGenusZero} 
{The lambencies $\ell$ 
that are Fricke and genus zero. 
}}
\begin{tabular}{ R@{\hspace{\tabcolsep}} R @{\hspace{\tabcolsep}} | R@{\hspace{\tabcolsep}} R@{\hspace{\tabcolsep}} | R @{\hspace{\tabcolsep}}R @{\hspace{\tabcolsep}}| R  }
\toprule
1  & \star~\ddagger & 17+17 & \ddagger &  31+31 & \star & 51+3,17,51 \\
2+2  & \ddagger & 18+2, 9,18 & & 32+32& & 54+2,27,54 \\
3+3 & \star & 18+18 & & 33+33 & & 55+5,11,55 \\
4+4 & & 19+19 & \star & 34+2,17,34 & \ddagger & 56+7,8,58 \\
5+5 & \ddagger & 20+4,5,20 & & 35+5,7,35 & & 59+59 \\
6+2,3,6 & & 20+20 & & 35+35 & & 60+3,4,12,15,20,60 \\
6+6 & & 21+3,7,21 & \star & 36+4,9,36 & & 60+4,15,60 \\
7+7 & \star &  21+21 & & 36+36 & & 62+2,31,62 \\
 8+8 & & 22+2,11,22 & & 38+2,19,38 & & 66+2,3,6,11,22,33,66 \\
9+9 & & 23+23 & &  39+3,13,39 & \star & 66+6,11,66 \\
10+2,5,10 & \ddagger & 24+3,8,24 & & 39+39 & & 69+3,23,69 \\
10+10 & & 24+24 & & 41+41 & \ddagger & 70+2,5,7,10,14,35,70 \\
11+11 & & 25+25 & \ddagger & 42+2,3,6,7,14,21,42 & & 71+71 \\
12+3,4,12 & & 26+2,13,26 & \ddagger & 42+3,14,42 & & 78+2,3, 6,13,26,39,78 \\
12+12 & & 26+26 & & 44+4,11,44 & & 87+3,29,87 \\
13+13 & \star~\ddagger & 27+27 & & 45+5,9,45 & & 92+4,23,92 \\
14+2,7,14 & & 28+4,7,28 & & 46+2,23,46 & & 94+2,47,94 \\
14+14 & & 29+29 & \ddagger & 47+47 & & 95+5,19,95 \\
15+3,5,15 & & 30+2,3,5,6,10,15,30 & & 49+49 & \star & 105+3,5,7,15,21,35,105 \\
15+15 & & 30+2,15,30 & & 50+2,25,50 & \ddagger & 110+2,5,10,11,22,55,110 \\
16+16 & & 30+5,6,30 & & 50+50 & & 119+7,17,119 \\
\bottomrule
\end{tabular}
\end{center}
\end{footnotesize}
\end{table}

Note from (\ref{eqn:pmo-gps:D0-3}--\ref{eqn:pmo-gps:D0-4}) 
that each lambency arising is of the form $\ell=m+n,n',\dots$, where $\{1,n,n',\dots\}$ is the entire set of exact divisors of $m$ (and this is even true when $m$ is not square-free).
It follows from this that $F^{(\pd)}$ as in (\ref{eqn:int-pfm:Fpdr_Fourier})---being a non-theta series
$D_0$-optimal (\ref{eqn:int:penumbral_optimality}) solution 
to (\ref{eqn:int:penumbral_lambency_condition}) for $\ell$---is uniquely determined by a choice of $C^{(\pd)}(D_0,r_0)$, and the condition that $C^{(\pd)}(1,1)=0$, for all the relevant lambdencies $\pd=(D_0,\ell)$, so long as we also impose the symmetry condition (\ref{eqn:int:franequalsfr}).
Thus the $F^{(\pd)}$ that we treat here furnish strong analogs of normalized principal moduli, according to the discussion of \S~\ref{sec:int-pmo}.

Having 
explained the content of Table \ref{tab:pmo-gps:FrickeGenusZero}, 
we now turn to Tables \ref{tab:pmo-gps:penumbralgroupsD0m3} and \ref{tab:pmo-gps:penumbralgroupsD0m4}. 
The main service these tables provide is the specification of the penumbral groups $G^{(\pd)}$, for $\pd=(D_0,\ell)$, for $D_0=-3$ in 
Table \ref{tab:pmo-gps:penumbralgroupsD0m3}, and for 
$D_0=-4$ 
in Table \ref{tab:pmo-gps:penumbralgroupsD0m4}.
In the top rows of these tables we just give the levels $m$ of the lambencies in question, rather than the full lambencies $\ell=m+n,n',\dots$, as the latter are uniquely determined by the former according to (\ref{eqn:pmo-gps:D0-3}--\ref{eqn:pmo-gps:D0-4}).

In Table \ref{tab:pmo-gps:penumbralgroupsD0m3} we write $\Th$ for the sporadic simple Thompson group, as studied in \cite{MR0399193,MR0409630}. (See also \cite{atlas}.) This is the only sporadic simple group that appears amongst the cases of penumbral moonshine that we present here, but the remaining penumbral groups do include some other exceptional examples. 
For instance, a group of the form $2.F_4(2).2$ appears in Table \ref{tab:pmo-gps:penumbralgroupsD0m4}. 
To understand what we mean by this 
write $F_4(2)$ for the Lie-type finite simple group of exceptional type $F_4$, defined over the field with $2$ elements, and let $F_4(2).2$ denote its automorphism group. 
(See \cite{atlas,MR1025610}.) 
The notation $2.F_4(2).2$ refers to a group $G$ that participates in a short exact sequence
\begin{gather}\label{eqn:pmo-gps:2F422seq}
	\mathds{1} \to \ZZ_2 \to G \xrightarrow{\pi} F_4(2).2 \to \mathds{1},
\end{gather}
and is not of the form 
$\ZZ_2\times F_4(2).2$.
Here, as in Tables \ref{tab:pmo-gps:penumbralgroupsD0m3} and \ref{tab:pmo-gps:penumbralgroupsD0m4}, we write $\ZZ_n$ for a cyclic group of order $n$, and $\mathds{1}$ is a shorthand for $\ZZ_1$.
Now there is actually more than one solution 
to (\ref{eqn:pmo-gps:2F422seq}). 
The group $G^{(-4,1)}$ is a 
solution 
for which the smallest order of a preimage under $\pi$ of an outer automorphism of $F_4(2)$ is $4$. 
We can also say that $G^{(-4,1)}$ is the unique 
solution 
to (\ref{eqn:pmo-gps:2F422seq}) that occurs as a subgroup of the monster. 

It is a 
similar but simpler story for $3.G_2(3)$, which appears in Table \ref{tab:pmo-gps:penumbralgroupsD0m3}. 
To explain this notation 
write $G_2(3)$ for the Lie-type finite simple group of exceptional type $G_2$, defined over the field with $3$ elements. 
(See \cite{atlas}.)
The notation $3.G_2(3)$ refers to the unique-up-to-isomorphism group $G$ that participates in a short exact sequence
\begin{gather}\label{eqn:pmo-gps:3G23seq}
	\mathds{1}\to \ZZ_3\to G\to G_2(3)\to \mathds{1},
\end{gather}
and is not of the form $\ZZ_3\times G_2(3)$.

\begin{table}
\begin{footnotesize}
\begin{center} 
\caption{\small \label{tab:pmo-gps:penumbralgroupsD0m3} {The penumbral groups,  and related structures, for $D_0=-3$. }}
\begin{tabular}{c | c c c c c c c c}
\toprule
$m$ & $1$ & $3$ & $7$ & $13$ & $19$ & $21$  & $31$ & $39$  \\\midrule
$G^{(\pd)}$ & $\Th$ & $3.G_2(3)$ & $L_2(7)$ & $\mathbb{Z}_3$ & $\mathds{1}$ & $\mathds{1}$ & $\mathds{1}$  & $\mathds{1}$  \\
$C(m)$& $\Th$ & $\mathbb{Z}_3\times G_2(3)$ & $\mathbb{Z}_7\times L_2(7)$ & $\mathbb{Z}_{13}\times\mathbb{Z}_3$ & $\mathbb{Z}_{19}$ & $\mathbb{Z}_{21}$ & $\mathbb{Z}_{31}$ & $\mathbb{Z}_{39}$ \\
$\Aut(\LL^{(\pd)})$& $\mathbb{Z}_2\times \Th$ & $\ZZ_2\times G_2(3)$ & $\mathbb{Z}_2\times L_2(7).2$ \\
\midrule
$C^{(\pd)}(0,0)$& $248$ & $28$ & $6$ & $1$ & $1$  & $0$ & $0$  & $1$ \\
$C^{(\pd)}(D_0,r_0)$ & $2$ & $4$ & $2$ & $1$ & $1$ & $1$ & $1$ &  $1$ \\
\bottomrule
\end{tabular}
\end{center}
\end{footnotesize}
\end{table}

To understand the notation $2.S_6.2$, which appears in Table \ref{tab:pmo-gps:penumbralgroupsD0m4}, recall that there is exactly one positive integer $n$ for which the symmetric group $S_n$ on $n$ letters admits an outer automorphism; namley, $n=6$. (See e.g.\ \cite{MR0338152,MR1400408}.) We write $S_6.2$ for the automorphism group of $S_6$, and note that the sporadic simple Mathieu group $M_{12}$ admits $S_6.2$ as a subgroup. (See e.g.\ \cite{atlas}.) We also note that there is a unique group $G$ that satisfies
\begin{gather}\label{eqn:pmo-gps:2M12seq}
	\mathds{1}\to \ZZ_2\to G
	\xrightarrow{\pi} 
	M_{12}
	\to 
	\mathds{1}
\end{gather}
and is not of the form $\ZZ_2\times M_{12}$. (This solution to (\ref{eqn:pmo-gps:2M12seq}) is commonly denoted $2.M_{12}$.) By $2.S_6.2$ we mean the 
preimage
under the $\pi$ in (\ref{eqn:pmo-gps:2M12seq}) of a subgroup $S_6.2$ in $M_{12}$. 

Two more notable groups appear in Table \ref{tab:pmo-gps:penumbralgroupsD0m4}. Namely, 
we have $^2F_4(2)$, denoting the first Ree group of type $^2F_4$ (cf.\ e.g.\ \cite{MR3100781}), and $^2F_4(2)'$, denoting the commutator subgroup of $^2F_4(2)$, which is simple, and also known as the Tits group \cite{MR164968}.

\begin{table}
\begin{footnotesize}
\begin{center} 
\caption{\small \label{tab:pmo-gps:penumbralgroupsD0m4} 
{The penumbral groups, and related structures, for $D_0=-4$. }
}
\begin{tabular}{c | c c c c c c}
\toprule
$m$ & $1$ & 
$5$ & 
$13$ & $17$ &
$29$ & $41$ \\\midrule
$G^{(\pd)}$ & $2.F_4(2).2$ & 
$2.S_6.2$  & 
$\mathbb{Z}_4$ & $\mathbb{Z}_2$  &
$\mathds{1}$ & $\mathds{1}$  \\
$C(m)$ & $2.F_4(2).2$ & 
$\mathbb{Z}_5\times 2.S_6.2$  & 
$\mathbb{Z}_{13}\times\mathbb{Z}_{4}$ & $\mathbb{Z}_{17}\times\mathbb{Z}_2$ &
 & \\
 $\Aut(\LL^{(\pd)})$&\\
 \midrule
$ C^{(\pd)}(0,0)$ & $-492$&$8$&$2$&$1$&$1$&$0$\\
$ C^{(\pd)}(D_0,r_0)$ & $2$&$1$&$1$&$1$&$1$&$1$\\
\bottomrule
\end{tabular}
\end{center}
\begin{center} 
\begin{tabular}{c | c c c c }
\toprule
$m$ & 
$2$ & 
$10$ & 
$26$ &
$34$
\\\midrule
$G^{(\pd)}$ & 
${^2}F_4(2)$ & 
$\ZZ_5:\ZZ_4$ & 
$\ZZ_{2}$ &
$\mathds{1}$
\\
$C(2m)$ & 
$\mathbb{Z}_4\times {^2}F_4(2)'$ & 
$\mathbb{Z}_{20}\times D_{10}$ & 
$\ZZ_{52}$ &
\\
$\Aut(\LL^{(\pd)})$&
$2.F_4(2)$
\\
 \midrule
$ C^{(\pd)}(0,0)$ & $52$	&$2$&$0$&$1$\\
$ C^{(\pd)}(D_0,r_0)$ & $2$	&$1$&$1$&$1$\\
\bottomrule
\end{tabular}
\end{center}
\end{footnotesize}
\end{table}

This explains all the notation for penumbral groups $G^{(\pd)}$ in Tables \ref{tab:pmo-gps:penumbralgroupsD0m3} and \ref{tab:pmo-gps:penumbralgroupsD0m4}, except for 
$L_2(7)$ and $L_2(7).2$, which are shorthands for $\PSL_2(\FF_7)$ and 
$
\Aut(\PSL_2(\FF_7))
\simeq 
\PGL_2(\FF_7)
$, 
respectively, 
and $D_{10}$, which denotes the dihedral group of order $10$. 

As we mentioned in \S~\ref{sec:int-pmo}, the groups $G^{(\pd)}$ for $\pd=(D_0,\ell)$ with $D_0=-3$ are closely related to centralizer subgroups of $G^{(-3,1)}=\Th$. More specifically, for each $m$, the relevant centralizer subgroup is the largest that occurs as 
the centralizer of $g$ in $G$, for $G=\Th$ and $g$ an element of order $m$ in $G$. This is the meaning of $C(m)$ in Table \ref{tab:pmo-gps:penumbralgroupsD0m3}. The reader will note, from Table \ref{tab:pmo-gps:penumbralgroupsD0m3}, that $G^{(\pd)}$ is isomorphic to $C(m)/\ZZ_m$, for all $m$ except for $m=3$, when $\pd=(-3,\ell)$ and $m$ is the level of $\ell$.

The meaning and significance of $C(m)$ in Table \ref{tab:pmo-gps:penumbralgroupsD0m4} is the same, except that $D_0=-4$, and $G^{(-4,1)}=2.F_4(2).2$ replaces $\Th$, and it does not work for $m=29$ or $m=41$ because there are no elements with these orders in $2.F_4(2).2$. Also, we need a slightly different formulation for the $m$ that are even, when $D_0=-4$. To explain this let $G$ be the unique normal subgroup of index $2$ in $G^{(-4,1)}=2.F_4(2).2$. Then, for $m$ the level of $\ell$ and $m$ even,
we should consider the largest subgroup of $G^{(-4,1)}$ that centralizes an element of order $2m$ in the complement of $G$ in $G^{(-4,1)}$. This is the meaning of $C(2m)$ in Table \ref{tab:pmo-gps:penumbralgroupsD0m4}. Unfortunately, this prescription also fails to make sense for all relevant $m$, because there is no element of order $68$ in $2.F_4(2).2$. 

We note here that the group $G$ of the previous paragraph, having index $2$ in $G^{(-4,1)}$, is the unique-up-to-isomorphism solution to
\begin{gather}\label{eqn:mns-gps:2F42seq}
	\mathds{1}\to \ZZ_2\to G \to F_4(2)\to \mathds{1}
\end{gather}
that is not of the form $\ZZ_2\times F_4(2)$. This is the group we have in mind when we write $2.F_4(2)$ (in Table \ref{tab:pmo-gps:penumbralgroupsD0m4}, and in \S~\ref{sec:mns-lts}).

For some values of $m$ the groups $G^{(\pd)}$ 
are closely related to automorphism groups of distinguished modular lattices $\LL^{(\pd)}$. These lattices are discussed in more detail in \S~\ref{sec:mns-lts}. Their automorphism groups are denoted $\Aut(\LL^{(\pd)})$, and are described, for the relevant lambdencies $\pd$, in Tables \ref{tab:pmo-gps:penumbralgroupsD0m3} and \ref{tab:pmo-gps:penumbralgroupsD0m4}.

We also specify the values of $C^{(\pd)}(D_0,r_0)$ and $C^{(\pd)}(0,0)$, for every penumbral lambdency $\pd$ with $D_0=-3$ or $D_0=-4$, in Tables \ref{tab:pmo-gps:penumbralgroupsD0m3} and \ref{tab:pmo-gps:penumbralgroupsD0m4}. As we have explained earlier in this section, the choice of $C^{(\pd)}(D_0,r_0)$ determines $F^{(\pd)}$ uniquely (at least for the lambdencies considered here), once we also require that $C^{(\pd)}(1,1)=0$. The value of $C^{(\pd)}(0,0)$ is given because it coincides with the rank of $\LL^{(\pd)}$, in all the cases that the latter object exists (except for $\pd=(-3,3+3)$, in which case $C^{(\pd)}(0,0)$ is twice the rank of the lattice we describe).

To conclude this section we
offer a few words as to how the $G^{(\pd)}$ were found. In short, we have relied upon the genus zero subgroups of $\SL_2(\RR)$ to guide us here.
Indeed, we have discussed in 
\S~\ref{sec:int-pfm} 
how it is that genus zero groups govern the forms $F^{(\pd)}$ that can serve as graded-dimension functions in penumbral moonshine. (And in \S~\ref{sec:int-pmo} we have noted the stronger sense that this is true, when $D_0=-3$ and $D_0=-4$.) Observe now that if $\varphi\in J^\wsh_{k,mM}(n)$ for some positive integers $k$, $m$, $M$ and $n$, then setting $\varphi'(\tau,z):=\varphi(M\tau,z)$ we have
\begin{gather}\label{eqn:mns-gps:mlttlnmap}
	C_{\varphi'}(D,r)=C_{\varphi}(D,r)
\end{gather}
at the level of Fourier coefficients.
So in particular, the map $\varphi(\tau,z)\mapsto \varphi(M\tau,z)$ of (\ref{eqn:pre-jac:tautostau}) 
preserves $D_0$-optimality (\ref{eqn:int:penumbral_optimality}). 
For this reason, if $M=p$ is 
a prime for which $D_0$ is a square modulo $4p$, 
then
the $n=1$ case of this map (\ref{eqn:pre-jac:tautostau}) sends $F^{(\pd')}$ to a candidate for the graded trace $F^{(\pd)}_g$ of an element $g\in G^{(\pd)}$ of order $p$, when $\pd=(D_0,\ell)$ for $m$ the level of $\ell$, and $\pd'=(D_0,\ell')$ for $mp$ the level of $\ell'$. 
So for example, the primes $p$ which appear as levels of Fricke genus zero lambencies $\ell'$ of level $p$, for which $D_0$ is $\ell'$-admissible, are primes we might expect to see arising as orders of elements of $G^{(\pd)}$, for $\pd=(D_0,\ell)$ and $\ell=1$. 
But this only controls the primes $p$ for which $D_0$ is a square modulo $4p$. As we explain in more detail in \cite{pmz}, the remaining primes $p$ can also be handled, but for such primes the obstruction to the appearance of $p$ in the prime spectrum of $G^{(\pd)}$, for $\ell=m+n,n',\dots$, is the positivity 
of the genus of a certain non-Fricke lambency $\ell'$, with level $mp$.

To summarize what we have just explained suppose that $\pd=(D_0,\ell)$ is a penumbral lambdency (cf.\ Table \ref{tab:pmo-gps:FrickeGenusZero}) and $m$ is the level of $\ell$. 
If $p$ is a prime such that $D_0$ is a square modulo $4p$ then $p$ can appear as the order of an element of $G^{(\pd)}$ if there is a Fricke genus zero lambency $\ell'$ of level $mp$,
whereas if $p$ is a prime such that $D_0$ is not a square modulo $4p$ then $p$ can appear as the order of an element of $G^{(\pd)}$ if there is a non-Fricke genus zero lambency $\ell'$ of level $mp$.

Before moving on we point out that the above prescription sometimes delivers more primes than we are able to use. For example, it perfectly predicts the prime spectrum of $G^{(\pd)}$ for $\pd=(-3,1)$, but we could not find a compelling candidate for $G^{(\pd)}$ with $\pd=(-4,1)$ that has elements of order $29$ and $41$ (cf.\ Tables \ref{tab:pmo-gps:FrickeGenusZero} and \ref{tab:pmo-gps:penumbralgroupsD0m4}). 
Also, the above prescription has not been formulated for forms with non-trivial characters, so it works best for $p>3$, because we have found empirically that non-trivial characters often occur for the McKay--Thompson series $F^{(\pd)}_g$ associated to elements $g\in G^{(\pd)}$, that have order not coprime to $6$. (Cf.\ \S~\ref{app:mts-chr}.)

\begin{table}[ht!]
\begin{footnotesize}
\begin{center} 
\caption{\small \label{tab:mns-fms:mltrln_m3} {
Multiplicative relationships for penumbral moonshine at $D_0=-3$.
}}
\begin{tabular}{c | r r }
\toprule
Relation & $[g]$ & $[h]$ \\ \midrule
\multirow{14}{*}{\shortstack{$F^{(-3,1)}_{g,0}(\tau) = \frac{1}{2}\left(F^{(-3,3+3)}_{h,0}(3\tau)-F^{(-3,3+3)}_{h,2}(3\tau) \right)$ \\ $F^{(-3,1)}_{g,1}(\tau) = \frac{1}{2}\left(F^{(-3,3+3)}_{h,3}(3\tau)-F^{(-3,3+3)}_{h,1}(3\tau) \right)$}} & 3A & 1A \\
& 6B & 2A \\
& 9A & 3B \\ 
& 9B & 3C \\ 
& 9C & 3D \\ 
& 12AB & 4A \\
& 18A & 6BD \\
& 18B & 6C \\
& 21A & 7A \\
& 24AB & 8A \\
& 27A & 9A \\
& 27BC & 9BC \\
& 36BC & 12C \\
& 39AB & 13A \\
\midrule
\multirow{4}{*}{\shortstack{$F^{(-3,1)}_{g,r}(\tau) = \frac{1}{2}\sum\limits_{\substack{s\xmod{14} \\ s\equiv r\xmod{2}}} F_{h,s}^{(-3,7+7)}(7\tau)$}} & 7A & 1A \\
& 14A & 2A\\ 
& 21A & 3A \\
& 28A & 4A \\\midrule
\multirow{2}{*}{\shortstack{$F^{(-3,1)}_{g,r}(\tau) = \sum\limits_{\substack{s\xmod{26} \\ s\equiv r\xmod{2}}}F_{h,s}^{(-3,13+13)}(13\tau)$}} 
& 13A & 1A \\
& 39AB & 3AB \\
[0.18em]\midrule
$F^{(-3,1)}_{g,r}(\tau) = \sum\limits_{\substack{s\xmod{38} \\ s\equiv r\xmod{2}}}F_{h,s}^{(-3,19+19)}(19\tau)$ & 
19A & 1A \\
[1.26em]\midrule
$F^{(-3,1)}_{g,r}(\tau) = \sum\limits_{\substack{s\xmod{42} \\ s\equiv r\xmod{2}}}I_s F_{h,s}^{(-3,21+3,7,21)}(21\tau)$   & 21A & 1A \\ 
$F^{(-3,3+3)}_{g,r}(\tau) = 2\sum\limits_{\substack{s\xmod{42} \\ s\equiv r\xmod{6}}} F_{h,s}^{(-3,21+3,7,21)}(7\tau)$ & 7A  & 1A \\ 
$F^{(-3,7+7)}_{g,r}(\tau) = 2\sum\limits_{\substack{s\xmod{42} \\ s\equiv r\xmod{24}}} I_s F_{h,s}^{(-3,21+3,7,21)}(3\tau)$ & 3A & 1A \\ 
[1.26em]\midrule
$F^{(-3,1)}_{g,r}(\tau) = \sum\limits_{\substack{s\xmod{62} \\ s\equiv r\xmod{2}}}F_{h,s}^{(-3,31+31)}(31\tau)$ & 31A & 1A \\
[1.26em]\midrule
$F^{(-3,1)}_{g,r}(\tau) = \sum\limits_{\substack{s\xmod{78} \\ s\equiv r\xmod{2}}}I_s F_{h,s}^{(-3,39+3,13,39)}(39\tau)$   & 39AB & 1A \\
$F^{(-3,3+3)}_{g,r}(\tau) = 2 \sum\limits_{\substack{s\xmod{78} \\ s\equiv r\xmod{6}}} F_{h,s}^{(-3,39+3,13,39)}(13\tau)$ & 13A & 1A \\
$F^{(13+13,-3)}_{g,r}(\tau) = \sum\limits_{\substack{s\xmod{78} \\ s\equiv r\xmod{26}}} I_s F_{h,s}^{(-3,39+3,13,39)}(3\tau)$ & 3A & 1A 
\\[1.26em]
\bottomrule
\end{tabular}
\end{center}
\end{footnotesize}
\end{table}

\begin{table}[!ht]
\begin{footnotesize}
\begin{center} 
\caption{\small 
\label{tab:mns-fms:mltrln_m4}
Multiplicative relations for penumbral moonshine at $D_0=-4$.}
\begin{tabular}{c | r r }
\toprule
Relation & $[g]$ & $[h]$ \\ \midrule
\multirow{15}{*}{\shortstack{$F^{(-4,1)}_{g,r}(\tau) = \sum\limits_{\substack{s\xmod 4\\s\equiv r\xmod 2}}(-1)^{{\frac s2+1}}F^{(-4,2+2)}_{h,s}(2\tau)$ }} & 4Q & 1A \\
& 4NR & 2A \\
& 4JKOS & 2B \\ 
& 12Q & 3A \\ 
& 8FGOP & 4AF \\ 
& 8HQ & 4BG \\
& 8DEJT & 4C \\
& 20D & 5A \\
& 12JKOPR & 6A \\
& 20E & 10A \\
& 24CDEF & 12A \\
& 52AB & 13A \\
& 8I & 4DE \\
& 16EF & 8D \\
& 24CDEF & 12BC \\\midrule
\multirow{14}{*}{ \shortstack{$F^{(-4,1)}_{g,r}(\tau) = \sum\limits_{\substack{s\xmod 10\\s\equiv r\xmod 2}}F^{(-4,5+5)}_{h,s}(5\tau)$ }} & 5A & 1A \\
& 10A & 2A \\
& 10D & 2B \\
& 10E & 2C \\
& 15A & 3A \\
& 30A & 6A \\
& 20ABC & 4ABC \\
& 10B & 2D \\
& 10C & 2E \\
& 30C & 6B \\
& 30B & 6C \\
& 20D & 4D \\
& 40ABCD & 8ABCD \\
& 20E & 4E \\\midrule
\multirow{2}{*}{ \shortstack{$F^{(-4,1)}_{g,r}(\tau) = \sum\limits_{\substack{s\xmod 20\\s\equiv r\xmod 2}}(-1)^{\frac{s}{2}+1}F^{(-4,10+2,5,10)}_{h,s}(10\tau)$ }} & 20D & 1A \\
& 20E & 2A 
\\[.18em] 
\multirow{2}{*}{ \shortstack{$F^{(-4,2+2)}_{g,r}(\tau) = \sum\limits_{\substack{s\xmod 20\\s\equiv r\xmod 4}}F^{(-4,10+2,5,10)}_{h,s}(5\tau)$ }} & 5A & 1A \\
& 10A & 2A \\[.18em]\midrule
\multirow{3}{*}{ \shortstack{$F^{(-4,1)}_{g,r}(\tau) = \sum\limits_{\substack{s\xmod 26\\s\equiv r\xmod 2}}F^{(-4,13+13)}_{h,s}(13\tau)$ }} & 13A & 1A \\
& 52AB & 4AB\\ 
& 26A & 2A \\\midrule
\multirow{2}{*}{ \shortstack{$F^{(-4,1)}_{g,r}(\tau) = \sum\limits_{\substack{s\xmod 34\\s\equiv r\xmod 2}}F^{(-4,17+17)}_{h,s}(17\tau)$ }} & 17A & 1A \\
& 34A & 2A \\[.18em] \midrule
\shortstack{$F^{(-4,1)}_{g,r}(\tau) = \sum\limits_{\substack{s\xmod 52\\s\equiv r\xmod 2}}(-1)^{\frac s2 +1}F^{(-4,26+2,13,26)}_{h,s}(26\tau)$}  & 52AB & 1A, 2A \\[.18em] 
 \shortstack{$F^{(-4,2+2)}_{g,r}(\tau) = \sum\limits_{\substack{s\xmod 20\\s\equiv r\xmod 4}}F^{(-4,26+2,13,26)}_{h,s}(13\tau)$ } & 13A & 1A \\[.18em]\bottomrule
\end{tabular}
\end{center}
\end{footnotesize}
\end{table}

\subsection{Forms}
\label{sec:mns-fms}

We now describe our predictions for the 
McKay--Thompson series $F^{(\pd)}_g=(F^{(\pd)}_{g,r})$ (cf.\ (\ref{eqn:int-pmo:Wpd}--\ref{eqn:int-pmo:Fpdgr})) 
of penumbral moonshine. For $D_0=-3$ we specify these forms as Rademacher sums, as follows. Take $\ell$ in (\ref{eqn:pmo-gps:D0-3}), and to begin assume that the level $m$ of $\ell$ is not composite. 
Then the corresponding vector-valued function $F^{(\pd)}$, 
for $\pd=(-3,\ell)$, may be described explicitly 
as 
\begin{align}\label{eqn:mns-fms:Radellnoncmp}
F^{(\uplambda)}(\tau) 
=C^{(\pd)}(D_0,r_0)R_1^{m,-3}(\tau)+b\theta_m^0(\tau),
\end{align}
where $R_1^{m,-3}$ is the level 1, index $m$ Rademacher sum defined in equation \eqref{rad-fourier}, and $b$ is chosen so that $C^{(\uplambda)}(1,1) = 0$. The constant $C^{(\pd)}(D_0,r_0)$ is chosen (see Table \ref{tab:pmo-gps:penumbralgroupsD0m3}) to be as small as possible such that the resulting $G^{(\uplambda)}$-module $W^{(\pd)}$ is well-defined in the sense of having integral multiplicities in its decomposition into irreducible representations. 

The case of composite level $m$ requires only a modest generalization of this. Namely, the Weil representation $\varrho_m$ is now reducible (cf.\ (\ref{eqn:pre-jac:Jwshkm_decomp})), and so we require our identity McKay--Thompson series to live inside a subrepresentation. This subrepresentation is exactly that which is specified, via the symmetry condition (\ref{eqn:int:franequalsfr}), by the lambency $\ell=m+n,n',\dots$ in question.
We may pass to this subrepresentation by acting on the result of Rademacher summation with an appropriate projection matrix $\Omega_m^{\alpha}$ (see (\ref{eqn:pre-jac:OmegamK})), so that in  general 
we have
\begin{align}\label{eqn:mns-fms:Radell}
F^{(\uplambda)}(\tau)=
C^{(\pd)}(D_0,r_0){R}_{1}^{(\uplambda)}(\tau),
\end{align}
where ${R}_1^{(\uplambda)}$ is now the normalized Rademacher sum defined in \eqref{normalized_rademacher}. 

We now describe Rademacher sum constructions for the McKay--Thompson series of penumbral moonshine at $D_0=-3$ attached to non-identity elements of 
the groups $G^{(\pd)}$. 
For every relevant lambency $\ell$ (see (\ref{eqn:pmo-gps:D0-3})), the putative graded trace function (\ref{eqn:int-pmo:Fpdgr}) can be written as 
\begin{align}\label{eqn:mns-fms:MTRad}
F_g^{(\uplambda)}(\tau) = M_g\left(C^{(\pd)}(D_0,r_0){R}^{(\pd)}_{n_g,\psi_g}(\tau)  + \sum_{k>1}\kappa_{k,g}\theta_m(k^2\tau)  \right),
\end{align}
where $\pd=(-3,\ell)$, and where $M_g$ is a matrix which should commute with $\varrho_m(\widetilde{\Gamma}_0(o(g)))$. (In most cases, $M_g$ is the identity matrix.) The various ingredients which go into this description (\ref{eqn:mns-fms:MTRad}) are presented in the tables in Tables \ref{tab:mts-chr:D0-3ell1}--\ref{tab:mts-chr:D0-3ell13+13}, and the tables of \S~\ref{app:mts-tht}.

The McKay--Thompson series we propose for penumbral moonshine at $D_0=-4$ may also be described in terms of Rademacher sums. In many instances an expression as in (\ref{eqn:mns-fms:MTRad}) suffices, but the general case requires a more complicated construction. Rather than taking that route we instead specify the $F^{(\pd)}_g$ with $\pd=(-4,\ell)$, and $\ell$ as in (\ref{eqn:pmo-gps:D0-4}), by explicitly describing
the poles at cusps of the corresponding scalar-valued forms $\breve F^{(\pd)}_g$ (in \S~\ref{app:mts-sng}), and providing enough coefficients of each component $F^{(\pd)}_{g,r}$ to pin down the contributions from theta series (in \S~\ref{app:mts-cff}). 
(This approach was more convenient, for the purpose of carrying out computer-aided computations using PARI \cite{PARI2}.)

With these prescriptions in place we note that the McKay--Thompson series of penumbral moonshine enjoy multiplicative relationships, similar to those that manifest for umbral moonshine (see \S~5.3 of \cite{MUM}).
These are coefficient-level relations between McKay--Thompson series with fixed $D_0$, but with different lambency
and 
level. For $D_0=-3$ we display these relationships in Table \ref{tab:mns-fms:mltrln_m3}. Look to Table \ref{tab:mns-fms:mltrln_m4} for the cases that $D_0=-4$.
In Table \ref{tab:mns-fms:mltrln_m3} we use a symbol $I_s$, defined by setting $I_s := 1$ if $s^2 \equiv 0\xmod{3}$, and $I_s:=-\frac{1}{2}$ otherwise.

Most of the relationships in Tables \ref{tab:mns-fms:mltrln_m3} and \ref{tab:mns-fms:mltrln_m4} are explained by the level-raising map (\ref{eqn:pre-jac:tautostau}) (see also (\ref{eqn:mns-gps:mlttlnmap})). The remaining ones, involving $I_s$, follow from exactly the same construction, but applied to forms with non-trivial characters. These multiplicative relationships between McKay--Thompson series are consistent with the close relationship between penumbral groups and centralizers that we discussed in \S~\ref{sec:mns-gps} (see Tables \ref{tab:pmo-gps:penumbralgroupsD0m3} and \ref{tab:pmo-gps:penumbralgroupsD0m4}).

Although we have formulated most of our discussion in terms of vector-valued modular forms, we remark that the functions $F^{(\pd)}_{g}$, for $\pd=(D_0,\ell)$ with $m$ the level of $\ell$, can equivalently be repackaged 
as weakly skew-holomorphic Jacobi forms $\varphi^{(\pd)}_g$ of weight $1$ and index $m$, with level $n=o(g)$, by setting
\begin{align}\label{eqn:mns-fms:varphipdg}
\varphi_g^{(\pd)}(\tau,z) 
:= \sum_{r\xmod{2m}} \overline{F^{(\pd)}_{g,r}(\tau)}\theta_{m,r}(\tau,z). 
\end{align}
Also, we adopt the notational convention of writing $C^{(\pd)}_g(D,r)$ as a shorthand for $C_{\vf}(D,r)$ when $\vf=\vf^{(\pd)}_g$. The prescriptions we have made here are such that
the $C^{(\pd)}_g(D,r)$ are all rational integers, and in particular real, so that we have
\begin{gather}\label{eqn:mns-fms:Fpdgr}
	F^{(\pd)}_{g,r}=\sum_{D\equiv r^2\xmod 4m}C^{(\pd)}_g(D,r)q^{\frac{D}{4m}}
\end{gather}
for the Fourier expansions of the $F^{(\pd)}_{g,r}$ in (\ref{eqn:mns-fms:varphipdg}) (cf.\ (\ref{eqn:pre-jac:Fou})).

\subsection{Modules}\label{sec:mns-mds}

With the previous section as preparation, we now formulate and prove our main results.
\begin{thm}\label{thm:mns-mds:D0-3}
For $\pd=(-3,\ell)$ with $\ell$ in (\ref{eqn:pmo-gps:D0-3}), and for $G^{(\pd)}$ the corresponding penumbral group (as described in \S~\ref{sec:mns-gps}, Table \ref{tab:pmo-gps:penumbralgroupsD0m3}),
there exists a virtual graded $G^{(\pd)}$-module $W^{(\pd)}$ (\ref{eqn:int-pmo:Wpd}) for which the associated McKay--Thompson series $F_g^{(\pd)}$ (\ref{eqn:int-pmo:Fpdgr}) are as specified by (\ref{eqn:mns-fms:MTRad}), and the data of Tables \ref{tab:mts-chr:D0-3ell1}--\ref{tab:mts-chr:D0-3ell13+13} and \ref{tab:mts-tht:D0-3ell1}--\ref{tab:mts-tht:D0-3ell13+13}.
\end{thm}

\begin{thm}\label{thm:mns-mds:D0-4}
For $\pd=(-4,\ell)$ with $\ell$ in (\ref{eqn:pmo-gps:D0-4}), and for $G^{(\pd)}$ the corresponding penumbral group (as described in \S~\ref{sec:mns-gps}, Table \ref{tab:pmo-gps:penumbralgroupsD0m4}),
there exists a virtual graded $G^{(\pd)}$-module $W^{(\pd)}$ (\ref{eqn:int-pmo:Wpd}) for which the associated McKay--Thompson series $F_g^{(\pd)}$ (\ref{eqn:int-pmo:Fpdgr}) are as specified by the data of Tables \ref{tab:mts-chr:D0-4ell1}--\ref{tab:mts-chr:D0-4ell26+021326}, \ref{tab:mts-sng:D0-4ell1}--\ref{tab:mts-sng:D0-4ell26+2,13,26}, and \ref{tab:mts-cff:D0-4ell1_1}--\ref{tab:mts-cff:D0-4ell_Gpdtrivial}.
\end{thm}

\begin{proof}[Proof of Theorems \ref{thm:mns-mds:D0-3} and \ref{thm:mns-mds:D0-4}]
Note that the $\ell=1$ case of Theorem \ref{thm:mns-mds:D0-3} has been established in \cite{MR3582425}. The remaining cases, of both theorems, may be proved by directly similar arguments. Specifically, for each lambdency $\pd$ (cf.\ (\ref{eqn:pmo-gps:D0-3}--\ref{eqn:pmo-gps:D0-4})), we require to verify the congruences amongst the coefficients of the $F^{(\pd)}_g$, for $g\in G^{(\pd)}$, that are given by Thompson's reformulation (see e.g.\ \cite{MR822245}) of Brauer's characterization of virtual characters. For example, we have to check that 
\begin{gather}\label{eqn:mns-mds:congmodp}
F^{(\pd)}_{g^p}\equiv F^{(\pd)}_g\xmod p
\end{gather}
for each $g\in G^{(\pd)}$, for each prime $p$ that divides $\#G^{(\pd)}$, and some more involved congruences are required, modulo powers of primes $p$ for which the $p$-valuation of $\#G^{(\pd)}$ is greater than $1$. 
Since the derivation of these congruences has appeared a number of times in recent literature, we suppress the details here, and instead refer to 
\S~4 of \cite{MR4230542}, for a detailed discussion 
in the case that $\pd=(-3,1)$, wherein $G^{(\pd)}=\Th$ is the Thompson group. 
(See \S~B.2 of \cite{MR822245} for a complete list of congruences arising in this case. See \cite{MR3539377} for a derivation of the congruences that apply in Mathieu moonshine (\ref{eqn:int:H2}).)

Note that no congruences are required in the case that $G^{(\pd)}$ is trivial. 
So for the remainder of this proof let $\pd=(D_0,\ell)$ be a penumbral lambdency for which $G^{(\pd)}$ is not trivial (cf.\ Tables \ref{tab:pmo-gps:penumbralgroupsD0m3} and \ref{tab:pmo-gps:penumbralgroupsD0m4}).
Also let $m$ be the level of $\ell$. 
Then to verify Thompson's congruences for $G^{(\pd)}$, we consider the scalar-valued modular forms 
\begin{gather}\label{eqn:mns-mod:breveFpdg}
\breve F^{(\pd)}_g(\tau):=\sum_{r\xmod2 m} F^{(\pd)}_{g,r}(4m\tau)
\end{gather} 
(cf.\ (\ref{eqn:pre-vvf:vector_to_scalar})), for $g\in G^{(\pd)}$.
Note that in principle it may happen that $C^{(\pd)}_g(D,r)\neq C^{(\pd)}_g(D,r')$, for $r\not \equiv r'\xmod 2m$, in which case information may be lost in the passage from the vector-valued $F^{(\pd)}_g$ to the scalar-valued $\breve F^{(\pd)}_g$. However, this does not occur for any of the lambencies in question (\ref{eqn:pmo-gps:D0-3}--\ref{eqn:pmo-gps:D0-4}), so, at least for the cases under consideration, it suffices to verify the required congruences (\ref{eqn:mns-mds:congmodp}), \&c., with $\breve F^{(\pd)}_g$ in place of $F^{(\pd)}_g$.

The next step is to find the minimal positive even integer $k$ for which there exists 
a cusp form 
${G}$ in the Kohnen plus space of weight $k-\frac12$ and level $4m$
such that 
$G$ (not to be confused with $G^{(\pd)}$) has integer Fourier coefficients, and satisfies
$G(\tau)=q^d+O(q^{d+1})$, where $d\geq |D_0|$. 
Then, for each $g\in G^{(\pd)}$, the function $f_g:=\breve F^{(\pd)}_gG$ is a modular form of weight $k$ for $\Gamma_0(4mN)$, where $N=nh$, for $n=o(g)$ and $h$ as given in the tables of \S~\ref{app:mts-chr}. 
Now it suffices to check (\ref{eqn:mns-mds:congmodp}), \&c., with $f_g$ in place of $F^{(\pd)}_g$, and since the $f_g$ are modular forms, we may apply Sturm's bound in order to do this. 
Specifically, taking $N_*$ to be the least common multiple of the $N$ arising in any one of the required congruences, the forms $f_g$ in that congruence all belong to the space of modular forms of weight $k$ (with trivial character) for $\Gamma_0(4mN_*)$, and Theorem 1 of \cite{MR894516} tells us that the congruence holds, if it holds up to $O(q^{B_*})$, for 
\begin{gather}\label{eqn:mns-mds:sturm}
	B_*=\frac{k}{12}\left[\SL_2(\ZZ):\Gamma_0(4mN_*)\right].
\end{gather}

So the final step we take is to find the largest value $B$ that occurs as the right-hand side of (\ref{eqn:mns-mds:sturm}), for every congruence arising (see Table \ref{tab:mns-mds:bounds}), and check all these congruences up to $O(q^B)$. 
We performed this check, at least up to $O(q^{2B})$ in each case, using PARI2 \cite{PARI2}.
\end{proof}

To supplement the above proof we give the 
relevant values of $k$ and $B$ in Table \ref{tab:mns-mds:bounds},
for each of the relevant lambdencies $\pd$. 
It turns out that, for each $\pd$ arising, the largest $N_*$ in (\ref{eqn:mns-mds:sturm}) is of the form $N=nh$ for some particular $g\in G^{(\pd)}$. So we specify 
these values of $N$, $n$ and $h$ too. 
(The conjugacy classes of corresponding elements $g\in G^{(\pd)}$ can be extracted from the tables in \S~\ref{app:mts-chr}.
For example, $n=o(g)=24$ and $h=12$ for $g$ in the classes $24CD$ of $G^{(-3,1)}=\Th$.)

\begin{table}[ht]
\begin{small}
\begin{center}
\caption{\small \label{tab:mns-mds:bounds} Sturm bounds for Theorems \ref{thm:mns-mds:D0-3} and \ref{thm:mns-mds:D0-4}.}
\begin{tabular}{cc|cccc|c}
\toprule
$D_0$&$\ell$&$k$&$N$&$n$&$h$&$B$\\
\midrule
$-3$	&$1$		&$10$&$288$&$24$&$12$	&$1920$\\
	&$3+3$	&$4$&$48$&$24$&$2$		&$384$\\
	&$7+7$	&$4$&$9$&$3$&$3$			&$192$\\
	&$13+13$	&$4$&$9$&$3$&$3$			&$336$\\
\midrule
$-4$	&$1$		&$16$&$320$&$40$&$8$		&$3072$\\
	&$2+2$	&$8$&$128$&$16$&$8$		&$1024$\\
	&$5+5$	&$6$&$80$&$20$&$4$		&$1440$\\
	&$13+13$&$4$&$16$&$4$&$4$		&$448$\\
	&$17+17$&$4$&$4$&$2$&$2$			&$144$\\
	&$10+2,5,10$	&$4$&$5$&$5$&$1$		&$120$\\
	&$26+2,13,26$	&$2$&$2$&$2$&$1$		&$56$\\
\bottomrule
\end{tabular}
\end{center}
\end{small}
\end{table}

We have formulated Theorems \ref{thm:mns-mds:D0-3} and \ref{thm:mns-mds:D0-4} in terms of virtual modules, but, as we alluded to in \S~\ref{sec:int-pmo}, this virtual-ness is of a special nature for $D_0=-3$. 
More specifically, for $\pd=(D_0,\ell)$, with $D_0=-3$ and $\ell$ any of the lambencies in (\ref{eqn:pmo-gps:D0-3}), it is the case that all the multiplicities of irreducible characters for $G^{(\pd)}$ in $W^{(\pd)}_{r,\frac{D}{4m}}$ are non-negative or non-positive, depending only on $r$, and whether or not $D$ is the minimal discriminant not less that $D_0=-3$ such that $D\equiv r^2\xmod 4m$. 
For example, in Thompson moonshine, where $\pd=(-3,1)$, we have that $W^{(\pd)}_{r,\frac{D}{4}}$ is $(-1)^r$ times an honest module if $D\geq 0$, and the remaining non-zero component $W^{(\pd)}_{1,\frac{-3}{4}}$ is $2$ times the trivial representation (cf.\ (\ref{eqn:int:F-31}--\ref{eqn:int:W-31})). 

We note here that for $\pd=(-3,1)$, we make the passage from $\breve W^{(\pd)}$ (see (\ref{eqn:int:W-31})) to $W^{(\pd)}$ (see (\ref{eqn:int-pmo:Wpd})) by taking $W^{(\pd)}_{r,\frac{D}{4}}$ to be $\breve W^{(\pd)}_{D}$ in case $D\equiv r^2\xmod 4$, and $0$ otherwise. Conversely, for any of the penumbral lambdencies $\pd$ (see (\ref{eqn:pmo-gps:D0-3}--\ref{eqn:pmo-gps:D0-4})), we obtain a virtual graded $G^{(\pd)}$-module $\breve W^{(\pd)}=\bigoplus_D\breve W^{(\pd)}_D$ from $W^{(\pd)}$ (\ref{eqn:int-pmo:Wpd}), for which the associated McKay--Thompson series are the scalar-valued forms $\breve F^{(\pd)}_g$ (cf.\ (\ref{eqn:int:F-31gW-31}), (\ref{eqn:mns-mod:breveFpdg})), by setting 
\begin{gather}\label{eqn:mns-mds:breveWpdfromWpd}
	\breve W^{(\pd)}_D :=\bigoplus_{\substack{r\xmod 2m\\D\equiv r^2\xmod 4m}}
	W^{(\pd)}_{r,\frac{D}{4m}}.
\end{gather}
Note though, that for all the lambdencies $\pd$ we consider, the $r$ such that $D\equiv r^2\xmod 4m$ are uniquely determined modulo $2m$. So in practice there is a just a single term in the summation in (\ref{eqn:mns-mds:breveWpdfromWpd}), for each of the lambdencies considered here.

The statement we have given above on non-virtual-ness fails in general for $D_0=-4$. For example, for $\pd=(-4,1)$, both positive and negative multiplicities arise in the decomposition of $W^{(\pd)}_{0,0}$ into irreducible modules for $G^{(\pd)}$ (cf.\ (\ref{eqn:int:F-41})). However, we see another feature here, in that the grading by $r$ predicts whether the corresponding (virtual) modules $W^{(\pd)}_{r,\frac{D}{4}}$ are faithful or not. Indeed, the irreducible constituents of the modules $W^{(\pd)}_{r,\frac{D}{4}}$ with $r=1$ are all faithful for $G^{(\pd)}$, whereas the modules $W^{(\pd)}_{r,\frac{D}{4}}$ with $r=0$ factor through to $F_4(2).2$ (cf.\ (\ref{eqn:pmo-gps:2F422seq})).

A directly similar statement holds for $\pd=(-3,3+3)$, where the group $G^{(\pd)}$ is $3.G_2(3)$ (cf.\ (\ref{eqn:pmo-gps:3G23seq})). Specifically, 
the decomposition of $W^{(\pd)}_{r,\frac{D}{12}}$ into irreducible representations of $G^{(\pd)}$ features only 
faithful
irreducible representations 
if $r$ is coprime to $3$,
while the representations with $r\equiv 0\xmod 3$ factor through to $G_2(3)$.

As we have commented in 
\S~\ref{sec:mns-lts}, one aspect of Thompson moonshine which was not emphasized in \cite{MR3521908} is the fact that the Thompson group is the automorphism group of a 248-dimensional unimodular lattice, 
which we label as $\Lambda^{(-3,1)}$. The associated complex representation appears as a particular graded component of the module $W^{(-3,1)}$. 
Specifically, we have
\begin{align}\label{eqn:mns-mds:Wpd00LLpd}
W^{(\pd)}_{0,0}\cong \Lambda^{(\pd)}\otimes \CC.
\end{align}
This property (\ref{eqn:mns-mds:Wpd00LLpd}) holds also for $\pd=(-3,7+7)$ and $\pd=(-4,2+2)$ (cf.\ \S~\ref{sec:mns-lts}), while for $\pd=(-3,3+3)$ we should take two copies of $\LL^{(\pd)}$ on the right-hand side.
Note that there is an obstruction to (\ref{eqn:mns-mds:Wpd00LLpd}) for $\pd=(-4,1)$, in that there is no non-trivial $G^{(\pd)}$-module of the correct dimension 
in this case (cf.\ (\ref{eqn:int:F-41})). That is, the left-hand side of (\ref{eqn:mns-mds:Wpd00LLpd}) is really virtual when $\pd=(-4,1)$. Also, there are four cases (see Tables \ref{tab:pmo-gps:penumbralgroupsD0m3} and \ref{tab:pmo-gps:penumbralgroupsD0m4}) where the left-hand side of (\ref{eqn:mns-mds:Wpd00LLpd}) is zero-dimensional.

A final 
feature 
that we emphasize here is 
a discriminant property similar to the one found in umbral moonshine (see \S~6.4 of \cite{MUM}). 
For certain lambdencies $\pd=(D_0,\ell)$, this property relates the number fields over which 
the $G^{(\pd)}$-modules $W^{(\pd)}_{r,\frac{D}{4m}}$ are defined, 
to the discriminants $D$ that index them. For example, for $\pd=(-3,1)$, whereby $G^{(\pd)}=\Th$ is the Thompson group, the $G^{(\pd)}$-module $W^{(\pd)}_{1,\frac{5}{4}}$ is ($-1$ times) the direct sum of a conjugate pair of irreducible modules of dimension $85995$ (cf.\ (\ref{eqn:int:F-31})), and the trace of an element of order $15$ on either of these modules takes values in $\QQ(\sqrt{-15})$ (but not in $\QQ$). 
This extends, with one consequence being that irrational numbers in $\QQ(\sqrt{-15})$ must be used, in order to write down the $G^{(\pd)}$-module $W^{(\pd)}_{r,\frac{D}{4}}$ explicitly using matrices, whenever $D=5f^2$ and $f$ is coprime to $3$.

A similar property holds for $\pd=(-4,1)$, where now $G^{(\pd)}$ takes the form $2.F_4(2).2$ (cf.\ (\ref{eqn:pmo-gps:2F422seq})), and $W^{(\pd)}_{1,\frac{D}{4}}$ involves the direct sum of a conjugate pair of irreducible (and faithful) modules of dimension $565760$, whenever $D=5f^2$ and $f$ is odd (cf.\ (\ref{eqn:int:F-41})). Furthermore, the trace of an element of order $40$ on either of these $565760$-dimensional modules takes values in $\QQ(\sqrt{-40})$ (but not in $\QQ$). For a general statement we offer the following.

\begin{con}\label{con:msn-mds:dscprp}
For $D_0$, $\ell$, $d$ and $D_1$ as in Table \ref{tab:mns-mds:dscprp} set $\pd=(D_0,\ell)$, suppose that $f$ is coprime to $D_0$, let $D=df^2$, and choose $r$ such that $D\equiv r^2\xmod 4m$, where $m$ is the level of $\ell$. Then there exists an element $g$ in $G^{(\pd)}$ of order $|D_1|$, and an irreducible constituent $M$ of $W^{(\pd)}_{r,\frac{D}{4m}}$ such that $\tr(g|M)$ is an irrational element of $\QQ(\sqrt{D_1})$.
\end{con}

\begin{table}[ht]
\begin{small}
\begin{center} 
\caption{\small \label{tab:mns-mds:dscprp} {Parameters for the discriminant property of penumbral moonshine. }}
\begin{tabular}{c c | c c}
\toprule
$D_0$ 	&$\ell$	&$d$&$D_1$ \\
\midrule
$-3$		&$1$		&$4$		&$-12$\\		
		&		&$5$		&$-15$\\
		&		&$8$		&$-24$\\
		&		&$9$		&$-27$\\
		&		&$13$	&$-39$\\
\midrule
$-3$		&$3+3$	&$4$		&$-3$, $-12$\\
		&		&$13$	&$-39$\\	
\midrule
$-4$		&$1$		&$4$		&$-16$\\
		&		&$8$		&$-32$\\
\midrule
$-4$		&$2+2$	&$4$		&$-4$, $-16$\\
\bottomrule
\end{tabular}
\qquad
\begin{tabular}{c c | c c}
\toprule
$D_0$ 	&$\ell$	&$d$&$D_1$ \\
\midrule
$-3$		&$13+13$	&$13$	&$-3$\\
\midrule
$-4$		&$1$		&$5$		&$-40$\\
		&		&$12$	&$-12$\\
\midrule
$-4$		&$2+2$	&$16$	&$-4$, $-16$\\
\midrule
$-4$		&$5+5$	&$5$		&$-8$\\
\midrule
$-4$		&$13+13$	&$13$		&$-4$\\
\bottomrule
\end{tabular}
\end{center}
\end{small}
\end{table}

At least some cases of Conjecture \ref{con:msn-mds:dscprp} may be proved by standard methods. 
For example, for $\pd=(-3,1)$, we may 
apply the argument of the proof of Theorem $2$ in Gannon's work \cite{MR3539377} to show that if $D>0$ and $M$ is an irreducible $G^{(\pd)}$-module that is isomorphic to its dual, then $M$ appears with even multiplicity in $W^{(\pd)}_{r,\frac{D}{4}}$. 
Then, after an inspection of the character table of $G^{(\pd)}=\Th$, the $d=5$ case of Conjecture \ref{con:msn-mds:dscprp}, for $\pd=(-3,1)$, comes down to a verification that the Fourier coefficient $C^{(\pd)}_g(D,r)$ (cf.\ (\ref{eqn:int-pfm:Fpdr_Fourier}) and (\ref{eqn:int-pmo:Fpdgr})) is odd, when $o(g)=15$ and $D=5f^2$, with $f$ coprime to $3$.
That statement follows if we can show that the Fourier coefficients of
\begin{gather}\label{eqn:msn-mds:dsc_D0-3ell1d5}
\breve F^{(\pd)}_g(\tau)+\theta(\tau)+\frac12\theta(5\tau)-\frac12\theta(45\tau)
\end{gather} 
are all even, when $o(g)=15$ and $\theta(\tau):=\sum_n q^{n^2}$. 
Similar to the proof of Theorems \ref{thm:mns-mds:D0-3} and \ref{thm:mns-mds:D0-4}, this may be verified by multiplying 
(\ref{eqn:msn-mds:dsc_D0-3ell1d5}) by a suitably chosen cusp form, and using Sturm's bound \cite{MR894516} to reduce to a finite check.
(These steps, except for the one due to Gannon, are essentially the same as those that are employed in \cite{MR3231314}, where the discriminant property of Mathieu moonshine is considered.) 

Our purpose in formulating the discriminant properties of penumbral moonshine as a conjecture is to emphasize that what we are really looking for is a constructive proof, that somehow incorporates the arithmetic of the number fields $\QQ(\sqrt{D_1})$. Indeed, we have similar hopes for the statements of Theorems \ref{thm:mns-mds:D0-3} and \ref{thm:mns-mds:D0-4}. Apart from the original vertex operator construction \cite{FLMPNAS,FLMBerk,FLM} of Frenkel, Lepowsky and Meurman (see \S~\ref{sec:int-mns}), and the vertex operator constructions that have proven relevant to 
umbral moonshine (see \cite{MR3922534,MR3995918,MR3649360,MR3465528,MR3859972}), the discriminant property of Conjecture \ref{con:msn-mds:dscprp} is perhaps the most powerful hint we have so far, as to the form that a constructive 
proof of Theorems \ref{thm:mns-mds:D0-3} and \ref{thm:mns-mds:D0-4} might 
take.

The reader will note that we have divided Table \ref{tab:mns-mds:dscprp} into two parts, the cases of the first part being just those for which 
\begin{gather}\label{eqn:mns-mds:D1dD0}
D_1=dD_0
\end{gather}
is an option. 
With one fewer free parameters, these cases are perhaps the more compelling. 
The cases of the second part stray from this, to a lesser or greater extent. For example, there is just a factor of $2$ separating $D_1=40$ and $dD_0=20$, when $\pd=(-4,1)$ and $d=5$. On the other hand, the factor $d$ completely disappears from $D_1$ when $\ell=5+5$ or $\ell=13+13$. Is there a natural modification of (\ref{eqn:mns-mds:D1dD0}), that incorporates the level of $\ell$?

Note also that all the values of $D_1$ in Table \ref{tab:mns-mds:dscprp} are discriminants. That is, they are congruent to $0$ or $1$ modulo $4$. For $\pd=(-3,1)$, the cases that $d$ is $5$, $8$ or $13$ are further distinguished in that the corresponding discriminants 
are fundamental (see (\ref{eqn:int-pfm:D0fund})). These cases of the discriminant property were discussed already in \cite{MR3521908}. (See Conjectures 4.3 and 4.4 of op.\ cit.)

\subsection{Lattices}\label{sec:mns-lts}

Historically, special lattices have played a significant role in 
moonshine. Examples include the role of the Leech lattice and the $E_8$ root lattice
in monstrous moonshine and Conway moonshine \cite{Dun_VACo,MR3376736}, respectively, and the classification of cases of umbral moonshine by Niemeier lattices (cf.\ (\ref{eqn:Gell})). Preliminary evidence suggests that lattices also play some role in penumbral moonshine, although the full connection remains to be elucidated. We collect here
some relevant facts about lattices and their automorphism groups and indicate some important structures that we expect to be relevant for penumbral moonshine.

Recall that a {\em lattice} $L$ is a free $\ZZ$-module equipped with a symmetric bilinear form $( \cdot, \cdot ):L\times L\to\RR$. A lattice is {\em positive-definite} if this form
induces a positive-definite inner product on the vector space $L_\RR = L \otimes_\ZZ \RR$. A lattice is {\em integral} if
$( \lambda, \mu ) \in \ZZ$ for all $\lambda, \mu \in L$, and it is {\em even} if $(\lambda, \lambda )\in 2 \ZZ$ for all $\lambda \in L$. The {\em dual lattice} $L^\ast$
is defined as
\begin{gather}
L^\ast = \{ \lambda^\ast \in L_\RR \mid (\lambda^\ast, \lambda ) \in \ZZ \text{ for all } \lambda \in L \}.
\end{gather}
If $L$ has rank $n$ then by choosing a basis $\lambda_1, \dots,\lambda_n$ we can express elements of  $L$ as 
\begin{gather}
\lambda = \sum_{i=1}^n \xi_i \lambda_i, \xi_i \in \ZZ.
\end{gather}
We then uniquely determine a matrix $A=(A_{ij})$, called the {\em Gram matrix} of $L$, by writing the norm of a lattice vector as
\begin{gather}
( \lambda, \lambda ) = \sum_{i,j} \xi_i A_{ij} \xi_j.
\end{gather}

The automorphism group of a lattice $\mathrm{Aut}(L)$ is defined as the set of elements $g \in \textsl{GL}_n(\QQ)$ such that $L g = L$ and $g A g^{\mathrm{t}} = A$. (Here the superscript in $g^{\rm t}$ denotes transposition of matrices.) A lattice $L$ is said to be {\it k-modular} if there is a matrix $T  \in \textsl{GL}_n(\QQ)$
with $L= L^* T$ and $ T A T^{\mathrm{t}}= k A$, and a $1$-modular lattice is usually called {\em unimodular}.  We will see below that the lattices $\Lambda^{(\uplambda)}$ that we have been
able to associate to examples of penumbral moonshine are all $m$-modular, where $m$ is the level of 
$\ell$, when $\pd=(D_0,\ell)$.

Special lattices of particular relevance to penumbral moonshine are picked out by the condition of {global irreducibility} that
was developed in \cite{MR1369418} following work of Thompson \cite{MR0399193}. 
Global irreducibility is the generalization of a property which we will follow \cite{MR1732353} 
in referring to as {Thompson's condition}. We define this simpler notion first. Let $G$ be a finite group and $L$ a torsion free $\mathbb{Z}[G]$-module of finite rank. We say that the pair $(G,L)$ satisfies {\em Thompson's condition} if $L/pL$ is irreducible as an $\mathbb{F}_p[G]$-module, for every prime $p$. We note that for such a pair, if $(\cdot,\cdot)$ is a $G$-invariant integral symmetric bilinear form defined on $L$, then the lattice $(L,(\cdot,\cdot))$ is $k$-modular for some integer $k$ 
\cite{MR399257},
and so we immediately connect back to the $k$-modularity foreshadowed in the previous paragraph.

Pairs $(G, L)$ obeying Thompson's condition are very rare. The three best known examples are
$(W_{E_8}, \Gamma_8)$ with $\Gamma_8$ the $E_8$ root lattice and $W(E_8)$ its associated Weyl group, $(2.\Co_1,\Lambda)$ with $\Lambda$ the Leech lattice and $\Co_1$ the first Conway group \cite{MR0237634,MR0248216}, and $(\ZZ_2\times \Th, \Lambda_{248})$ with $\Lambda_{248}$ the rank $248$ unimodular Thompson--Smith lattice and $\Th$ the sporadic simple Thompson group. In light of Thompson moonshine \cite{MR3521908}, we label the third of these as $\Lambda^{(-3,1)} := \Lambda_{248}$ and associate it to the $\uplambda=(-3,1)$ case of penumbral moonshine. Note in particular that the unimodularity (i.e.\ the 1-modularity) of $\Lambda^{(-3,1)}$ matches the index of the weakly skew-holomorphic Jacobi forms which carry Thompson moonshine, or equivalently, the level of the Fricke genus zero 
lambency $\ell=1$.

To find lattices associated to other cases of penumbral moonshine it is necessary to consider the more general notion of global irreducibility, as formulated
by Gross \cite{MR1071117}. 
To define this, again let $G$ be a finite group, and let $V$ be a finite-dimensional $\QQ[G]$-module such that the $\RR[G]$-module $V\otimes_\QQ\RR$ is irreducible. Write $K=\mathrm{End}_G(V)$ for the algebra consisting of the $\alpha \in \mathrm{End}_\QQ(V)$ which satisfy $g\alpha = \alpha g$ for all $g\in G$. (In  the previous paragraph, we took $V = L\otimes_\ZZ\QQ$, in which case $K = \QQ$.) Let $R$ be a maximal order in $K$ and $\Lambda$ an $R[G]$-lattice in $V$, i.e.\ a free $\mathbb{Z}$-submodule of rank $n=\mathrm{dim}_\QQ(V)$ inside $V$ which is mapped by $R$ and $G$ to itself. Following op.\ cit.\ we say that $V$ is {\em globally irreducible} if $V_{\mathfrak{p}}:= \Lambda/\mathfrak{p}\Lambda$ is an irreducible representation of $G$ over $R/\mathfrak{p}$, for every maximal ideal $\mathfrak{p}$.

Groups and lattices obeying this more general
condition of global irreducibility have been widely studied. See e.g.\ \cite{MR1071117, MR1308713, MR1432536,MR1732353}.
An example of such a lattice with $K= \QQ[\sqrt{-7}]$ is described in \cite{MR1071117} and in \cite{MR1722413}. The group $\PSL_2(\FF_7)$ (denoted $L_2(7)$ in Table \ref{tab:pmo-gps:penumbralgroupsD0m3}) has a $3$-dimensional irreducible complex representation $\psi$ defined over $K$, and $\psi \oplus \bar \psi$ is a globally irreducible representation of dimension $6$. The associated lattice is a rank $6$, $7$-modular lattice with automorphism group $\ZZ_2 \times \PGL_2(\FF_7)$ which arises from the theta polarization of the Klein quartic, the genus $3$ Riemann surface with maximal automorphism group $\PSL_2(\FF_7)$. We denote this lattice as $\Lambda^{(-3,7+7)}$ and associate it to penumbral moonshine at $\uplambda = (-3,7+7)$, which, according to Theorem \ref{thm:mns-mds:D0-3}, furnishes moonshine for the group $G^{(-3,7+7)}=\PSL_2(\FF_7)$. We again emphasize that the 7-modularity of $\Lambda^{(-3,7+7)}$ matches nicely onto the fact that the 
level of the lambency $\ell=7+7$ in question is $7$.

A number of globally irreducible pairs of lattices and groups can be constructed using Lie theory \cite{MR0399193,MR1369418}. See \cite{MR1308713} for detailed constructions.  Every complex simple Lie algebra ${\cal L}$, viewed as a vector space, can be decomposed into a direct sum of Cartan subalgebras
\begin{gather}
{\cal L} = \bigoplus_{i=0}^h {\cal H}_i
\end{gather}
with $h$ the Coxeter number. The Killing form is non-degenerate on ${\cal L}$ and this can be used to equip  a $\ZZ$-module $\Lambda \subset {\cal L}$ 
with a symmetric bilinear form. In the case that ${\cal L}$ is the Lie algebra of $E_8$, Thompson and Smith used this idea to construct the Thompson--Smith lattice $\Lambda^{(-3,1)}$ (see \cite{MR0409630}). An important ingredient in their construction was the fact that the $248$-dimensional representation of $\Th$ is globally irreducible.  Related constructions for other Lie algebras can be
found in \cite{MR1369418, MR1308713}. One other example is of particular interest for penumbral moonshine. Taking ${\cal L}$ to the $14$-dimensional Lie algebra of $G_2$ leads to a rank $14$, $3$-modular lattice with automorphism group $\ZZ_2\times G_2(3)$ (see \cite{MR2997018} for a construction), which acts irreducibly via its $14$-dimensional representation. We identify this lattice with the lattice $\Lambda^{(-3,3+3)}$ corresponding to the $\pd=(-3,3+3)$ case of penumbral moonshine. 
As we've seen, the group at this lambency 
takes the form $3.G_2(3)$ (cf.\ (\ref{eqn:pmo-gps:3G23seq})).

Taking instead the Lie algebra of $F_4$ gives rise to a rank $52$, $2$-modular lattice with automorphism group $2.F_4(2)$ (cf.\ (\ref{eqn:mns-gps:2F42seq}) and see \cite{MR1830550}). The Tits group ${}^2 F_4(2)'$ is a subgroup of the automorphism group and acts on the lattice via a globally irreducible representation of dimension $26$ over $\QQ[\sqrt{-2}]$. We associate this lattice
with the lattice $\Lambda^{(-4,2+2)}$ of penumbral moonshine, and find that we can take the penumbral group to be the larger group $G^{(-4,2+2)}= {}^2 F_4(2)$, which contains ${}^2 F_4(2)'$ as its derived subgroup with index 2. 

Further information
about the lattices we have discussed can be found in the online database \cite{catalogueoflattices}. At the time of writing, there are several cases of penumbral moonshine to which we have not been able to attach a corresponding lattice. 
For example, we don't have a candidate in case $\pd=(4,-1)$, where the corresponding representation of $G^{(\pd)}$ is virtual (cf.\ (\ref{eqn:int:F-41})). 
There is also the curious fact of a $5$-modular lattice of rank $8$ 
(called $Q8(1)$ in \cite{catalogueoflattices}, cf.\ also Example 6.2 of \cite{MR1432536}). Although the rank is right for $\pd=(-4,5+5)$, according to (\ref{eqn:mns-mds:Wpd00LLpd}) and Table \ref{tab:pmo-gps:penumbralgroupsD0m4}, the automorphism group is rather far from the penumbral group $2.S_6.2$ (cf.\ (\ref{eqn:pmo-gps:2M12seq})) at this lambdency.
Clearly it is desirable to develop a more systematic understanding of the
role that modular lattices play in penumbral moonshine.

\section{Summary and Outlook}\label{sec:sum}

We have presented an overview of penumbral moonshine, including a concrete description of the corresponding McKay--Thompson series for $D_0=-3$ and $D_0=-4$. Moreover, we have proven that these distinguished modular forms arise from virtual graded modules. 
In many
ways the structure of penumbral moonshine parallels the structure of umbral moonshine, but in such a way that mock modular forms do not appear. 
However, many elements of penumbral moonshine need to be better understood. 

The most important open problem is the explicit construction of the 
penumbral moonshine modules that we have discussed. Do CFTs and VOAs also play a role here, or are some other structures required?
In the case of Mathieu moonshine, the Conway moonshine module serves as a ``fake'' moonshine VOA---in the sense that it is close, but not quite right (see \cite{MR3465528})---and this leads us to concrete constructions of moonshine modules, at least for some cases of umbral moonshine. Are there analogous ``fake'' theories for any of the cases of penumbral moonshine?

Also, it would be interesting to see examples beyond those we have given at $D_0=-3$ and $D_0=-4$. To what extent does penumbral moonshine occur at other lambdencies? We comment briefly on VOA constructions which yield 
graded representations of finite groups at lambdencies with $D_0<-4$ in \S~\ref{app:non}. However, these examples appear to be qualitatively different from penumbral moonshine at $D_0=-3$ and $D_0=-4$. 

An important yet-to-be-realized development is the 
determination of the algebraic objects that are meant to be organizing the instances of penumbral moonshine, as the Niemeier lattices do (\ref{eqn:Gell}) for umbral moonshine. We have provided preliminary evidence that modular lattices might be involved, but they do not appear to be involved in all cases.

Finally, we acknowledge the possibility that 
it might be possible to improve upon some of the specifications we have made here. The rapid growth in the coefficients of the forms involved (see the tables in \S~\ref{app:mts-cff}), the existence of theta series ambiguities in these forms (cf.\ (\ref{eqn:int-pfm:frthetamr0})), and the absence of an organizing family of algebraic objects (cf.\ \S~\ref{sec:mns-lts}), are three factors which make the pursuit of penumbral moonshine especially challenging, and especially fascinating.

\appendix

\section{Some Non-Examples}\label{app:non}

Relationships between modular forms and finite groups are a relatively common occurrence. For example, one expects such a connection any time there is a (suitably nice) vertex operator algebra with a non-trivial group of automorphisms. On the other hand, phenomena which earn the name ``moonshine'' are special, distinguished by properties which set them apart from the generic situation. For example, as we have stressed in the introduction, \S~\ref{sec:int}, monstrous moonshine is distinguished by its genus zero property, and genus zero groups also play a central role in characterizing the possible occurrences of  umbral and penumbral moonshine. Another feature which further privileges umbral and penumbral moonshine is their discriminant property, discussed in \S~\ref{sec:mns-mds}.

In generalizing penumbral moonshine to lambencies beyond $D_0=-3$ and $D_0=-4$, we can expect to come across examples of weakly skew-holomorphic Jacobi forms which are related to finite groups, but which should not be thought of as belonging to the penumbral family. Without attempting to be complete, it is the purpose of this appendix to illustrate this in a few familiar examples, which all occur at $D_0<-4$.

\subsection*{Monstrous moonshine at $D_0=-23$} 

We will start our discussion at the lambdency $\uplambda = (-23,6+6)$. It is possible to show that the 
vector-valued modular form 
\begin{align}\label{skew-holomorphic monster function}
f(\tau) = (0,J(\tau)\eta(\tau),0,0,0,-J(\tau)\eta(\tau),0,-J(\tau)\eta(\tau),0,0,0,J(\tau)\eta(\tau))
\end{align}
satisfies the conditions (\ref{eqn:int:penumbral_optimality}) and (\ref{eqn:int:penumbral_lambency_condition}), where $J(\tau) = q^{-1}+0+196884q + \cdots$ is the elliptic modular invariant (\ref{eqn:int-mns:ellmodinv}), and $\eta(\tau):=q^{\frac{1}{24}}\prod_{n>0}(1-q^n)$ is the Dedekind eta function. 
From the standard connection between $J$ and the monster group $\mathbb{M}$ (see \S~\ref{sec:int-mns}), we expect this function to exhibit monstrous moonshine. However, the additional factor of the eta function alters the standard interpretation slightly, as follows. The monster CFT $V^\natural$ contains a Virasoro subalgebra at central charge $c=24$, and therefore its partition function decomposes into characters $\chi^{(c)}_h(\tau)$ of the Virasoro algebra, 
\begin{align}
J(\tau) = \sum_{h\geq 0}n(h) \chi^{(24)}_h(\tau),
\end{align}
where $n(h)$ is the number of Virasoro primaries of conformal dimension $h$, and 
\begin{align}\label{virasoro characters}
\chi^{(24)}_h(\tau) = \begin{cases}
\frac{(1-q)q^{-\frac{23}{24}}}{\eta(\tau)}, & h=0\\
\frac{q^{h-\frac{23}{24}}}{\eta(\tau)}, & h>0 .
\end{cases}
\end{align}
The Virasoro primaries at a given conformal dimension transform into one another under the action of the automorphism group, so $n(h)$ will still decompose into dimensions of irreducible representations of the monster. Moreover, from the explicit expressions for the Virasoro characters in \eqref{virasoro characters}, we see that any of the non-zero components of $f(\tau)$ can (almost) serve as the generating function for the numbers $n(h)$, 
\begin{align}
J(\tau)\eta(\tau) = -q^{\frac{1}{24}}+ q^{-\frac{23}{24}}\sum_{h\geq 0}n(h)q^{h},
\end{align}
up to the term $-q^{\frac{1}{24}}$ appearing on the right hand side. Similarly, the functions 
\begin{align}
T_g(\tau)\eta(\tau)=-q^{\frac{1}{24}}+q^{-\frac{23}{24}}\sum_{h\geq 0} 
\tr(g|\textsl{VP}^\natural_h)
q^{h},
\end{align}
with $T_g(\tau)$ the McKay-Thompson series of the moonshine module $V^\natural$, can be embedded into weakly skew-holomorphic Jacobi forms of higher level and thought of as encoding the graded characters of an auxiliary $\mathbb{M}$-module, $\textsl{VP}^\natural = \bigoplus_{h\geq 0} \textsl{VP}_h^\natural$, where $\textsl{VP}_h^\natural$ is the space of Virasoro primaries of conformal dimension $h$ in $V^\natural$.

The additional term $-q^{\frac{1}{24}}$ present in all these functions contributes, what would be seen by someone unaware of the monster CFT, a mild virtuality of the $\mathbb{M}$-module (specifically, a trivial representation would come in with a minus sign at the grading $\frac{1}{24}$). The $\mathbb{M}$-module $\textsl{VP}^\natural$ one attaches to this function has otherwise positive multiplicities when decomposing into $\mathbb{M}$-irreps, and the ``true'' $\mathbb{M}$-module $V^\natural$ is of course completely positive.

\subsection*{Baby monster CFT at $D_0=-15$}

We can obtain another example in the same spirit by going to $D_0=-15$ and $m=4$. We construct the function
\begin{align}
f(\tau) = \Omega (R^{(4,-15,1)}_{1}(\tau)-16\theta^0_4(\tau))
\end{align}
where $R^{(4,-15,1)}_1(\tau)$ is the Rademacher sum of equation \eqref{rad-fourier}, and
\begin{equation}
    \Omega=\frac12(\Omega_4(1)-\Omega_4(2)+\Omega_4(4))
\end{equation}
is a projector onto an irreducible subrepresentation of the $m=4$ Weil representation (cf.\ (\ref{eqn:pre-jac:ztosz}), (\ref{eqn:pre-jac:Jwshkm_decomp})). 
We find that this function takes the shape 
\begin{align}
f(\tau) =  (f_0(\tau),f_1(\tau),0,f_3(\tau),-f_0(\tau),f_3(\tau),0,f_1(\tau))
\end{align}
with the non-trivial components given (up to a Dedekind eta function factor) by the characters $\chi_{\textsl{V}\mathbb{B}^\natural(\alpha)}(\tau)$ of the baby monster CFT
\cite{Hoehn2007}
\begin{align}
\begin{split}
f_0(\tau) = \chi_{\textsl{V}\mathbb{B}^\natural(2)}(\tau)\eta(\tau)&= 962566q +10506240q^2+410132480q^3+\cdots \\
f_1(\tau) = \chi_{\textsl{V}\mathbb{B}^\natural(0)}(\tau)\eta(\tau)&=q^{-\frac{15}{16}}(1-q+96255q^2+9550635q^3+\cdots) \\
f_3(\tau) = -\chi_{\textsl{V}\mathbb{B}^\natural(1)}(\tau)\eta(\tau)& = -q^{\frac{9}{16}}(4371+1139374q+63532485q^2+\cdots).
\end{split}
\end{align}
The factor of $\eta(\tau)$ requires an interpretation analogous to the one we provided in the context of the monster, but otherwise, we expect (and of course, find) that this weakly skew-holomorphic Jacobi form is related to the representation theory of $2.\mathbb{B}$, the Schur cover of the inner automorphism group of the baby monster VOA $\textsl{V}\mathbb{B}^\natural$ introduced by H\"ohn \cite{Hoehn2007,MR2734692}.

The baby monster VOA can be described as the commutant of a $c=\frac12$ Virasoro subalgebra of the moonshine module. It is possible to construct other commutant subVOAs of $V^\natural$ in a similar manner. One can ask if these more general models have relationships to weakly skew-holomorphic Jacobi forms, like the connection we established for the baby monster VOA. We have found at least two other examples, though we do not provide the details. 
One occurs at 
$D_0=-111$, 
with $m=30$, 
and is related to the Fischer VOA \cite{MR2890302}, whose modules transform with respect to the triple cover $3.\textsl{Fi}_{24}'$ of its inner automorphism group. 
The other occurs at 
$D_0=-44$, 
with $m=12$, 
and is related to a VOA referred to as $\widetilde{\mathcal{W}}_{\mathbb{Z}_{4\mathrm{A}}}$ in \cite{bae2020conformal}, which can be obtained as the charge conjugation orbifold of a lattice VOA attached to a particular 23-dimensional sublattice of the Leech lattice, and whose modules transform under representations of a quadruple cover $4.2^{22}.\textsl{Co}_3$ of its inner automorphism group. See \cite{bae2020conformal} for a further discussion of these models and their characters.

\section{McKay--Thompson Series}\label{app:mts}

In this section we present data that characterizes the McKay--Thompson series discussed in \S~\ref{sec:mns-fms}. 

In order to present this data we need names for the conjugacy classes of the penumbral groups $G^{(\pd)}$ arising. We adopt the conventions of \cite{GAP4} for this purpose. Specifically, for each of the lambdencies $\pd$ we study (see (\ref{eqn:pmo-gps:D0-3}--\ref{eqn:pmo-gps:D0-4})) such that the corresponding group $G^{(\pd)}$ is non-cyclic (see Tables \ref{tab:pmo-gps:penumbralgroupsD0m3} and \ref{tab:pmo-gps:penumbralgroupsD0m4}), we use the class names, and character tables, that can be accessed in \cite{GAP4} by using the character table names given in Table \ref{tab:mts:chrtblnms}.
\begin{table}[ht]
\begin{small}
\begin{center}
\caption{\small \label{tab:mts:chrtblnms} Character table names.}
\begin{tabular}{cc|c|c}
\toprule
$D_0$&$\ell$&$G^{(\pd)}$&Table name\\
\midrule
$-3$	&$1$		&$\Th$&{\tt "Th"}\\
	&$3+3$	&$3.G_2(3)$&{\tt "3.G2(3)"}\\
	&$7+7$	&$L_2(7)$&{\tt "L3(2)"}\\
\midrule
$-4$	&$1$		&$2.F_4(2).2$&{\tt "Isoclinic(2.F4(2).2)"}\\
	&$2+2$	&$^2F_4(2)$&{\tt "2F4(2)'.2"}\\
	&$5+5$	&$2.S_6.2$&{\tt "2.M12M4"}\\
	&$10+2,5,10$	&$\ZZ_5:\ZZ_4$&{\tt "5:4"}\\
\bottomrule
\end{tabular}
\end{center}
\end{small}
\end{table}

For the cases that $G^{(\pd)}$ is cyclic it suffices to say that $nA$, $nB$, \&c., denote conjugacy classes of elements of order $n$. 

The data we present is as follows. We specify the levels and characters of the McKay--Thompson series $F^{(\pd)}_g$ (cf.\ (\ref{eqn:int-pmo:Fpdgr})) for all $\pd$ arising (cf.\ \S~\ref{sec:mns-gps}) in the tables in \S~\ref{app:mts-chr}, we specify the theta contributions for 
the McKay--Thompson series with $D_0=-3$ in the tables in \S~\ref{app:mts-tht}, 
we specify the singular parts of the scalar-valued versions $\breve F^{(\pd)}_g$ of the McKay--Thompson series with $D_0=-4$ in the tables in \S~\ref{app:mts-sng}, and specify the first few coefficients of each of the McKay--Thompson series, both for $D_0=-3$ and $D_0=-4$, in the tables in \S~\ref{app:mts-cff}.

In all the tables that follow, the entries in the first rows are (mostly) conjugacy class names. 
Note that we use $nAB\dots$ as a shorthand for the union of conjugacy classes $nA$, $nB$, \dots in these tables.

In the tables in \S~\ref{app:mts-chr}, the entries of the second rows are characters, given in the form $n|h_v$, except that we simply write $n$ in case $h=v=1$. In Table \ref{tab:mts-chr:D0-3ell3+3}, for $\pd=(-3,3+3)$, we also specify the presence, or not, of a matrix $A := \mathrm{diag}(1,-\frac{1}{2},-\frac{1}{2},1,-\frac{1}{2},-\frac{1}{2})$. If this $A$ is present then we take $M_g=A$ in (\ref{eqn:mns-fms:MTRad}). Otherwise we take $M_g$ to be the identity. (In particular, $M_g$ may be omitted from (\ref{eqn:mns-fms:MTRad}) whenever $\pd\neq (-3,3+3)$.)

In the tables in \S~\ref{app:mts-tht}, the entries of the second rows are theta series, given as sums of symbols of the form $\kappa(k^2)$. Such a symbol is to be interpreted as $\kappa\theta_m(k^2\tau)$, where $m$ is the level of the relevant lambency $\ell$. For an example, consider the sum 
$\frac12(9)-\frac32({81})$, 
which is 
attached to the conjugacy class $27BC$ of $G^{(-3,1)}=\Th$ by Table \ref{tab:mts-tht:D0-3ell1}. The corresponding theta series is
$\frac12\theta_1(9\tau)-\frac32\theta_1(81\tau)$.

The tables in \S~\ref{app:mts-chr} and \S~\ref{app:mts-tht} provide all the parameters necessary for the Rademacher sum description (\ref{eqn:mns-fms:MTRad}) of the McKay--Thompson series $F^{(\pd)}_g$ of penumbral moonshine, for $\pd=(D_0,\ell)$ with $D_0=-3$.

The tables in \S~\ref{app:mts-sng} determine the McKay--Thompson series $F^{(\pd)}_g$ of penumbral moonshine, for $\pd=(D_0,\ell)$ with $D_0=-4$, up to theta series, by specifying the singular parts of the scalar-valued functions $\breve F^{(\pd)}_g$ (see (\ref{eqn:mns-mod:breveFpdg})). More precisely, if $m$ is the level of the lambency $\ell$ in question, then $\breve F^{(\pd)}_g$ is a weakly holomorphic modular form of weight $\frac12$ with trivial multiplier for $\Gamma_0(4mnh)$, where $n$ and $h$ are as specified in the tables of \S~\ref{app:mts-chr}. The expansion at infinity takes the form $2q^{-4}+O(1)$ in each case, so we suppress this information from the tables in question. For each non-infinite cusp of $\Gamma_0(4mnh)$ where the function $\breve F^{(\pd)}_g$ does not remain bounded, we specify the corresponding singular part. 
For example, taking $\pd=(-4,1)$ and $g$ in the class $12A$ or $12B$,
we see from Table \ref{tab:mts-sng:D0-4ell1} that the expansion of $\breve F^{(\pd)}_g$ in a neighborhood of the cusp of $\Gamma_0(96)$ represented by $\frac1{24}$ takes the form $4iq^{-1}+O(1)$, and this function is bounded near all other non-infinite cusps. (We have $h=2$ for $g$  in $12A$ or $12B$, according to Table \ref{tab:mts-chr:D0-4ell1}.) 
Note that conjugacy classes whose McKay--Thompson series remain bounded near non-infinite cusps are suppressed from the tables of \S~\ref{app:mts-sng}. Note also that we use $\xi$ as a shorthand for $\ex(\frac1{64})$ in these tables, and we use $\omega$ as a shorthand for $\ex(\frac1{20})$.

With the singular parts in place, by virtue of \S~\ref{app:mts-sng}, the coefficients given in the tables in \S~\ref{app:mts-cff} are sufficient to remove any ambiguity arising from theta series, in the cases that $D_0=-4$. We also present coefficients for the cases that $D_0=-3$ in \S~\ref{app:mts-cff}.

\clearpage

\subsection{Characters}\label{app:mts-chr}

\begin{table}[h!]
\begin{center}
\begin{small}
\caption{\small 
Characters of penumbral moonshine at $\pd=(-3,1)$.
\label{tab:mts-chr:D0-3ell1}}
\begin{tabular}{c|cccccccccc}
\toprule
$[g]$	&1A	&2A	&3A	&3B	&3C	&4A	&4B	&5A	&6A	&6B		\\
	\midrule
$n\vert h_v$ &$1$  &$2$ &$3|3_2$ &$3$ &$3\vert 3_1$ &$4$ &$4\vert 2_1$ &$5$ &$6\vert 3_2$ & $6\vert 3_1$   \\
\midrule 
\midrule
$[g]$	&6C &7A & 8A & 8B & 9A & 9B & 9C & 10A & 12AB & 12C  	\\
\midrule
$n\vert h_v$ &$6$  &$7$ &$8\vert 2_1$ &$8\vert 4_3$ &$9$ &$9$ &$9\vert 3_2$ &$10$ &$12\vert 3_2$ & $12$    \\
	\midrule
\midrule
$[g]$	&  12D & 13A & 14A & 15AB & 18A & 18B & 19A &20A & 21A & 24AB  	\\
\midrule
$n\vert h_v$ &$12\vert 6_5$  &$13$  &$14$ &$15\vert 3_2$ &$18$ &$18\vert 3_1$ &$19$ &$20\vert 2_1$ &$21\vert 3_2$ &$24\vert 6_5$       \\
	\midrule
\midrule
$[g]$ & 24CD & 27A & 27BC & 28A & 30AB & 31AB & 36A & 36BC & 39AB &  \\
\midrule
$n\vert h_v$ &$24\vert 12_{11}$  &$27\vert 3_2$ &$27\vert 3_2$ &$28$ &$30\vert 3_1$ &$31$ &$36$ &$36$ &$39\vert 3_2$ &      \\
\bottomrule
\end{tabular}
\end{small}
\end{center}
\end{table}

\vspace{-15pt}

 \begin{table}[h!]
 \begin{center}
 \begin{small}
\caption{\small Characters of penumbral moonshine at $\pd = (-3,3+3)$. 
\label{tab:mts-chr:D0-3ell3+3}} 
\begin{tabular}{c|ccccccccc}
 \toprule
$[g]$  & 1A & 3AB & 2A & 6AB & 3CD & 3E  &  3F & 3G\\
 \midrule
$M\cdot n\vert h_v$ & $1$ & $A\cdot 1$ & $2$ & $A\cdot 2$ & $3$ & $3$ & $3\vert 3_2$ & $3$ \\
 \midrule
 \midrule
$[g]$ & 4AB & 12ABCD & 6CDF & 6E & 7A & 21AB &  8AB & 24ABCD\\
 \midrule
$M\cdot n\vert h_v$& $4$ & $A\cdot 4$ & $6$ & $6\vert 3_1$ & $7$ & $A\cdot 7$ &  $8\vert 2_1$ & $A\cdot(8\vert 2_1)$\\
 \midrule
 \midrule
$[g]$  & 9A & 9BC & 12EH & 12FGIJ & 13AB & 39ABCD\\
 \midrule
$M\cdot n\vert h_v$ & $9\vert 3_2$ & $9\vert 3_2$ & $12$ & $12$ & $13$ & $A\cdot 13$\\
 \bottomrule
  \end{tabular}
 \end{small}
 \end{center}
 \end{table}

\vspace{-15pt}

\begin{table}[h!]
\begin{center}
\begin{small}
\caption{\small Characters of penumbral moonshine at $\pd = (-3,7+7)$. 
\label{tab:mts-chr:D0-3ell7+7}} 
\begin{tabular}{c|ccccc}
\toprule
$[g]$	&        1A      & 2A	   & 3A   & 4A	  	& 7AB	           		   	   	\\
	\midrule
$n\vert h_v$   & $1$  & $2$ & $3\vert 3_2$   & $4$ 			  &  $7$ 	   \\
\bottomrule 
\end{tabular}
\end{small}
\end{center}
\end{table}
 
 \vspace{-15pt}
 
\begin{table}[h!]
\begin{center}
\begin{small}
\caption{\small Characters of penumbral moonshine at $\pd = (-3,13+13)$. 
\label{tab:mts-chr:D0-3ell13+13}} 
\begin{tabular}{c|cc}
\toprule
 $[g]$&      1A      & 3AB	      	           		   	   	\\
	\midrule
 $n\vert h_v$ & $1$  & $3\vert 3_2$ 		 	   \\
\bottomrule 
\end{tabular}
\end{small}
\end{center}
\end{table}

\vspace{-15pt}

\begin{table}[ht]
\begin{center}
\begin{small}
\caption{\small Characters of penumbral moonshine at $\pd = (-4,1)$. 
\label{tab:mts-chr:D0-4ell1}
} 
\begin{tabular}{c|ccccccccc}
\toprule
$[g]$ & {1A} & {2A} & {2B} & {2CD} & {2EF} & {2G} & {3A} & {6A} & {3B} \\\midrule
$n\vert h_v$& {$1$} & {$2\vert 2_1$} & {$2$} & {$2\vert 2_1$} & {$2$} & {$2\vert 2_1$} & {$3$} & {$6\vert 2_1$} & {$3$} \\
\midrule
\end{tabular}
\begin{tabular}{c|ccccccc}
\midrule
$[g]$ & {6B} & {4ABC} & {4DEF} & {4G} & {4HI} & {4JK} & {4LM} \\\midrule
$n\vert h_v$& {$6\vert 2_1$} & {$4\vert 2_1$} & {$4\vert 4_3$} & {$4\vert 2_1$} & {$4$} & {$4\vert 4_1$} & {$4\vert 2_1$} \\\midrule
\end{tabular}
\begin{tabular}{c|ccccccccccc}
\midrule
$[g]$ & {4N} & {4O} & {4P} & {5A} & {10A} & {6C} & {6D} & {6E} & {6F} & {6G} & {6HI} \\\midrule
$n\vert h_v$ & {$4\vert 4_3$} & {$4\vert 4_1$} & {$4$} & {$5$} & {$10\vert 2_1$} & {$6$} & {$6\vert 2_1$} & {$6$} & {$6\vert 2_1$} & {$6$} & {$6\vert 2_1$} \\
\midrule
\end{tabular}
\begin{tabular}{c|cccccccc}
\midrule
$[g]$ &{6JK} & {6L} & {6M} & {6N} & {7A} & {14A} & {8ABC} & {8DE} \\
\midrule
$n\vert h_v$  & {$6$} & {$6\vert 2_1$} & {$6$} & {$6\vert 2_1$} & {$7$} & {$14\vert 2_1$}  & {$8$} & {$8\vert 4_3$} \\
\midrule
\end{tabular}
\begin{tabular}{c|cccccccccc}
\midrule
$[g]$ & {8FG} & {8H} & {8I} & {8J} & {8KL} & {9A} & {18A} & {10B} & {10CD} & {10E} \\
\midrule
$n\vert h_v$&  {$8\vert 4_1$} & {$8$} & {$8\vert 2_1$} & {$8\vert 4_3$} & {$8\vert 2_1$} & {$9$} & {$18\vert 2_1$} & {$10$} & {$10\vert 2_1$}  & {$10$} \\
\midrule
\end{tabular}
\begin{tabular}{c|ccccc}
\midrule
$[g]$ & {12ABC} & {12DEFG} & {12HI} & {12JK} & {12LMN} \\
\midrule
$n\vert h_v$& {$12\vert 2_1$}  & {$12\vert 4_1$} & {$12\vert 2_1$} & {$12\vert 4_3$} & {$12\vert 2_1$} \\
\midrule
\end{tabular}
\begin{tabular}{c|ccccccccc}
\midrule
$[g]$ & {12OP} & {13A} & {26A} & {14B} & {14C} & {15A} & {30A} & {16AB} & {16CD} \\
\midrule
$n\vert h_v$& {$12\vert 4_3$} & {$13$} & {$26\vert 2_1$} & {$14$} & {$14\vert 2_1$} & {$15$} & {$30\vert 2_1$} & {$16$} &  {$16\vert 2_1$} \\
\midrule
\end{tabular}
\begin{tabular}{c|cccccccccccccc}
\midrule
$[g]$ & {16EF} & {17A} & {34A} & {18B} & {18C} & {20ABC} & {21A} & {42A} & {24AB} \\\midrule
$n\vert h_v$& {$16\vert 4_3$} & {$17$} & {$34\vert 2_1$} & {$18$} & {$18\vert 2_1$} & {$20\vert 4_3$} & {$21$} & {$42\vert 2_1$} & {$24\vert 4_1$} \\
\midrule
\end{tabular}
\begin{tabular}{c|cccccccc}
\midrule
$[g]$ & {24CD} & {28AB} & {30B} & {30C} & {4QR} & {4S} & {12Q} & {8MN} \\
\midrule
$n\vert h_v$& {$24\vert 4_3$}  &  {$28\vert 2_1$} & {$30\vert 2_1$} & {$30$} & {$4\vert 4_3$} & {$4\vert 4_1$} & {$12\vert 4_1$} & {$8\vert 8_3$} \\
\midrule
\end{tabular}
\begin{tabular}{c|ccccccccc}
\midrule
$[g]$ & {8OP} & {8Q} & {8RS} & {8T} & {8U} & {20D} & {12R} & {16GHI} & {16J} \\
\midrule
$n\vert h_v$ & {$8\vert 4_1$} & {$8$} & {$8\vert 8_5$} & {$8\vert 4_3$} & {$8\vert 8_7$} & {$20\vert 4_3$} & {$12\vert 4_3$} & {$16\vert 4_3$} & {$16$} \\
\midrule
\end{tabular}
\begin{tabular}{c|cccccccc}
\midrule
$[g]$ & {16KLMN} & {20E} & {24EF} & {24GHIJ} & {52AB} & {32ABCD} & {32EFGH} & {40ABCD} \\
\midrule
$n\vert h_v$& {$16\vert 4_3$} & {$20\vert 4_3$} & {$24\vert 4_3$} & {$24\vert 8_7$} & {$52\vert 4_3$} & {$32$} & {$32\vert 4_3$} & {$40\vert 8_7$} \\
\bottomrule
\end{tabular}
\end{small}
\end{center}
\end{table}

\begin{table}[ht]
\begin{center}
\begin{small}
\caption{\small Characters of penumbral moonshine at $\pd = (-4,2+2)$. 
\label{tab:mts-chr:D0-4ell2+2}
} 
\begin{tabular}{c|ccccccccc}
\toprule
$[g]$ & {1A} & {2A} & {2B} & {3A} & {4A} & {4B} & {4C} & {5A} & {6A} \\\midrule
$n\vert h_v$& {$1$} & {$2$} & {$2\vert 2_1$} & {$3$} & {$4\vert 4_3$} & {$4\vert 2_1$} & {$4\vert 4_1$} & {$5$} & {$6\vert 2_1$} \\
\midrule
\end{tabular}
\begin{tabular}{c|cccccc}
\midrule
$[g]$ & {8AB} & {8C} & {10A} & {12A} & {13A} & {16AB} \\\midrule
$n\vert h_v$ & {$8\vert 4_1$} & {$8\vert 8_1$} & {$10$} & {$12\vert 4_1$} & {$13$} & {$16\vert 8_5$} \\
\midrule
\end{tabular}
\begin{tabular}{c|cccccccccccc}
\midrule
$[g]$ & {4DE} & {4F} & {4G} & {8D} & {8E} & {12BC} & {16CD} & {20AB} \\
\midrule
$n\vert h_v$ & {$4$} & {$4\vert 4_3$} & {$4\vert 2_1$} & {$8\vert 4_3$} & {$8\vert 8_3$} & {$12\vert 4_1$} & {$16\vert 8_5$} & {$20$} \\
\bottomrule
\end{tabular}
\end{small}
\end{center}
\end{table}

\begin{table}
\begin{center}
\begin{small}
\caption{\small Characters of penumbral moonshine at $\pd = (-4,5+5)$. 
\label{tab:mts-chr:D0-4ell5+5}
} 
\begin{tabular}{c|ccccccccccc}
\toprule
$[g]$ & {1A} & {2AB} & {2CD} & {3A} & {6A} & {4ABCDE} & {5A} & {10A} &  {2E} &  {6B} & {6C} \\
\midrule
$n\vert h_v$& {$1$} & {$2\vert 2_1$} & {$2$} & {$3$} & {$6\vert 2_1$} & {$4\vert 4_3$} & {$5$} & {$10\vert 2_1$}  & {$2\vert 2_1$}  & {$6$} & {$6\vert 2_1$} \\
\midrule
\end{tabular}
\begin{tabular}{c|cc}
\midrule
$[g]$ &  {8ABCD} & {20AB}  \\\midrule
$n\vert h_v$ & {$8\vert 8_7$} & {$20\vert 4_3$}   \\
\bottomrule
\end{tabular}
\end{small}
\end{center}
\end{table}

\begin{table}[ht]
\begin{center}
\begin{small}
\caption{\small Characters of penumbral moonshine at $\pd = (-4,10+2,5,10)$. 
\label{tab:mts-chr:D0-4ell10+020510}
} 
\begin{tabular}{c|ccccc}\toprule
$[g]$ & {1A} & {5A} & {4AB} & {2A} \\
\midrule
$n\vert h_v$& {$1$} & {$5$} & {$4$} & {$2$} \\
\bottomrule
\end{tabular}
\end{small}
\end{center}
\end{table}

\begin{table}[ht]
\begin{center}
\begin{small}
\caption{\small Characters of penumbral moonshine at $\pd = (-4,13+13)$. 
\label{tab:mts-chr:D0-4ell13+13}
} 
\begin{tabular}{c|ccc}
\toprule
$[g]$ & {1A} & {2A} & {4AB} \\
\midrule
$n\vert h_v$& {$1$} & {$2\vert 2_1$} & {$4\vert 4_3$} \\
 \bottomrule
\end{tabular}
\end{small}
\end{center}
\end{table}

\begin{table}[ht]
\begin{center}
\smallskip
\begin{small}
\caption{\small Characters of penumbral moonshine at $\pd = (-4,17+17)$. 
\label{tab:mts-chr:D0-4ell17+17}
} 
\begin{tabular}{c|cc}
\toprule
$[g]$ & {1A} & {2A} \\
\midrule
$n\vert h_v$& {$1$} & {$2\vert 2_1$} \\
\bottomrule
\end{tabular}
\end{small}
\end{center}
\end{table}

\begin{table}[ht]
\begin{center}
\begin{small}
\caption{\small Characters of penumbral moonshine at $\pd = (-4,26+2,13,26)$. 
\label{tab:mts-chr:D0-4ell26+021326}
} 
\begin{tabular}{c|cc}
\toprule
$[g]$ & {1A} & {2A} \\
\midrule
$n\vert h_v$& {$1$} & {$2$} \\
 \bottomrule
\end{tabular}
\end{small}
\end{center}
\end{table}

\clearpage

\subsection{Theta Corrections}\label{app:mts-tht}

\begin{table}[h!]
\begin{center}
\begin{small}
\caption{\small Theta corrections for penumbral moonshine at $\pd = (-3,1)$. 
\label{tab:mts-tht:D0-3ell1}
} 
\begin{tabular}{c|ccccccc}
\toprule
$[g]$	&3A	&4A	 & 9A & 9B  & 12AB & 12C  &21A	\\
	\midrule
  $\sum_k\kappa(k^2)$ & $18(9)$ &$8(4)$  & $ 6(9)$ &$-3(9)$ & $-1(4)+3({36})$ & $-1(4)$  &$-\frac{3}{8}(9) $ \\
\midrule
\midrule
$[g]$ & 27A & 27BC & 28A  & 36A & 36BC & 39AB  \\
	\midrule
  $\sum_k\kappa(k^2)$ &$-1(9)+3(81)$ &$\frac{1}{2}(9)-\frac{3}{2}(81)$ &$(4)$  &$2(4)-3(36)$ & $-1(4)$ &$\frac{9}{7}(9)$   \\ \bottomrule
\end{tabular}
\end{small}
\end{center}
\end{table}

\vspace{-15pt}

 \begin{table}[h!]
 \begin{center}
 \begin{small}
\caption{\small Theta corrections for penumbral moonshine at $\pd = (-3,3+3)$. 
\label{tab:mts-tht:D0-3ell3+3}
} 
 \begin{tabular}{c|cccc}
 \toprule
$[g]$  & 3CD & 3E &  4AB & 12ABCD  \\
 \midrule
 $\sum_k\kappa(k^2)$  &  $ 12(9) $ & $ -6(9) $ &  $ 4(4) $ & $ 4(4) $ \\ 
\midrule
 \midrule
$[g]$  &  9A & 9BC & 12EH & 12FGIJ \\
 \midrule
 $\sum_k\kappa(k^2)$  & $ -2(9)+6(81) $ & $ (9)-3(81) $ & $ 4(4)-6(36) $ & $ -2(4) $ \\ \bottomrule
  \end{tabular}
 \end{small}
 \end{center}
 \end{table}

\vspace{-15pt}

 \begin{table}[h!]
\begin{center}
\begin{small}
\caption{\small Theta corrections for penumbral moonshine at $\pd = (-3,7+7)$. 
\label{tab:mts-tht:D0-3ell7+7}
} 
\begin{tabular}{c|cc}
\toprule
$[g]$	&     3A   	& 7AB	           		   	   	\\
	\midrule
  $\sum_k\kappa(k^2)$& $-\frac{3}{4}(9)$ & $-2(49)$		 \\
\bottomrule 
\end{tabular}
\end{small}
\end{center}
\end{table}

\vspace{-15pt}

 \begin{table}[h!]
\begin{center}
\begin{small}
\caption{\small Theta corrections for penumbral moonshine at $\pd = (-3,13+13)$. 
\label{tab:mts-tht:D0-3ell13+13}
} 
\begin{tabular}{c|cc}
\toprule
 $[g]$   & 3AB	      	           		   	   	\\
	\midrule
  $\sum_k\kappa(k^2)$  & $\frac{9}{7}(9)$ 	 \\
\bottomrule 
\end{tabular}
\smallskip
\end{small}
\end{center}
\end{table}

\clearpage

\subsection{Singular Parts}\label{app:mts-sng}

\begin{table}[ht]
\begin{center}
\begin{small}
\caption{\small Singular parts for penumbral moonshine at $\pd = (-4,1)$. 
\label{tab:mts-sng:D0-4ell1}
} 
\begin{tabular}{
c|cc|cc|cc|cc|cc|cc|cc|cc|c}
\toprule
\multicolumn{2}{c}{1A} & \multicolumn{2}{c}{2A} & \multicolumn{2}{c}{2B} & \multicolumn{2}{c}{2C} 
& \multicolumn{2}{c}{3AB} & \multicolumn{2}{c}{6AB} 
& \multicolumn{2}{c}{4A} & \multicolumn{2}{c}{4BC}
\\
\midrule
$0$ & $-\frac{i - 1}{q}$& $\frac{1}{4}$ & $-\frac{2 i}{q^{4}}$& $\frac{1}{4}$ & $\frac{16}{q}$& $0$ & $\frac{8 i - 8}{q}$
& $\frac{1}{3}$ & $\frac{i - 1}{q}$& $\frac{1}{12}$ & $\frac{2 i}{q^{4}}$
& $\frac{1}{8}$ & $-\frac{16 i}{q}$
& $\frac{1}{8}$ & $\frac{16 i}{q}$
\\
[3pt]
\multicolumn{1}{l}{} & \multicolumn{1}{l}{} & $\frac{1}{2}$ & $-\frac{{2\xi^{-8}}}{q^{4}}$& $\frac{1}{2}$ & $\frac{16 i}{q}$& $\frac{1}{4}$ & $\frac{16 i}{q}$
& \multicolumn{1}{l}{} & \multicolumn{1}{l}{} 
& $\frac{1}{6}$ & $\frac{{2\xi^{8}}}{q^{4}}$
& \multicolumn{1}{l}{} & \multicolumn{1}{l}{} 
& \multicolumn{1}{l}{} & \multicolumn{1}{l}{} 
\\
[3pt]
 \midrule
\end{tabular}
\begin{tabular}{
c|cc|cc|cc|cc|cc|cc|c}
\midrule
\multicolumn{2}{c}{5A} & \multicolumn{2}{c}{10A} & \multicolumn{2}{c}{6C} & \multicolumn{2}{c}{6D} & \multicolumn{2}{c}{6E} & \multicolumn{2}{c}{6F} & \multicolumn{2}{c}{6G} 
\\
\midrule
$\frac{1}{5}$ & $-\frac{i + 1}{q}$& $\frac{1}{20}$ & $-\frac{2 i}{q^{4}}$& $\frac{1}{12}$ & $-\frac{8}{q}$& $\frac{1}{12}$ & $\frac{8 i}{q}$& $\frac{1}{12}$ & $\frac{4}{q}$& $\frac{1}{12}$ & $-\frac{4 i}{q}$& $\frac{1}{12}$ & $-\frac{2}{q}$
\\
[3pt]
\multicolumn{1}{l}{} & \multicolumn{1}{l}{} & $\frac{1}{10}$ & $\frac{{2\xi^{-8}}}{q^{4}}$& $\frac{1}{6}$ & $\frac{8 i}{q}$& $\frac{1}{3}$ & $-\frac{{\sqrt{32}\xi^{-4}}}{q}$& $\frac{1}{6}$ & $-\frac{4 i}{q}$& $\frac{1}{3}$ & $\frac{{\sqrt{8}\xi^{-4}}}{q}$& $\frac{1}{6}$ & $\frac{2 i}{q}$
\\
[3pt]
 \midrule
\end{tabular}

\begin{tabular}{c|c c|c c|c c|c c|c c|c c|c}
\midrule
 \multicolumn{2}{c}{6H}  & \multicolumn{2}{c}{7A} & \multicolumn{2}{c}{14A} &\multicolumn{2}{c}{9A} & \multicolumn{2}{c}{18A} & \multicolumn{2}{c}{10B} & \multicolumn{2}{c}{10C} 
\\
\midrule
 $\frac{1}{12}$ & $\frac{2 i}{q}$ & $\frac{1}{7}$ & $\frac{i - 1}{q}$& $\frac{1}{28}$ & $\frac{2 i}{q^{4}}$ 
& 
$\frac{1}{9}$ & $-\frac{i + 1}{q}$& $\frac{1}{36}$ & $-\frac{2 i}{q^{4}}$& $\frac{1}{20}$ & $-\frac{4}{q}$& $\frac{1}{20}$ & $-\frac{4 i}{q}$
\\
[3pt]
 $\frac{1}{3}$ & $-\frac{{\sqrt{2}\xi^{-4}}}{q}$ & \multicolumn{1}{l}{} & \multicolumn{1}{l}{} & $\frac{1}{14}$ & $-\frac{{2\xi^{8}}}{q^{4}}$
& 
\multicolumn{1}{l}{} & \multicolumn{1}{l}{} & $\frac{1}{18}$ & $-\frac{{2\xi^{-8}}}{q^{4}}$& $\frac{1}{10}$ & $-\frac{4 i}{q}$& $\frac{1}{5}$ & $\frac{{\sqrt{8}\xi^{4}}}{q}$
\\
[3pt]
 \midrule
\end{tabular}
\begin{tabular}{c|cc|cc|cc|c}
\midrule
\multicolumn{2}{c}{12AB}  & \multicolumn{2}{c}{12C} 
& \multicolumn{2}{c}{12HI} 
& \multicolumn{2}{c}{12L} 
\\\midrule
$\frac{1}{24}$ & $\frac{4 i}{q}$
& $\frac{1}{24}$ & $-\frac{2 i}{q}$
& $\frac{1}{24}$ & $-\frac{4 i}{q}$
& $\frac{1}{24}$ & $-\frac{6 i}{q}$
\\
[3pt]
 \midrule
\end{tabular}

\begin{tabular}{c|cc|cc|cc|cc|cc|c}
\midrule
\multicolumn{2}{c}{13A} & \multicolumn{2}{c}{26A} & \multicolumn{2}{c}{14B} & \multicolumn{2}{c}{14C} & \multicolumn{2}{c}{15A} & \multicolumn{2}{c}{30A} 
\\\midrule
$\frac{1}{13}$ & $-\frac{i + 1}{q}$& $\frac{1}{52}$ & $-\frac{2 i}{q^{4}}$& $\frac{1}{28}$ & $\frac{2}{q}$& $\frac{1}{28}$ & $-\frac{2 i}{q}$& $\frac{1}{15}$ & $\frac{i - 1}{q}$& $\frac{1}{60}$ & $\frac{2 i}{q^{4}}$
\\
[3pt]
\multicolumn{1}{l}{} & \multicolumn{1}{l}{} & $\frac{1}{26}$ & $\frac{{2\xi^{-8}}}{q^{4}}$& $\frac{1}{14}$ & $-\frac{2 i}{q}$& $\frac{1}{7}$ & $\frac{i{\sqrt{2}\xi^{-4}}}{q}$& \multicolumn{1}{l}{} & \multicolumn{1}{l}{} & $\frac{1}{30}$ & $-\frac{{2\xi^{8}}}{q^{4}}$
\\
[3pt]
 \midrule
\end{tabular}

\begin{tabular}{c|cc|cc|cc|cc|cc|c}
\midrule
\multicolumn{2}{c}{17A} & \multicolumn{2}{c}{34A} & \multicolumn{2}{c}{18B} & \multicolumn{2}{c}{18C} 
& \multicolumn{2}{c}{21A} & \multicolumn{2}{c}{42A} 
\\
\midrule
$\frac{1}{17}$ & $-\frac{i + 1}{q}$& $\frac{1}{68}$ & $-\frac{2 i}{q^{4}}$& $\frac{1}{36}$ & $-\frac{2}{q}$& $\frac{1}{36}$ & $-\frac{2 i}{q}$
& $\frac{1}{21}$ & $-\frac{i + 1}{q}$& $\frac{1}{84}$ & $-\frac{2 i}{q^{4}}$
\\
[3pt]
\multicolumn{1}{l}{} & \multicolumn{1}{l}{} & $\frac{1}{34}$ & $-\frac{{2\xi^{-8}}}{q^{4}}$& $\frac{1}{18}$ & $-\frac{2 i}{q}$& $\frac{1}{9}$ & $\frac{i{\sqrt{2}\xi^{4}}}{q}$
& \multicolumn{1}{l}{} & \multicolumn{1}{l}{} 
& $\frac{1}{42}$ & $\frac{{2\xi^{-8}}}{q^{4}}$
\\
[3pt]
 \midrule
\end{tabular}

\begin{tabular}{c|cc|cc|cc|cc|cc|cc|c}\midrule
\multicolumn{2}{c}{28AB} 
& \multicolumn{2}{c}{30B} & \multicolumn{2}{c}{30C} & \multicolumn{2}{c}{4Q} 
& \multicolumn{2}{c}{12Q} & \multicolumn{2}{c}{20D} 
& \multicolumn{2}{c}{52AB} 
\\
\midrule
$\frac{1}{56}$ & $-\frac{2 i}{q}$
& $\frac{1}{60}$ & $-\frac{2 i}{q}$& $\frac{1}{60}$ & $\frac{2}{q}$& $\frac{1}{8}$ & $\frac{{2\xi^{-8}}}{q^{4}}$
& $\frac{1}{24}$ & $-\frac{{2\xi^{8}}}{q^{4}}$
& $\frac{1}{40}$ & $-\frac{{2\xi^{-8}}}{q^{4}}$
& $\frac{1}{104}$ & $-\frac{{2\xi^{-8}}}{q^{4}}$
\\
[3pt]
\multicolumn{1}{l}{} & \multicolumn{1}{l}{} 
& $\frac{1}{15}$ & $-\frac{i{\sqrt{2}\xi^{-4}}}{q}$& $\frac{1}{30}$ & $-\frac{2 i}{q}$
& \multicolumn{1}{l}{} & \multicolumn{1}{l}{} 
& \multicolumn{1}{l}{} & \multicolumn{1}{l}{} 
& \multicolumn{1}{l}{} & \multicolumn{1}{l}{} 
& \multicolumn{1}{l}{} & \multicolumn{1}{l}{} 
\\
[3pt]
 \bottomrule
\end{tabular}
\end{small}
\end{center}

\end{table}

\clearpage

\begin{table}
\begin{center}
\begin{small}
\caption{\small Singular parts for penumbral moonshine at $\pd = (-4,2+2)$. 
\label{tab:mts-sng:D0-4ell2+2}
} 
\begin{tabular}{c|cc|cc|cc|cc|cc|cc|cc|c}
\toprule
 \multicolumn{2}{c}{1A} 
& \multicolumn{2}{c}{3A} 
& \multicolumn{2}{c}{5A} 
&\multicolumn{2}{c}{8AB} 
 & \multicolumn{2}{c}{8C} 
 & \multicolumn{2}{c}{13A} 
 & \multicolumn{2}{c}{16ABCD} 
& \multicolumn{2}{c}{20AB} 
 \\
 \midrule
 $\frac{1}{4}$ & $-\frac{2}{q}$
& $\frac{1}{12}$ & $-\frac{2}{q}$
& $\frac{1}{20}$ & $-\frac{2}{q}$
&   $\frac{1}{16}$ & $\frac{2 \, \xi^{28}}{q}$
& $\frac{1}{8}$ & $\frac{4 \, \xi^{23}}{q}$
& $\frac{1}{52}$ & $-\frac{2}{q}$
& $\frac{1}{64}$ & $\frac{\sqrt{2} \xi^{20}}{q^{4}}$
& $\frac{1}{8}$ & $\frac{2 \, \omega^{7}}{q}$
\\
[3pt]
 $\frac{1}{2}$ & $-\frac{2 i}{q}$
& $\frac{1}{6}$ & $\frac{2 i}{q}$
& $\frac{1}{10}$ & $-\frac{2 i}{q}$
&   $\frac{1}{8}$ & $\frac{2 \, \xi^{22}}{q}$
& $\frac{3}{8}$ & $\frac{4 \, \xi^{21}}{q}$
& $\frac{1}{26}$ & $-\frac{2 i}{q}$
& $\frac{1}{32}$ & $\frac{2 \, \xi^{46}}{q}$
& $\frac{3}{8}$ & $\frac{2 \, \omega^{6}}{q}$
\\
[3pt]
 \multicolumn{1}{l}{} & \multicolumn{1}{l}{} 
& \multicolumn{1}{l}{} & \multicolumn{1}{l}{} 
& \multicolumn{1}{l}{} & \multicolumn{1}{l}{} 
&$\frac{3}{16}$ & $\frac{2 \, \xi^{36}}{q}$
& $\frac{5}{8}$ & $\frac{4 \, \xi^{27}}{q}$
& \multicolumn{1}{l}{} & \multicolumn{1}{l}{} 
& $\frac{3}{32}$ & $\frac{2 \, \xi^{10}}{q}$
\\
[3pt]
 \multicolumn{1}{l}{} & \multicolumn{1}{l}{} 
& \multicolumn{1}{l}{} & \multicolumn{1}{l}{} 
& \multicolumn{1}{l}{} & \multicolumn{1}{l}{} 
&$\frac{3}{8}$ & $\frac{2 \, \xi^{18}}{q}$
& $\frac{7}{8}$ & $\frac{4 \, \xi^{25}}{q}$
& \multicolumn{1}{l}{} & \multicolumn{1}{l}{} 
& \multicolumn{1}{l}{} & \multicolumn{1}{l}{}
\\
[3pt]
 \multicolumn{1}{l}{} & \multicolumn{1}{l}{} 
& \multicolumn{1}{l}{} & \multicolumn{1}{l}{} 
& \multicolumn{1}{l}{} & \multicolumn{1}{l}{} 
& $\frac{5}{8}$ & $\frac{2 \, \xi^{30}}{q}$
& \multicolumn{1}{l}{} & \multicolumn{1}{l}{} 
& \multicolumn{1}{l}{} & \multicolumn{1}{l}{} 
& \multicolumn{1}{l}{} & \multicolumn{1}{l}{} 
\\
\bottomrule
\end{tabular}
\end{small}
\end{center}

\end{table}

\begin{table}

\begin{center}
\begin{small}
\caption{\small Singular parts for penumbral moonshine at $\pd = (-4,5+5)$. 
\label{tab:mts-sng:D0-4ell5+5}
} 
\begin{tabular}{c|cc|cc|cc|cc|cc|c}
\toprule
\multicolumn{2}{c}{1A} & \multicolumn{2}{c}{2A} 
& \multicolumn{2}{c}{3A} & \multicolumn{2}{c}{6A} 
&\multicolumn{2}{c}{5A} & 
\multicolumn{2}{c}{10A} 
\\\midrule
$\frac{1}{5}$ & $\frac{\sqrt{2} \xi^{40}}{q}$
& $\frac{1}{20}$ & $-\frac{2 i}{q^{4}}$
& $\frac{1}{15}$ & $\frac{\sqrt{2} \xi^{24}}{q}$
& $\frac{1}{60}$ & $\frac{2 i}{q^{4}}$
&$\frac{1}{25}$ & $\frac{\sqrt{2} \xi^{40}}{q}$& 
$\frac{1}{100}$ & $-\frac{2 i}{q^{4}}$
\\
 [3pt]
\multicolumn{1}{l}{} & \multicolumn{1}{l}{} 
& $\frac{1}{10}$ & $\frac{2 \, \xi^{56}}{q^{4}}$
& \multicolumn{1}{l}{} & \multicolumn{1}{l}{} 
& $\frac{1}{30}$ & $\frac{2 \, \xi^{40}}{q^{4}}$
&\multicolumn{1}{l}{} & \multicolumn{1}{l}{} & 
$\frac{1}{50}$ & $\frac{2 \, \xi^{24}}{q^{4}}$
\\
 [3pt]
 \midrule
\end{tabular}

\begin{tabular}{c|cc|cc|cc|cc|cc|c}
\midrule
\multicolumn{2}{c}{2D} & \multicolumn{2}{c}{2E} 
& \multicolumn{2}{c}{6B} & \multicolumn{2}{c}{6C} 
&\multicolumn{2}{c}{4D} 
& \multicolumn{2}{c}{20AB} 
\\\midrule
$\frac{1}{20}$ & $-\frac{4}{q}$& 
$\frac{1}{20}$ & $-\frac{4 i}{q}$& 
$\frac{1}{60}$ & $\frac{2}{q}$& $\frac{1}{60}$ & $-\frac{2 i}{q}$
&$\frac{1}{40}$ & $\frac{2 \, \xi^{24}}{q^{4}}$
& $\frac{13}{200}$ & $\frac{2 \, \xi^{56}}{q^{4}}$
\\
 [3pt]
$\frac{1}{10}$ & $-\frac{4 i}{q}$& $\frac{1}{5}$ & $\frac{\sqrt{8} \xi^4}{q}$
& $\frac{1}{30}$ & $-\frac{2 i}{q}$& $\frac{1}{15}$ & $\frac{\sqrt{2} \xi^{44}}{q}$
\\
 [3pt]
 \bottomrule
\end{tabular}
\end{small}
\end{center}
\end{table}

\begin{table}[ht]
\begin{center}
\begin{small}
\caption{\small Singular parts for penumbral moonshine at $\pd = (-4,10+2,5,10)$. 
\label{tab:mts-sng:D0-4ell10+2,5,10}
} 
\begin{tabular}{c|cc|cc|cc|cc|c}\toprule
\multicolumn{2}{c}{1A} & \multicolumn{2}{c}{5A} & \multicolumn{2}{c}{4AB} 
\\\midrule
 $\frac{1}{20}$ & $-\frac{2}{q}$& $\frac{1}{100}$ & $-\frac{2}{q}$& $\frac{1}{8}$ & $\frac{2\omega^7}{q}$
 \\
 [3pt]
 $\frac{1}{10}$ & $-\frac{2 i}{q}$& $\frac{1}{50}$ & $-\frac{2 i}{q}$& $\frac{3}{8}$ & $\frac{2\omega^6}{q}$
 \\
 [3pt]
 \bottomrule
\end{tabular}
\end{small}
\end{center}
\end{table}

\begin{table}[ht]
\begin{center}
\begin{small}
\caption{\small Singular parts for penumbral moonshine at $\pd = (-4,13+13)$. 
\label{tab:mts-sng:D0-4ell13+13}
} 
\begin{tabular}{c|cc|cc|c}
\toprule
\multicolumn{2}{c}{1A} & \multicolumn{2}{c}{4AB} & \multicolumn{2}{c}{2A} 
\\\midrule
 $\frac{1}{13}$ & $-\frac{i + 1}{q}$& $\frac{1}{104}$ & $-\frac{\sqrt{2}(1-i)}{q^{4}}$& $\frac{1}{52}$ & $-\frac{2 i}{q^{4}}$
\\
[3pt]
 \multicolumn{1}{l}{} & \multicolumn{1}{l}{} & \multicolumn{1}{l}{} & \multicolumn{1}{l}{} & $\frac{1}{26}$ & $\frac{\sqrt{2}(1-i)}{q^{4}}$
\\
[3pt]
 \bottomrule
\end{tabular}
\end{small}
\end{center}
\end{table}

\begin{table}[ht]
\begin{center}
\begin{small}
\caption{\small Singular parts for penumbral moonshine at $\pd = (-4,17+17)$. 
\label{tab:mts-sng:D0-4ell17+17}
} 
\begin{tabular}{c|cc|c}
\toprule
\multicolumn{2}{c}{1A} & \multicolumn{2}{c}{2A} 
\\\midrule
 $\frac{1}{17}$ & $-\frac{i + 1}{q}$& $\frac{1}{68}$ & $-\frac{2 i}{q^{4}}$
 \\
 [3pt]
 \multicolumn{1}{l}{} & \multicolumn{1}{l}{} & $\frac{1}{34}$ & $-\frac{\sqrt{2}(1-i)}{q^{4}}$
 \\
 [3pt]
 \bottomrule
\end{tabular}
\end{small}
\end{center}
\end{table}

\begin{table}[ht]
\begin{center}
\begin{small}
\caption{\small Singular parts for penumbral moonshine at $\pd = (-4,26+2,13,26)$. 
\label{tab:mts-sng:D0-4ell26+2,13,26}
} 
\begin{tabular}{c|cc|c}\toprule
\multicolumn{2}{c}{1A} & \multicolumn{2}{c}{2A} \\\midrule
$\frac{1}{52}$ & $-\frac{2}{q}$& $\frac{1}{104}$ & $-\frac{2}{q^{4}}$\\
[3pt]
$\frac{1}{26}$ & $-\frac{2 i}{q}$& $\frac{1}{52}$ & $-\frac{2 i}{q}$\\
[3pt]
\multicolumn{1}{l}{} & \multicolumn{1}{l}{} & $\frac{3}{52}$ & $\frac{2}{q}$\\
[3pt]
\multicolumn{1}{l}{} & \multicolumn{1}{l}{} & $\frac{1}{13}$ & $-\frac{\sqrt{2}\xi^4}{ q}$\\
[3pt]
 \bottomrule
\end{tabular}
\end{small}
\end{center}
\end{table}

\clearpage

\subsection{Coefficients}\label{app:mts-cff}

\begin{table}[ht]
\begin{tiny}
\begin{center}
\caption{\small McKay--Thompson series for penumbral moonshine at $\pd = (-3,1)$, Part 1.
\label{tab:mts-cff:D0-3ell1_1}} 
\begin{tabular}{rr|@{ }r@{ }r@{ }r@{ }r@{ }r@{ }r@{ }r@{ }r@{ }r@{ }r@{ }r@{ }r@{ }r
}\toprule
$r$ &  $D$ &1A &2A &3A &3B &3C &4A &4B &5A &6A &6B &6C &7A 
\\
 \midrule 
$0$
 &  $0$ & $248$ & $-8$ & $14$ & $5$ & $-4$ & $8$ &  .& $-2$ & $4$ & $-2$ & $1$ & $3$ 
 \\
 & $4$ & $54000$ & $240$ & $-54$ & $54$ &  .& $16$ &  .&  .&  .& $-6$ & $6$ & $2$ 
 \\
 & $8$ & $3414528$ & $-1536$ &  .& $-108$ & $108$ &  .&  .& $28$ & $-12$ &  .& $12$ & $12$ 
 \\
 & $12$ & $88660992$ & $7168$ & $336$ & $12$ & $-312$ &  .&  .& $-8$ & $-8$ & $16$ & $4$ &  .
 \\
 & $16$ & $1417878000$ & $-26640$ & $-702$ & $702$ &  .& $16$ &  .&  .&  .& $18$ & $-18$ & $-14$ 
 \\
 & $20$ & $16555069440$ & $86016$ &  .& $-1512$ & $1512$ &  .&  .& $-60$ & $24$ &  .& $-24$ &  .
 \\
 & $24$ & $154184576000$ & $-250880$ & $2912$ & $-4$ & $-2920$ &  .&  .&  .& $40$ & $-32$ & $4$ &  .
 \\
\bottomrule

\end{tabular}

\end{center}
\end{tiny}
\end{table}

\begin{table}[ht]
\begin{tiny}
\begin{center}
\caption{\small McKay--Thompson series for penumbral moonshine at $\pd = (-3,1)$, Part 2.
\label{tab:mts-cff:D0-3ell1_2}} 
\begin{tabular}{rr|@{ }r@{ }r@{ }r@{ }r@{ }r@{ }r@{ }r@{ }r@{ }r@{ }r@{ }r@{ }r@{ }r@{ }r@{ }r@{ }r@{ }r@{ }r@{ }r@{ }r@{ }r@{ }r@{ }r}\toprule
$r$ &  $D$ 
&8A 		&8B 		&9A 		&9B 		&9C 		&10A 
&12A 	&12B 	&12C 	&12D 	&13A 	&14A 
&15A 	&15B 	&18A 	&18B 	&19A 	&20A 
&21A 	&24A 	&24B 	&24C 	&24D 
\\
 \midrule 
$0$
 &  $0$ 
	&  .		&  .		& $5$ 	& $-4$ 	& $2$ 	& $2$ 
 	& $2$ 	& $2$ 	& $-1$ 	&  .		& $1$ 	& $-1$ 
	& $1$ 	& $1$ 	& $1$ 	& $-2$ 	& $1$ 	&  .	
	&  .		&  .		&  .		&  .		&  .		
	\\
 & $4$ 
	&  .		&  .		&  .		&  .		&  .		&  .
 	& $-2$ 	& $-2$ 	& $-2$ 	&  .		& $-2$ 	& $2$ 
	&  .		&  .		&  .		&  .		& $2$ 	&  .
	& $2$ 	&  .		&  .		&  .		&  .
	\\
 & $8$ 
 	&  .		&  .		&  .		&  .		&  .		& $4$ 
	&  .		&  .		&  .		&  .		&  .		& $4$ 
	& $-2$ 	& $-2$ 	&  .		&  .		&  .		&  .
	&  .		&  .		&  .		&  .		&  .
	\\
 & $12$ 
	&  .		&  .		& $12$ 	& $12$ 	& $-6$ 	& $8$ 
 	&  .		&  .		&  .		&  .		& $4$ 	&  .
	& $-2$ 	& $-2$ 	& $4$ 	& $-2$ 	&  .		&  .
	&  .		&  .		&  .		&  .		&  .
	\\
 & $16$ 
 	&  .		&  .		&  .		&  .		&  .		&  .
 	& $-2$ 	& $-2$ 	& $-2$ 	&  .		& $6$ 	& $2$ 
	&  .		&  .		&  .		&  .		& $-2$ 	&  .
	& $-2$ 	&  .		&  .		&  .		&  .
	\\
 & $20$ 
 	&  .		&  .		&  .		&  .		&  .		& $-4$ 
 	&  .		&  .		&  .		&  .		&  .		&  .
	& $-3$ 	& $-3$ 	&  .		&  .		& $4$ 	&  .
	&  .		&  .		&  .		&  .		&  .
	\\
 & $24$ 
 	&  .		&  .		& $-4$ 	& $-4$ 	& $-4$ 	&  .
 	&  .		&  .		&  .		&  .		&  .		&  .
	&  .		&  .		& $4$ 	& $4$ 	& $4$ 	&  .
	&  .		&  .		&  .		&  .		&  .
	\\
\bottomrule

\end{tabular}

\end{center}
\end{tiny}
\end{table}

\begin{table}
\begin{tiny}
\begin{center}
\caption{\small McKay--Thompson series for penumbral moonshine at $\pd = (-3,1)$, Part 3.
\label{tab:mts-cff:D0-3ell1_3}} 
\begin{tabular}{rr|@{ }r@{ }r@{ }r@{ }r@{ }r@{ }r@{ }r@{ }r@{ }r@{ }r@{ }r@{ }r@{ }r@{ }r@{ }r@{ }r@{ }r@{ }r@{ }r@{ }r@{ }r@{ }r@{ }r@{ }r@{ }r@{ }r@{ }r@{ }r@{ }r@{ }r@{ }r@{ }r@{ }r@{ }r@{ }r@{ }r@{ }r}\toprule
$r$ &  $D$ 
&27A &27B &27C &28A &30A &30B &31A &31B &36A &36B &36C &39A &39B \\
 \midrule 
$0$
 &  $0$ 
	& $2$ & $-1$ & $-1$ & $1$ & $-1$ & $-1$ &  .&  .& $-1$ & $-1$ & $-1$ & $1$ & $1$ \\
 & $4$ 
	&  .&  .&  .& $2$ &  .&  .& $-2$ & $-2$ & $4$ & $-2$ & $-2$ & $-2$ & $-2$ \\
 & $8$ 
	&  .&  .&  .&  .& $-2$ & $-2$ & $2$ & $2$ &  .&  .&  .&  .&  .\\
 & $12$ 
	&  .&  .&  .&  .& $2$ & $2$ &  .&  .&  .&  .&  .& $-2$ & $-2$ \\
 & $16$ 
	&  .&  .&  .& $2$ &  .&  .&  .&  .& $4$ & $-2$ & $-2$ &  .&  .\\
 & $20$ 
	&  .&  .&  .&  .& $-1$ & $-1$ & $2$ & $2$ &  .&  .&  .&  .&  .\\
 & $24$ 
	& $-4$ & $-4$ & $-4$ &  .&  .&  .&  .&  .&  .&  .&  .&  .&  .\\
\bottomrule

\end{tabular}

\end{center}
\end{tiny}
\end{table}

 \begin{table}
 \begin{tiny}
 \begin{center}
\caption{\small McKay--Thompson series for penumbral moonshine at $\pd = (-3,3+3)$, Part 1.
\label{tab:mts-cff:D0-3ell3+3_1}} 
 \begin{tabular}{rr|@{ }r@{ }r@{ }r@{ }r@{ }r@{ }r@{ }r@{ }r@{ }r@{ }r@{ }r@{ }r@{ }r@{ }r@{ }r@{ }r@{ }r@{ }r@{ }r@{ }r@{ }r}\toprule
 $r$& $D$ &1A &3A &3B &2A &6A &6B &3C &3D &3E &3F &3G &4A &12A &12B &4B &12C &12D &6C &6D &6E \\
  \midrule 
 $0$& $0$ & $28$ & $28$ & $28$ & $-4$ & $-4$ & $-4$ & $10$ & $10$ & $-8$ & $4$ & $-2$ & $4$ & $4$ & $4$ & $4$ & $4$ & $4$ & $2$ & $2$ & $-4$ \\
 &$12$ & $672$ & $672$ & $672$ & $32$ & $32$ & $32$ & $24$ & $24$ & $24$ & $-12$ & $24$ &  .&  .&  .&  .&  .&  .& $8$ & $8$ & $-4$ \\
 &$24$ & $5824$ & $5824$ & $5824$ & $-64$ & $-64$ & $-64$ & $-8$ & $-8$ & $-8$ & $-8$ & $-8$ &  .&  .&  .&  .&  .&  .& $8$ & $8$ & $8$ \\
 &$36$ & $35672$ & $35672$ & $35672$ & $152$ & $152$ & $152$ & $-4$ & $-4$ & $-40$ & $56$ & $-28$ & $8$ & $8$ & $8$ & $8$ & $8$ & $8$ & $-4$ & $-4$ & $8$ \\
 &$48$ & $165312$ & $165312$ & $165312$ & $-320$ & $-320$ & $-320$ & $72$ & $72$ & $72$ & $-36$ & $72$ &  .&  .&  .&  .&  .&  .& $-8$ & $-8$ & $4$ \\
 &$60$ & $653120$ & $653120$ & $653120$ & $576$ & $576$ & $576$ & $-64$ & $-64$ & $-64$ & $-64$ & $-64$ &  .&  .&  .&  .&  .&  .&  .&  .&  .\\
 &$72$ & $2276352$ & $2276352$ & $2276352$ & $-1024$ & $-1024$ & $-1024$ & $-72$ & $-72$ & $-72$ & $144$ & $-72$ &  .&  .&  .&  .&  .&  .& $8$ & $8$ & $-16$ \\
\end{tabular}
 \end{center}
  \end{tiny}
 \end{table}

 \begin{table}
  \begin{tiny}
 \begin{center}
\caption{\small McKay--Thompson series for penumbral moonshine at $\pd = (-3,3+3)$, Part 2.
\label{tab:mts-cff:D0-3ell3+3_2}} 
 \begin{tabular}{rr|@{ }r@{ }r@{ }r@{ }r@{ }r@{ }r@{ }r@{ }r@{ }r@{ }r@{ }r@{ }r@{ }r@{ }r@{ }r@{ }r@{ }r@{ }r@{ }r@{ }r@{ }r@{ }r@{ }r@{ }r@{ }r@{ }r}\toprule
 $r$& $D$ &6F &7A &21A &21B &8A &24A &24B &8B &24C &24D &9A &9B &9C &12E &12F &12G &12H &12I &12J &13A &39A &39B &13B &39C &39D \\
  \midrule 
 $0$ & $0$ & $2$ &  .&  .&  .&  .&  .&  .&  .&  .&  .& $4$ & $-2$ & $-2$ & $-2$ & $-2$ & $-2$ & $-2$ & $-2$ & $-2$ & $2$ & $2$ & $2$ & $2$ & $2$ & $2$ \\
 &$12$ & $8$ &  .&  .&  .&  .&  .&  .&  .&  .&  .&  .&  .&  .&  .&  .&  .&  .&  .&  .& $-4$ & $-4$ & $-4$ & $-4$ & $-4$ & $-4$ \\
 &$24$ & $8$ &  .&  .&  .&  .&  .&  .&  .&  .&  .& $-8$ & $-8$ & $-8$ &  .&  .&  .&  .&  .&  .&  .&  .&  .&  .&  .&  .\\
 &$36$ & $-4$ &  .&  .&  .&  .&  .&  .&  .&  .&  .& $-4$ & $2$ & $2$ & $-4$ & $-4$ & $-4$ & $-4$ & $-4$ & $-4$ &  .&  .&  .&  .&  .&  .\\
 &$48$ & $-8$ &  .&  .&  .&  .&  .&  .&  .&  .&  .&  .&  .&  .&  .&  .&  .&  .&  .&  .& $4$ & $4$ & $4$ & $4$ & $4$ & $4$ \\
 &$60$ &  .& $-8$ & $-8$ & $-8$ &  .&  .&  .&  .&  .&  .& $8$ & $8$ & $8$ &  .&  .&  .&  .&  .&  .&  .&  .&  .&  .&  .&  .\\
 &$72$ & $8$ & $8$ & $8$ & $8$ &  .&  .&  .&  .&  .&  .&  .&  .&  .&  .&  .&  .&  .&  .&  .&  .&  .&  .&  .&  .&  .\\
\end{tabular}
 \end{center}
 \end{tiny}
 \end{table}

\begin{table}[t]
\captionsetup{justification=centering}
\begin{minipage}[t]{0.48\textwidth}
\begin{tiny}
\begin{center}
\caption{\small McKay--Thompson series for penumbral moonshine at $\pd = (-3,7+7)$.
\label{tab:mts-cff:D0-3ell7+7}} 
\begin{tabular}{rr|@{ }r@{ }r@{ }r@{ }r@{ }r@{ }r@{ }r}\toprule
 $r$ & $D$ &1A &2A &3A &4A &7A &7B \\
 \midrule 
$0$ & $0$ & $6$ & $-2$ &  .& $2$ & $-1$ & $-1$ \\
 &$28$ & $56$ & $8$ & $-4$ &  .&  .&  .\\
 &$56$ & $168$ & $-8$ &  .&  .&  .&  .\\
 &$84$ & $560$ & $16$ & $8$ &  .&  .&  .\\
 &$112$ & $1400$ & $-24$ & $-4$ &  .&  .&  .\\
 &$140$ & $3360$ & $32$ &  .&  .&  .&  .\\
 $1$ & $29$ & $-30$ & $2$ &  .& $2$ & $-2$ & $-2$ \\
\bottomrule
\end{tabular}
\end{center}
\end{tiny}
\end{minipage}
\hspace*{-\textwidth} \hfill
\captionsetup{justification=centering}
\begin{minipage}[t]{0.48\textwidth}
\begin{tiny}
\begin{center}
\caption{\small McKay--Thompson series for penumbral moonshine at $\pd = (-3,13+13)$.
\label{tab:mts-cff:D0-3ell13+13}} 
\begin{tabular}{rr|@{ }r@{ }r@{ }r@{ }r}\toprule
$r$ &  $D$ &1A &3A &3B \\
 \midrule 
$0$
 &  $0$ & $1$ & $1$ & $1$ \\
 & $52$ & $8$ & $-1$ & $-1$ \\
 & $104$ & $12$ &  .&  .\\
 & $156$ & $32$ & $2$ & $2$ \\
$1$
 &  $105$ & $-1$ & $-1$ & $-1$ \\
 & $157$ & $-5$ & $1$ & $1$ \\
$2$
 &  $4$ & $-1$ & $-1$ & $-1$ \\
 & $56$ & $-6$ &  .&  .\\
 & $108$ & $-14$ & $1$ & $1$ \\
 & $160$ & $-30$ &  .&  .\\
$3$
 &  $9$ & $1$ & $1$ & $1$ \\
 & $61$ & $3$ &  .&  .\\
 & $113$ & $6$ &  .&  .\\
 & $165$ & $11$ & $-1$ & $-1$ \\
$4$
 &  $16$ & $3$ &  .&  .\\
 & $68$ & $6$ &  .&  .\\
 & $120$ & $18$ &  .&  .\\
$5$
 &  $25$ & $-1$ & $-1$ & $-1$ \\
 & $77$ & $-6$ &  .&  .\\
 & $129$ & $-11$ & $1$ & $1$ \\
$6$
 &  $36$ & $-3$ &  .&  .\\
 & $88$ & $-8$ & $1$ & $1$ \\
 & $140$ & $-18$ &  .&  .\\
$7$
 &  $-3$ & $1$ & $1$ & $1$ \\
 & $49$ & $5$ & $-1$ & $-1$ \\
 & $101$ & $9$ &  .&  .\\
 & $153$ & $22$ & $1$ & $1$ \\
$8$
 &  $12$ & $2$ & $-1$ & $-1$ \\
 & $64$ & $3$ &  .&  .\\
 & $116$ & $12$ &  .&  .\\
 & $168$ & $18$ &  .&  .\\
$9$
 &  $29$ & $-3$ &  .&  .\\
 & $81$ & $-8$ & $1$ & $1$ \\
 & $133$ & $-19$ & $-1$ & $-1$ \\
$10$
 &  $48$ & $-2$ & $1$ & $1$ \\
 & $100$ & $-3$ &  .&  .\\
 & $152$ & $-12$ &  .&  .\\
$11$
 &  $17$ & $3$ &  .&  .\\
 & $69$ & $8$ & $-1$ & $-1$ \\
 & $121$ & $18$ &  .&  .\\
$12$
 &  $40$ & $2$ & $-1$ & $-1$ \\
 & $144$ & $5$ & $2$ & $2$ \\
$13$
 &  $13$ & $-2$ & $1$ & $1$ \\
 & $65$ & $-6$ &  .&  .\\
 & $117$ & $-18$ &  .&  .\\
\bottomrule

\end{tabular}
\end{center}
\end{tiny}
\end{minipage}
\end{table}

\clearpage

\begin{table}
\begin{center}
\caption{\small McKay--Thompson series for penumbral moonshine at $\pd = (-3,\ell)$, for the $\ell$ such that $G^{(\pd)}$ is trivial.
\label{tab:mts-cff:D0-3ell_Gpdtrivial}} 
\begingroup
\setlength{\tabcolsep}{10pt}
%
\endgroup

\end{center}
\end{table}

\clearpage

\begin{table}[p]
\begin{tiny}
\begin{center}
\caption{\small McKay--Thompson series for penumbral moonshine at $\pd = (-4,1)$, Part 1.
\label{tab:mts-cff:D0-4ell1_1}} 
\begin{tabular}{rr|@{ }r@{ }r@{ }r@{ }r@{ }r@{ }r}\toprule
$r$ &  $D$ &1A &2A &2B &2C &2D &2E \\
 \midrule 
$0$
 &  $-4$ & $2$ & $2$ & $2$ & $2$ & $2$ & $2$ \\
 & $0$ & $-492$ & $-492$ & $-236$ & $-236$ & $20$ & $20$ \\
 & $4$ & $285768$ & $285768$ & $-4536$ & $-4536$ & $72$ & $72$ \\
 & $8$ & $36946000$ & $36946000$ & $-40880$ & $-40880$ & $80$ & $80$ \\
 & $12$ & $1637253000$ & $1637253000$ & $-245880$ & $-245880$ & $-120$ & $-120$ \\
 & $16$ & $41113156392$ & $41113156392$ & $-1147608$ & $-1147608$ & $-216$ & $-216$ \\
 & $20$ & $714279354880$ & $714279354880$ & $-4521472$ & $-4521472$ & $512$ & $512$ \\
\bottomrule
\end{tabular}

\end{center}
\end{tiny}
\end{table}

\begin{table}[p]
\begin{tiny}
\begin{center}
\caption{\small McKay--Thompson series for penumbral moonshine at $\pd = (-4,1)$, Part 2.
\label{tab:mts-cff:D0-4ell1_2}} 
\begin{tabular}{rr|@{ }r@{ }r@{ }r@{ }r@{ }r@{ }r@{ }r@{ }r@{ }r@{ }r@{ }r@{ }r}\toprule
$r$ &  $D$ &2F &2G &3A &6A &3B &6B &4A &4B &4C &4D &4E \\
 \midrule 
$0$
 &  $-4$ & $2$ & $2$ & $2$ & $2$ & $2$ & $2$ & $2$ & $2$ & $2$ & $2$ & $2$ \\
 & $0$ & $20$ & $20$ & $-6$ & $-6$ & $-6$ & $-6$ & $36$ & $-28$ & $-28$ & $-20$ & $-20$ \\
 & $4$ & $72$ & $72$ &  .&  .&  .&  .& $-216$ & $168$ & $168$ & $72$ & $72$ \\
 & $8$ & $80$ & $80$ & $280$ & $280$ & $280$ & $280$ & $464$ & $-560$ & $-560$ & $-80$ & $-80$ \\
 & $12$ & $-120$ & $-120$ & $-810$ & $-810$ & $-810$ & $-810$ & $-888$ & $1160$ & $1160$ & $-120$ & $-120$ \\
 & $16$ & $-216$ & $-216$ &  .&  .&  .&  .& $2376$ & $-1848$ & $-1848$ & $216$ & $216$ \\
 & $20$ & $512$ & $512$ & $5224$ & $5224$ & $5224$ & $5224$ & $-4608$ & $3584$ & $3584$ & $512$ & $512$ \\
\bottomrule

\end{tabular}
\end{center}
\end{tiny}
\end{table}

\begin{table}
\begin{tiny}
\begin{center}
\caption{\small McKay--Thompson series for penumbral moonshine at $\pd = (-4,1)$, Part 3.
\label{tab:mts-cff:D0-4ell1_3}} 
\begin{tabular}{rr|@{ }@{\;}r@{ }r@{ }r@{ }r@{ }r@{ }r@{ }r@{ }r@{ }r@{ }r@{ }r@{ }r@{ }r@{ }r}\toprule
$r$ &  $D$ &4F &4G &4H &4I &4J &4K &4L &4M &4N &4O &4P &5A &10A \\
 \midrule 
$0$
 &  $-4$ & $2$ & $2$ & $2$ & $2$ & $2$ & $2$ & $2$ & $2$ & $2$ & $2$ & $2$ & $2$ & $2$ \\
 & $0$ & $-20$ & $4$ & $-12$ & $-12$ & $-4$ & $-4$ & $4$ & $4$ & $12$ & $-4$ & $4$ & $8$ & $8$ \\
 & $4$ & $72$ & $-24$ & $8$ & $8$ & $-24$ & $-24$ & $-24$ & $-24$ & $8$ & $-24$ & $40$ & $18$ & $18$ \\
 & $8$ & $-80$ & $-48$ & $80$ & $80$ & $48$ & $48$ & $-48$ & $-48$ & $-80$ & $48$ & $80$ &  .&  .\\
 & $12$ & $-120$ & $136$ & $-120$ & $-120$ & $136$ & $136$ & $136$ & $136$ & $-120$ & $136$ & $-120$ &  .&  .\\
 & $16$ & $216$ & $264$ & $-280$ & $-280$ & $-264$ & $-264$ & $264$ & $264$ & $280$ & $-264$ & $-248$ & $-108$ & $-108$ \\
 & $20$ & $512$ & $-512$ & $512$ & $512$ & $-512$ & $-512$ & $-512$ & $-512$ & $512$ & $-512$ & $512$ & $130$ & $130$ \\
\bottomrule

\end{tabular}

\end{center}
\end{tiny}
\end{table}

\clearpage

\begin{table}
\begin{tiny}
\begin{center}
\caption{\small McKay--Thompson series for penumbral moonshine at $\pd = (-4,1)$, Part 4.
\label{tab:mts-cff:D0-4ell1_4}} 
\begin{tabular}{rr|@{ }@{ }r@{ }r@{ }r@{ }r@{ }r@{ }r@{ }r@{ }r@{ }r@{ }r@{ }r@{ }r@{ }r@{ }r@{ }r@{ }r@{ }r@{ }r}\toprule
$r$ &  $D$ &6C &6D &6E &6F &6G &6H &6I &6J &6K &6L &6M &6N &7A &14A &8A &8B &8C \\
 \midrule 
$0$
 &  $-4$ & $2$ & $2$ & $2$ & $2$ & $2$ & $2$ & $2$ & $2$ & $2$ & $2$ & $2$ & $2$ & $2$ & $2$ & $2$ & $2$ & $2$ \\
 & $0$ & $-14$ & $-14$ & $10$ & $10$ & $-2$ & $-2$ & $2$ & $2$ & $2$ & $2$ & $2$ & $2$ & $-2$ & $-2$ & $-4$ & $-4$ & $-4$ \\
 & $4$ &  .&  .&  .&  .&  .&  .&  .&  .&  .&  .&  .&  .&  .&  .& $-8$ & $-8$ & $-8$ \\
 & $8$ & $40$ & $40$ & $-8$ & $-8$ & $16$ & $16$ & $8$ & $8$ & $8$ & $8$ & $8$ & $8$ &  .&  .& $16$ & $16$ & $16$ \\
 & $12$ & $-90$ & $-90$ & $54$ & $54$ & $-18$ & $-18$ & $6$ & $6$ & $6$ & $6$ & $6$ & $6$ & $26$ & $26$ & $8$ & $8$ & $8$ \\
 & $16$ &  .&  .&  .&  .&  .&  .&  .&  .&  .&  .&  .&  .&  .&  .& $-8$ & $-8$ & $-8$ \\
 & $20$ & $200$ & $200$ & $-88$ & $-88$ & $56$ & $56$ & $8$ & $8$ & $8$ & $8$ & $8$ & $8$ & $-28$ & $-28$ &  .&  .&  .\\
\bottomrule

\end{tabular}

\end{center}
\end{tiny}
\end{table}

\clearpage

\begin{table}
\begin{tiny}
\begin{center}
\caption{\small McKay--Thompson series for penumbral moonshine at $\pd = (-4,1)$, Part 5.
\label{tab:mts-cff:D0-4ell1_5}} 
\begin{tabular}{rr|@{ }r@{ }r@{ }r@{ }r@{ }r@{ }r@{ }r@{ }r@{ }r@{ }r@{ }r@{ }r@{ }r@{ }r@{ }r@{ }r@{ }r@{ }r@{ }r}\toprule
$r$ &  $D$ &8D &8E &8F &8G &8H &8I &8J &8K &8L &9A &18A &10B &10C &10D &10E &12A &12B &12C &12D \\
 \midrule 
$0$
 &  $-4$ & $2$ & $2$ & $2$ & $2$ & $2$ & $2$ & $2$ & $2$ & $2$ & $2$ & $2$ & $2$ & $2$ & $2$ & $2$ & $2$ & $2$ & $2$ & $2$ \\
 & $0$ & $4$ & $4$ & $-4$ & $-4$ & $4$ & $4$ & $4$ &  .&  .& $3$ & $3$ & $4$ & $4$ &  .&  .& $-6$ & $-6$ &  .& $-2$ \\
 & $4$ &  .&  .&  .&  .& $8$ & $-8$ &  .&  .&  .&  .&  .& $-6$ & $-6$ & $2$ & $2$ &  .&  .&  .&  .\\
 & $8$ &  .&  .&  .&  .& $16$ & $-16$ &  .& $-16$ & $-16$ & $10$ & $10$ &  .&  .&  .&  .& $8$ & $8$ & $-4$ & $-8$ \\
 & $12$ & $-8$ & $-8$ & $-8$ & $-8$ & $8$ & $8$ & $-8$ & $8$ & $8$ &  .&  .&  .&  .&  .&  .& $6$ & $6$ & $-6$ & $6$ \\
 & $16$ & $-24$ & $-24$ & $24$ & $24$ & $8$ & $8$ & $-24$ &  .&  .&  .&  .& $12$ & $12$ & $4$ & $4$ &  .&  .&  .&  .\\
 & $20$ &  .&  .&  .&  .&  .&  .&  .&  .&  .& $-14$ & $-14$ & $18$ & $18$ & $2$ & $2$ & $-24$ & $-24$ &  .& $8$ \\
\bottomrule

\end{tabular}

\end{center}
\end{tiny}
\end{table}

\clearpage

\begin{table}
\caption{\small McKay--Thompson series for penumbral moonshine at $\pd = (-4,1)$, Part 6.
\label{tab:mts-cff:D0-4ell1_6}} 
\begin{tiny}
\begin{center}
\begin{tabular}{rr|@{ }r@{ }r@{ }r@{ }r@{ }r@{ }r@{ }r@{ }r@{ }r@{ }r@{ }r@{ }r@{ }r@{ }r@{ }r@{ }r@{ }r@{ }r@{ }r@{ }r@{ }r@{ }r@{ }r@{ }r@{ }r}\toprule
$r$ &  $D$ &12E &12F &12G &12H &12I &12J &12K &12L &12M &12N &12O &12P &13A &26A &14B &14C &15A &30A &16A &16B &16C &16D &16E &16F \\
 \midrule 
$0$
 &  $-4$ & $2$ & $2$ & $2$ & $2$ & $2$ & $2$ & $2$ & $2$ & $2$ & $2$ & $2$ & $2$ & $2$ & $2$ & $2$ & $2$ & $2$ & $2$ & $2$ & $2$ & $2$ & $2$ & $2$ & $2$ \\
 & $0$ & $-2$ & $-2$ & $-2$ & $2$ & $2$ & $2$ & $2$ & $4$ & $-2$ & $-2$ & $2$ & $2$ & $2$ & $2$ & $2$ & $2$ & $-1$ & $-1$ & $4$ & $4$ &  .&  .&  .&  .\\
 & $4$ &  .&  .&  .&  .&  .&  .&  .&  .&  .&  .&  .&  .& $2$ & $2$ &  .&  .&  .&  .&  .&  .&  .&  .&  .&  .\\
 & $8$ & $-8$ & $-8$ & $-8$ & $-8$ & $-8$ &  .&  .& $-12$ &  .&  .&  .&  .&  .&  .&  .&  .&  .&  .&  .&  .&  .&  .&  .&  .\\
 & $12$ & $6$ & $6$ & $6$ & $-10$ & $-10$ & $-2$ & $-2$ & $-14$ & $-2$ & $-2$ & $-2$ & $-2$ & $6$ & $6$ & $2$ & $2$ &  .&  .& $8$ & $8$ & $-8$ & $-8$ &  .&  .\\
 & $16$ &  .&  .&  .&  .&  .&  .&  .&  .&  .&  .&  .&  .& $-4$ & $-4$ &  .&  .&  .&  .& $8$ & $8$ &  .&  .&  .&  .\\
 & $20$ & $8$ & $8$ & $8$ & $8$ & $8$ & $-8$ & $-8$ & $16$ & $-8$ & $-8$ & $-8$ & $-8$ &  .&  .& $-4$ & $-4$ & $4$ & $4$ &  .&  .&  .&  .&  .&  .\\
\bottomrule

\end{tabular}

\end{center}
\end{tiny}
\end{table}

\begin{table}
\begin{tiny}
\begin{center}
\caption{\small McKay--Thompson series for penumbral moonshine at $\pd = (-4,1)$, Part 7.
\label{tab:mts-cff:D0-4ell1_7}} 

\begin{tabular}{rr|@{ }r@{ }r@{ }r@{ }r@{ }r@{ }r@{ }r@{ }r@{ }r@{ }r@{ }r@{ }r@{ }r@{ }r@{ }r@{ }r@{ }r@{ }r@{ }r@{ }r@{ }r}\toprule
$r$ &  $D$ &17A &34A &18B &18C &20A &20B &20C &21A &42A &24A &24B &24C &24D &28A &28B &30B &30C &4Q &4R &4S \\
 \midrule 
$0$
 &  $-4$ & $2$ & $2$ & $2$ & $2$ & $2$ & $2$ & $2$ & $2$ & $2$ & $2$ & $2$ & $2$ & $2$ & $2$ & $2$ & $2$ & $2$ & $2$ & $2$ & $2$ \\
 & $0$ & $1$ & $1$ & $1$ & $1$ &  .&  .&  .& $1$ & $1$ & $-2$ & $-2$ & $2$ & $2$ &  .&  .& $1$ & $1$ & $-52$ & $12$ & $-4$ \\
 & $4$ & $-2$ & $-2$ &  .&  .& $2$ & $2$ & $2$ &  .&  .&  .&  .&  .&  .&  .&  .&  .&  .& $648$ & $8$ & $-24$ \\
 & $8$ & $2$ & $2$ & $-2$ & $-2$ &  .&  .&  .&  .&  .&  .&  .&  .&  .&  .&  .&  .&  .& $-5200$ & $-80$ & $48$ \\
 & $12$ &  .&  .&  .&  .&  .&  .&  .& $2$ & $2$ & $-2$ & $-2$ & $-2$ & $-2$ & $-2$ & $-2$ &  .&  .& $30600$ & $-120$ & $136$ \\
 & $16$ & $4$ & $4$ &  .&  .& $-4$ & $-4$ & $-4$ &  .&  .&  .&  .&  .&  .&  .&  .&  .&  .& $-143208$ & $280$ & $-264$ \\
 & $20$ &  .&  .& $2$ & $2$ & $2$ & $2$ & $2$ & $2$ & $2$ &  .&  .&  .&  .&  .&  .&  .&  .& $565760$ & $512$ & $-512$ \\
\bottomrule

\end{tabular}

\end{center}
\end{tiny}
\end{table}

\begin{table}
\begin{tiny}
\begin{center}
\caption{\small McKay--Thompson series for penumbral moonshine at $\pd = (-4,1)$, Part 8.
\label{tab:mts-cff:D0-4ell1_8}} 
\begin{tabular}{rr|@{ }r@{ }r@{ }r@{ }r@{ }r@{ }r@{ }r@{ }r@{ }r@{ }r@{ }r@{ }r@{ }r@{ }r@{ }r@{ }r@{ }r@{ }r@{ }r@{ }r@{ }r@{ }r@{ }r}\toprule
$r$ &  $D$ &12Q &8M &8N &8O &8P &8Q &8R &8S &8T &8U &20D &12R &16G &16H &16I &16J &16K &16L &16M &16N &20E &24E \\
 \midrule 
$0$
 &  $-4$ & $2$ & $2$ & $2$ & $2$ & $2$ & $2$ & $2$ & $2$ & $2$ & $2$ & $2$ & $2$ & $2$ & $2$ & $2$ & $2$ & $2$ & $2$ & $2$ & $2$ & $2$ & $2$ \\
 & $0$ & $2$ &  .&  .& $-4$ & $-4$ & $4$ &  .&  .& $4$ &  .& $-2$ & $2$ & $-4$ & $-4$ & $4$ & $4$ & $2$ & $2$ & $2$ & $2$ & $2$ & $2$ \\
 & $4$ &  .& $-8$ & $-8$ &  .&  .& $8$ &  .&  .&  .& $8$ & $-2$ &  .&  .&  .&  .&  .&  .&  .&  .&  .& $-2$ &  .\\
 & $8$ & $-16$ &  .&  .&  .&  .& $16$ &  .&  .&  .&  .&  .&  .&  .&  .&  .&  .&  .&  .&  .&  .&  .&  .\\
 & $12$ & $-18$ &  .&  .& $-8$ & $-8$ & $8$ & $-16$ & $-16$ & $-8$ &  .&  .& $-2$ &  .&  .&  .& $8$ &  .&  .&  .&  .&  .& $-2$ \\
 & $16$ &  .&  .&  .& $24$ & $24$ & $8$ &  .&  .& $-24$ &  .& $-8$ &  .& $8$ & $8$ & $-8$ & $8$ & $-4$ & $-4$ & $-4$ & $-4$ &  .&  .\\
 & $20$ & $56$ & $32$ & $32$ &  .&  .&  .&  .&  .&  .& $-32$ & $10$ & $-8$ &  .&  .&  .&  .&  .&  .&  .&  .& $2$ &  .\\
\bottomrule

\end{tabular}

\end{center}
\end{tiny}
\end{table}

\begin{table}
\begin{tiny}
\begin{center}
\caption{\small McKay--Thompson series for penumbral moonshine at $\pd = (-4,1)$, Part 9.
\label{tab:mts-cff:D0-4ell1_9}} 

\begin{tabular}{rr|@{ }r@{ }r@{ }r@{ }r@{ }r@{ }r@{ }r@{ }r@{ }r@{ }r@{ }r@{ }r@{ }r@{ }r@{ }r@{ }r@{ }r@{ }r@{ }r@{ }r}\toprule
$r$ &  $D$ &24F &24G &24H &24I &24J &52A &52B &32A &32B &32C &32D &32E &32F &32G &32H &40A &40B &40C &40D \\
 \midrule 
$0$
 &  $-4$ & $2$ & $2$ & $2$ & $2$ & $2$ & $2$ & $2$ & $2$ & $2$ & $2$ & $2$ & $2$ & $2$ & $2$ & $2$ & $2$ & $2$ & $2$ & $2$ \\
 & $0$ & $2$ &  .&  .&  .&  .&  .&  .& $2$ & $2$ & $2$ & $2$ &  .&  .&  .&  .&  .&  .&  .&  .\\
 & $4$ &  .&  .&  .&  .&  .& $-2$ & $-2$ &  .&  .&  .&  .&  .&  .&  .&  .& $2$ & $2$ & $2$ & $2$ \\
 & $8$ &  .&  .&  .&  .&  .&  .&  .&  .&  .&  .&  .&  .&  .&  .&  .&  .&  .&  .&  .\\
 & $12$ & $-2$ & $2$ & $2$ & $2$ & $2$ & $-2$ & $-2$ &  .&  .&  .&  .&  .&  .&  .&  .&  .&  .&  .&  .\\
 & $16$ &  .&  .&  .&  .&  .&  .&  .& $4$ & $4$ & $4$ & $4$ &  .&  .&  .&  .&  .&  .&  .&  .\\
 & $20$ &  .&  .&  .&  .&  .&  .&  .&  .&  .&  .&  .&  .&  .&  .&  .& $2$ & $2$ & $2$ & $2$ \\
\bottomrule
\end{tabular}
\end{center}
\end{tiny}
\end{table}

\clearpage

\begin{table}
\begin{tiny}
\begin{center}
\caption{\small McKay--Thompson series for penumbral moonshine at $\pd = (-4,2+2)$, Part 1.
\label{tab:mts-cff:D0-4ell2+2_1}} 
\begin{tabular}{rr|@{ }r@{ }r@{ }r@{ }r@{ }r@{ }r@{ }r@{ }r@{ }r@{ }r@{ }r@{ }r@{ }r@{ }r@{ }r@{ }r@{ }r@{ }r}\toprule
$r$ &  $D$ &1A &2A &2B &3A &4A &4B &4C &5A &6A &8A &8B &8C &10A &12A &13A &16A &16B \\
 \midrule 
$0$
 &  $0$ & $52$ & $-12$ & $4$ & $-2$ & $4$ & $-4$ & $-4$ & $2$ & $-2$ & $-4$ & $-4$ & $-4$ & $-2$ & $-2$ & $0$ & $-2$ & $-2$ \\
 & $8$ & $5200$ & $80$ & $-48$ & $16$ & $0$ & $-16$ & $0$ & $0$ & $0$ & $0$ & $0$ & $-8$ & $0$ & $0$ & $0$ & $0$ & $0$ \\
 & $16$ & $143208$ & $-280$ & $264$ & $0$ & $-24$ & $-8$ & $24$ & $8$ & $0$ & $-24$ & $-24$ & $-8$ & $0$ & $0$ & $0$ & $0$ & $0$ \\
 & $24$ & $1973088$ & $864$ & $-928$ & $-72$ & $0$ & $32$ & $0$ & $-12$ & $8$ & $0$ & $0$ & $-16$ & $4$ & $0$ & $0$ & $0$ & $0$ \\
 & $32$ & $18470400$ & $-2560$ & $2560$ & $132$ & $64$ & $0$ & $-64$ & $0$ & $4$ & $-64$ & $-64$ & $-16$ & $0$ & $4$ & $0$ & $8$ & $8$ \\
 & $40$ & $134615520$ & $6624$ & $-6432$ & $0$ & $0$ & $-96$ & $0$ & $20$ & $0$ & $0$ & $0$ & $-16$ & $4$ & $0$ & $0$ & $0$ & $0$ \\
 & $48$ & $818611200$ & $-15360$ & $15360$ & $-396$ & $-128$ & $0$ & $128$ & $0$ & $-12$ & $-128$ & $-128$ & $-32$ & $0$ & $4$ & $4$ & $0$ & $0$ \\
\bottomrule

\end{tabular}

\end{center}
\end{tiny}
\end{table}

\begin{table}
\begin{tiny}
\begin{center}
\caption{\small McKay--Thompson series for penumbral moonshine at $\pd = (-4,2+2)$, Part 2.
\label{tab:mts-cff:D0-4ell2+2_2}} 
\begin{tabular}{rr|@{ }r@{ }r@{ }r@{ }r@{ }r@{ }r@{ }r@{ }r@{ }r@{ }r@{ }r@{ }r@{ }r}\toprule
$r$ &  $D$ &4D &4E &4F &4G &8D &8E &12B &12C &16C &16D &20A &20B \\
 \midrule 
$0$
 &  $0$ & $-4$ & $-4$ & $4$ & $-4$ & $0$ & $0$ & $-2$ & $-2$ & $-2$ & $-2$ & $-4$ & $-4$ \\
 & $8$ & $16$ & $16$ & $0$ & $-16$ & $0$ & $0$ & $0$ & $0$ & $0$ & $0$ & $-4$ & $-4$ \\
 & $16$ & $-8$ & $-8$ & $-24$ & $-8$ & $0$ & $0$ & $0$ & $0$ & $0$ & $0$ & $-8$ & $-8$ \\
 & $24$ & $-32$ & $-32$ & $0$ & $32$ & $0$ & $0$ & $0$ & $0$ & $0$ & $0$ & $-12$ & $-12$ \\
 & $32$ & $0$ & $0$ & $64$ & $0$ & $0$ & $0$ & $4$ & $4$ & $8$ & $8$ & $-20$ & $-20$ \\
 & $40$ & $96$ & $96$ & $0$ & $-96$ & $0$ & $0$ & $0$ & $0$ & $0$ & $0$ & $-24$ & $-24$ \\
 & $48$ & $0$ & $0$ & $-128$ & $0$ & $0$ & $0$ & $4$ & $4$ & $0$ & $0$ & $-40$ & $-40$ \\
\bottomrule

\end{tabular}

\end{center}
\end{tiny}
\end{table}

\clearpage

\begin{table}
\begin{tiny}
\begin{center}
\caption{\small McKay--Thompson series for penumbral moonshine at $\pd = (-4,5+5)$.
\label{tab:mts-cff:D0-4ell5+5}} 
\begin{tabular}{rr|@{ }r@{ }r@{ }r@{ }r@{ }r@{ }r@{ }r@{ }r@{ }r@{ }r@{ }r@{ }r@{ }r@{ }r@{ }r@{ }r@{ }r@{ }r@{ }r@{ }r@{ }r@{ }r@{ }r@{ }r}\toprule
$r$ &  $D$ &1A &2A &2B &2C &3A &6A &4A &4B &5A &10A &2D &2E &4C &6B &6C &4D &8A &8B &20A &20B &4E &8C &8D \\
 \midrule 
$0$
 &  $0$ & $8$ & $8$ &  .&  .& $-1$ & $-1$ &  .&  .& $-2$ & $-2$ & $4$ & $4$ &  .& $1$ & $1$ & $-2$ &  .&  .& $-2$ & $-2$ & $2$ &  .&  .\\
 & $20$ & $130$ & $130$ & $2$ & $2$ & $4$ & $4$ & $2$ & $2$ &  .&  .& $18$ & $18$ & $2$ &  .&  .& $10$ & $2$ & $2$ &  .&  .& $2$ & $2$ & $2$ \\
 & $40$ & $900$ & $900$ & $4$ & $4$ &  .&  .& $-4$ & $-4$ &  .&  .& $36$ & $36$ & $-4$ &  .&  .& $-20$ &  .&  .&  .&  .& $-4$ &  .&  .\\
 & $60$ & $4350$ & $4350$ & $-2$ & $-2$ & $-6$ & $-6$ & $-2$ & $-2$ &  .&  .& $62$ & $62$ & $-2$ & $2$ & $2$ & $30$ & $-2$ & $-2$ &  .&  .& $-2$ & $-2$ & $-2$ \\
 & $80$ & $17020$ & $17020$ & $-4$ & $-4$ & $10$ & $10$ & $4$ & $4$ &  .&  .& $124$ & $124$ & $4$ & $-2$ & $-2$ & $-60$ &  .&  .&  .&  .& $4$ &  .&  .\\
 & $100$ & $57366$ & $57366$ & $6$ & $6$ &  .&  .& $6$ & $6$ & $-4$ & $-4$ & $222$ & $222$ & $6$ &  .&  .& $114$ & $-2$ & $-2$ & $4$ & $4$ & $2$ & $-2$ & $-2$ \\
 & $120$ & $173700$ & $173700$ & $4$ & $4$ & $-18$ & $-18$ & $-4$ & $-4$ &  .&  .& $356$ & $356$ & $-4$ & $2$ & $2$ & $-180$ &  .&  .&  .&  .& $-4$ &  .&  .\\
\bottomrule

\end{tabular}
\end{center}
\end{tiny}
\end{table}

\clearpage

\begin{table}[t]
\captionsetup{justification=centering}
\begin{minipage}[t]{0.48\textwidth}
\begin{tiny}
\begin{center}
\caption{\small McKay--Thompson series for penumbral moonshine at $\pd = (-4,10+2,5,10)$.
\label{tab:mts-cff:D0-4ell10+2510}} 
\begin{tabular}{rr|@{ }r@{ }r@{ }r@{ }r@{ }r@{ }r}\toprule
$r$ &  $D$ &1A &5A &4A &2A &4B \\
 \midrule 
$0$
 &  $0$ & $2$ & $2$ & $-4$ & $-2$ & $-4$ \\
 & $40$ & $20$ &  .& $-24$ & $4$ & $-24$ \\
 & $80$ & $60$ &  .& $-120$ & $-4$ & $-120$ \\
 & $120$ & $180$ &  .& $-404$ & $4$ & $-404$ \\
 & $160$ & $440$ &  .& $-1140$ & $-8$ & $-1140$ \\
$1$
 &  $41$ & $-4$ & $1$ &  .&  .&  .\\
 & $81$ & $-20$ &  .&  .&  .&  .\\
\bottomrule

\end{tabular}
\end{center}
\end{tiny}
\end{minipage}
\hspace*{-\textwidth} \hfill
\captionsetup{justification=centering}
\begin{minipage}[t]{0.48\textwidth}
\begin{tiny}
\begin{center}
\caption{\small McKay--Thompson series for penumbral moonshine at $\pd = (-4,13+13)$.
\label{tab:mts-cff:D0-4ell13+13}} 
\begin{tabular}{rr|@{ }r@{ }r@{ }r@{ }r@{ }r}\toprule
$r$ &  $D$ &1A &4A &2A &4B \\
 \midrule 
$0$
 &  $0$ & $2$ &  .& $2$ &  .\\
 & $52$ & $10$ & $2$ & $10$ & $2$ \\
 & $104$ & $28$ & $-4$ & $28$ & $-4$ \\
 & $156$ & $66$ & $2$ & $66$ & $2$ \\
$1$
 &  $53$ & $-8$ &  .& $8$ &  .\\
 & $105$ & $-18$ &  .& $18$ &  .\\
 & $157$ & $-52$ &  .& $52$ &  .\\
$2$
 &  $4$ & $1$ & $-1$ & $1$ & $-1$ \\
 & $56$ & $2$ & $2$ & $2$ & $2$ \\
 & $108$ & $3$ & $-1$ & $3$ & $-1$ \\
 & $160$ & $10$ & $2$ & $10$ & $2$ \\
$3$
 &  $9$ & $2$ &  .& $-2$ &  .\\
 & $61$ & $8$ &  .& $-8$ &  .\\
 & $113$ & $16$ &  .& $-16$ &  .\\
 & $165$ & $48$ &  .& $-48$ &  .\\
$4$
 &  $16$ & $-2$ &  .& $-2$ &  .\\
 & $68$ & $-12$ &  .& $-12$ &  .\\
 & $120$ & $-36$ &  .& $-36$ &  .\\
 & $172$ & $-83$ & $1$ & $-83$ & $1$ \\
$5$
 &  $25$ & $6$ &  .& $-6$ &  .\\
 & $77$ & $12$ &  .& $-12$ &  .\\
 & $129$ & $42$ &  .& $-42$ &  .\\
$6$
 &  $36$ & $-1$ & $1$ & $-1$ & $1$ \\
 & $88$ & $-6$ & $-2$ & $-6$ & $-2$ \\
 & $140$ & $-19$ & $1$ & $-19$ & $1$ \\
$7$
 &  $49$ & $-2$ &  .& $2$ &  .\\
 & $101$ & $-8$ &  .& $8$ &  .\\
 & $153$ & $-24$ &  .& $24$ &  .\\
$8$
 &  $12$ & $3$ & $-1$ & $3$ & $-1$ \\
 & $64$ & $12$ & $2$ & $12$ & $2$ \\
 & $116$ & $30$ & $-2$ & $30$ & $-2$ \\
 & $168$ & $72$ & $4$ & $72$ & $4$ \\
$9$
 &  $29$ & $-4$ &  .& $4$ &  .\\
 & $81$ & $-16$ &  .& $16$ &  .\\
 & $133$ & $-44$ &  .& $44$ &  .\\
$10$
 &  $-4$ & $1$ & $1$ & $1$ & $1$ \\
 & $48$ & $6$ & $-2$ & $6$ & $-2$ \\
 & $100$ & $15$ & $1$ & $15$ & $1$ \\
 & $152$ & $34$ & $-2$ & $34$ & $-2$ \\
$11$
 &  $17$ & $2$ &  .& $-2$ &  .\\
 & $121$ & $8$ &  .& $-8$ &  .\\
 & $173$ & $8$ &  .& $-8$ &  .\\
$12$
 &  $40$ & $-4$ &  .& $-4$ &  .\\
 & $92$ & $-17$ & $-1$ & $-17$ & $-1$ \\
 & $144$ & $-40$ & $2$ & $-40$ & $2$ \\
$13$
 &  $13$ & $4$ &  .& $-4$ &  .\\
 & $65$ & $12$ &  .& $-12$ &  .\\
 & $117$ & $36$ &  .& $-36$ &  .\\
 & $169$ & $80$ &  .& $-80$ &  .\\
\bottomrule

\end{tabular}
\end{center}
\end{tiny}

\end{minipage}

\end{table}

\clearpage

\begin{table}[t]
\captionsetup{justification=centering}
\begin{minipage}[t]{0.48\textwidth}
\begin{tiny}
\begin{center}
\caption{\small McKay--Thompson series for penumbral moonshine at $\pd = (-4,17+17)$.
\label{tab:mts-cff:D0-4ell17+17}} 
\begin{tabular}{rr|@{ }r@{ }r@{ }r}\toprule
$r$ &  $D$ &1A &2A \\
 \midrule 
$0$
 &  $0$ & $1$ & $1$ \\
 & $68$ & $6$ & $6$ \\
 & $136$ & $14$ & $14$ \\
$1$
 &  $69$ & $2$ & $-2$ \\
$2$
 &  $4$ & $-1$ & $-1$ \\
 & $72$ & $-5$ & $-5$ \\
 & $140$ & $-13$ & $-13$ \\
$3$
 &  $77$ & $-2$ & $2$ \\
 & $145$ & $-2$ & $2$ \\
$4$
 &  $16$ & $2$ & $2$ \\
 & $84$ & $7$ & $7$ \\
 & $152$ & $14$ & $14$ \\
$5$
 &  $25$ & $2$ & $-2$ \\
 & $93$ & $2$ & $-2$ \\
 & $161$ & $10$ & $-10$ \\
$6$
 &  $36$ & $-2$ & $-2$ \\
 & $104$ & $-8$ & $-8$ \\
 & $172$ & $-16$ & $-16$ \\
$7$
 &  $49$ & $-2$ & $2$ \\
 & $117$ & $-6$ & $6$ \\
 & $185$ & $-14$ & $14$ \\
$8$
 &  $-4$ & $1$ & $1$ \\
 & $64$ & $4$ & $4$ \\
 & $132$ & $9$ & $9$ \\
$9$
 &  $13$ & $2$ & $-2$ \\
 & $81$ & $4$ & $-4$ \\
 & $149$ & $12$ & $-12$ \\
$10$
 &  $32$ & $-1$ & $-1$ \\
 & $100$ & $-5$ & $-5$ \\
 & $168$ & $-12$ & $-12$ \\
$11$
 &  $53$ & $-4$ & $4$ \\
 & $121$ & $-8$ & $8$ \\
$12$
 &  $8$ & $1$ & $1$ \\
 & $76$ & $3$ & $3$ \\
 & $144$ & $6$ & $6$ \\
$13$
 &  $33$ & $4$ & $-4$ \\
 & $101$ & $6$ & $-6$ \\
 & $169$ & $20$ & $-20$ \\
$14$
 &  $60$ & $-1$ & $-1$ \\
 & $128$ & $-3$ & $-3$ \\
$15$
 &  $21$ & $-2$ & $2$ \\
 & $89$ & $-6$ & $6$ \\
 & $157$ & $-16$ & $16$ \\
$16$
 &  $52$ & $1$ & $1$ \\
 & $120$ & $2$ & $2$ \\
 & $188$ & $1$ & $1$ \\
$17$
 &  $17$ & $2$ & $-2$ \\
 & $85$ & $8$ & $-8$ \\
 & $153$ & $14$ & $-14$ \\
\bottomrule

\end{tabular}
\end{center}
\end{tiny}
\end{minipage}
\hspace*{-\textwidth} \hfill
\captionsetup{justification=centering}
\begin{minipage}[t]{0.48\textwidth}
\begin{tiny}
\begin{center}
\caption{\small McKay--Thompson series for penumbral moonshine at $\pd = (-4,26+2,13,26)$.
\label{tab:mts-cff:D0-4ell26+21326}} 
\begin{tabular}{rr|@{ }r@{ }r@{ }r}\toprule
$r$ &  $D$ &1A &2A \\
 \midrule 
$0$
 &  $104$ & $4$ & $4$ \\
 & $208$ & $4$ & $4$ \\
$1$
 &  $105$ & $1$ & $-1$ \\
 & $209$ & $2$ & $-2$ \\
$2$
 &  $4$ & $-1$ & $-1$ \\
 & $108$ & $-1$ & $-1$ \\
$3$
 &  $9$ & $-1$ & $1$ \\
 & $113$ & $-2$ & $2$ \\
$4$
 &  $224$ & $2$ & $2$ \\
$5$
 &  $25$ & $1$ & $-1$ \\
 & $129$ & $3$ & $-3$ \\
$6$
 &  $36$ & $1$ & $1$ \\
 & $140$ & $1$ & $1$ \\
$7$
 &  $49$ & $-1$ & $1$ \\
 & $153$ & $-2$ & $2$ \\
$8$
 &  $64$ & $-2$ & $-2$ \\
 & $168$ & $-4$ & $-4$ \\
$9$
 &  $185$ & $-1$ & $1$ \\
$10$
 &  $-4$ & $1$ & $1$ \\
 & $100$ & $1$ & $1$ \\
$11$
 &  $17$ & $1$ & $-1$ \\
 & $121$ & $2$ & $-2$ \\
$12$
 &  $144$ & $-2$ & $-2$ \\
$13$
 &  $65$ & $-2$ & $2$ \\
 & $169$ & $-4$ & $4$ \\
$14$
 &  $92$ & $-1$ & $-1$ \\
 & $196$ & $-1$ & $-1$ \\
$15$
 &  $17$ & $1$ & $-1$ \\
 & $121$ & $2$ & $-2$ \\
$16$
 &  $48$ & $2$ & $2$ \\
 & $152$ & $2$ & $2$ \\
$17$
 &  $185$ & $-1$ & $1$ \\
$18$
 &  $12$ & $-1$ & $-1$ \\
 & $116$ & $-2$ & $-2$ \\
$19$
 &  $49$ & $-1$ & $1$ \\
 & $153$ & $-2$ & $2$ \\
$20$
 &  $88$ & $2$ & $2$ \\
 & $192$ & $2$ & $2$ \\
$21$
 &  $25$ & $1$ & $-1$ \\
 & $129$ & $3$ & $-3$ \\
$22$
 &  $172$ & $1$ & $1$ \\
$23$
 &  $9$ & $-1$ & $1$ \\
 & $113$ & $-2$ & $2$ \\
$24$
 &  $56$ & $-2$ & $-2$ \\
 & $160$ & $-2$ & $-2$ \\
$25$
 &  $105$ & $1$ & $-1$ \\
 & $209$ & $2$ & $-2$ \\
$26$
 &  $52$ & $2$ & $2$ \\
 & $156$ & $2$ & $2$ \\
\bottomrule

\end{tabular}
\end{center}
\end{tiny}

\end{minipage}
\end{table}

\begin{table}
\begin{center}
\caption{\small McKay--Thompson series for penumbral moonshine at $\pd = (-4,\ell)$, for the $\ell$ such that $G^{(\pd)}$ is trivial.
\label{tab:mts-cff:D0-4ell_Gpdtrivial}} 
\begingroup
\setlength{\tabcolsep}{15pt}
%
\end{tiny}
\end{tabular}
\endgroup

\end{center}
\end{table}

\clearpage


\setstretch{1.08}

\addcontentsline{toc}{section}{References}
\providecommand{\bysame}{\leavevmode\hbox to3em{\hrulefill}\thinspace}
\providecommand{\MR}{\relax\ifhmode\unskip\space\fi MR }
\providecommand{\MRhref}[2]{
  \href{http://www.ams.org/mathscinet-getitem?mr=#1}{#2}
}
\providecommand{\href}[2]{#2}

\end{document}